%% file: main.tex
\documentclass[11pt, reqno]{amsart}
\usepackage[varg]{txfonts}

\usepackage{anyfontsize}

\input{mypreamble}

\theoremstyle{plain}
\newtheorem{thm}{Theorem}[section]
\newtheorem{lem}[thm]{Lemma}
\newtheorem{prp}[thm]{Proposition}
\newtheorem{cor}[thm]{Corollary}

\newtheorem{alg}[thm]{Algorithm}
\newtheorem{con}[thm]{Conjecture}

\theoremstyle{definition}
\newtheorem{dfn}[thm]{Definition}

\theoremstyle{remark}

\newcommand{\eem}{\textcolor{red}}
\newcommand{\kd}{\textcolor{orange}}

\author{}
\date{\today}
\title{Intertwiners of Representations of Untwisted Quantum Affine Algebras and Yangians Revisited}

\begin{document}

\author{Keshav Dahiya and Evgeny Mukhin} 
\address{EM: Department of Mathematical Sciences,
Indiana University Indianapolis,
402 N. Blackford St., LD 270, 
Indianapolis, IN 46202, USA}
\email{emukhin@iu.edu} 

\address{KD: Department of Mathematical Sciences,
Indiana University Indianapolis,
402 N. Blackford St., LD 270, 
Indianapolis, IN 46202, USA}
\email{kkeshav@iu.edu} 

\begin{abstract}
    We discuss applications of the $q$-characters to the computation of the $R$-matrices. In particular, we describe the $R$-matrix acting in the tensor square of the first fundamental representation of E$_8$ and in a number of other cases, where the decomposition of the tensor squares with respect to non-affine quantum algebra has non-trivial multiplicities. As an illustration, we also recover $R$-matrices acting in the multiplicity free-case on the tensor squares of the first fundamental representations of all other types of untwisted quantum affine algebras. The answer is   written in terms of projectors related to the decomposition of the tensor squares with respect to non-affine quantum algebras. Then we give explicit expressions for the $R$-matrices in terms of matrix units with respect to a natural basis (except for the case of E$_8$). We give similar formulas for the Yangian $R$-matrices.

\medskip

\centerline{
  }

  \centerline{
  \textbf{\textit{AMS Classification numbers:\ }}{16T25, 17B38, 18M15 (primary), 17B37, 81R12.}}
\end{abstract}

\maketitle


\section{Introduction}
It is well-known that the solutions of Quantum Yang-Baxter equation (QYBE) or $R$-matrices, are the main source of commutative families of Hamiltonians. Quite generally, if $R_{i,j}\in\en(V_i\otimes V_j)$ are invertible operators such that $R_{12}R_{13}R_{23}=R_{23}R_{13}R_{12}\in \en (V_1\otimes V_2\otimes V_3)$ then $T_1=\Tr_{V_1}R_{13}$ and $T_2=\Tr_{V_2}R_{23}$ commute in $\en(V_3)$ since 
$$
T_1T_2=\Tr_{V_1\otimes V_2} (R_{13} R_{23}) =\Tr_{V_1\otimes V_2} (R_{12}^{-1} R_{23}R_{13} R_{12})=\Tr_{V_1\otimes V_2} ( R_{23}R_{13})=T_2T_1.
$$

The majority of the known $R$-matrices are obtained from the quantum affine algebras. Given a quantum affine algebra $U_q\tl{\fk{g}}$ corresponding to a simple Lie algebra $\mathfrak{g}$, one has an invertible element $\fk{R}\in U_q\tl{\fk{g}}\,\widetilde{\otimes} U_q\tl{\fk{g}}$ which satisfies QYBE. The element $\fk{R}$ is called the universal $R$-matrix. Evaluation of the universal $R$-matrix $\fk{R}$ on the tensor product $V_i\otimes V_j$ of any two $U_q\tl{\fk{g}}$ finite-dimensional irreducible modules results in the $R$-matrices. Moreover, we have the shift of the spectral parameter automorphism of   $U_q\tl{\fk{g}}$ which given a module $V$ produces a family of modules $V(z)$ depending on $z\in\C^\times$. The $R$-matrix computed on $V_i(z)\otimes V_j$ has a rational dependence on the parameter $z$, satisfies QYBE with a parameter, see Lemma  \ref{lemR} (\ref{tqybe1}), and it is used to construct various integrable systems. 
We call such $R$-matrices trigonometric $R$-matrices.

Taking the limit $q\to 1$ one obtains another family of $R$-matrices which we call rational $R$-matrices. Rational $R$-matrices come from Yangians and satisfy the rational version of QYBE with a parameter, \eqref{rqybe1}.

\medskip

There was a considerable effort to compute the $R$-matrices explicitly. The operator $\check{R}_{ij}(z)=PR_{ij}(z): V_i(z)\otimes V_j\to V_j\otimes V_i(z) $, where $P$ is the flip operator, is an intertwiner of  $U_q\tl{\fk{g}}$-modules. Therefore, in principle,  one can compute $\check{R}_{ij}(z)$ by solving a linear system of equations. 
However, such calculations are pretty heavy.  Another approach allows to compute the $R$-matrices in terms of projectors.

Let $V=V_i=V_j$ and let
$$
V\otimes V\cong \mathop{\oplus}_k M_k \otimes V_k, \qquad m_k=\dim M_k.
$$
be the decomposition as $U_q\fk{g}$-modules, where $M_k$ are multiplicity spaces and $V_k$ irreducible $U_q\fk{g}$-modules. Then, clearly, 
$$
\check{R}(z)=\sum_k f_k(z)P_k,
$$
where $f_k(z)\in \operatorname{End} M_k$ and $P_k$ are projectors of $V\otimes V$ to $M_k \otimes V_k$ along other summands. We say that multiplicities are trivial if $m_k=1$ for all $k$. Then $f_k(z)$ are rational scalar functions. With some knowledge of action of $E_0$ generator and Casimir operators, one  
can  compute function $f_k$ recursively using Jimbo's equation, \cite{J89}.

Other methods and formulas for $R$-matrices are described in \cite{Ma14} and \cite{DF24}.

Much less is known when the multiplicities are non-trivial, see \cite{Ju20}.

In this paper we discuss how the theory of $q$-characters can be used for the computation of $R$-matrices. The method of $q$-characters provides an alternative to the computation of the cases with trivial multiplicities and gives a way to compute some non-trivial multiplicity cases up to a few signs under the assumption that the poles of $R$-matrix are simple (see Conjecture \ref{conj:simple poles}). In addition, it improves our understanding of the final answer. We illustrate how it works for the case when $V$ is the first fundamental module of $U_q\tl{\fk{g}}$. For that case we have non-trivial multiplicities only in the case of E$_8$. 

In addition, we choose a weighted orthonormal (with respect to Shapovalov form) basis in those representations. Such a basis is  (up to a common constant) characterized by the condition that generators $E_i$ of $U_q\fk{g}$ are transposes of $F_i$ generators, cf. Lemma \ref{lin alg lemma}. Then we describe the $R$-matrices in terms of matrix units (except for the case of 
E$_8$), which seems to have been missing in literature for the exceptional types. The entries of $R$-matrices can be interpreted as Boltzmann weights in XXZ-type models. The formulas involve some lists given in the Section \ref{app}.  

The $R$-matrices for the first fundamental representations except for type E$_8$ have been computed explicitly in terms of projectors in \cite{M90}, \cite{BGZD94}, \cite{DGZ94}. The rational $R$-matrices in classical types in matrix units are given in \cite{KS82}; for G$_2$ in  \cite{O86}. Trigonometric $R$-matrices in classical types in matrix units are given in \cite{J86}, for type G$_2$ in \cite{K90}. The case of E$_8$ was considered in \cite{Ju20}.

\medskip

The $q$-characters encode eigenfunctions of Cartan generators in $U_q\tl{\fk{g}}$ and can be used to find the decomposition of $V(z)\otimes V$ in the Grothendieck ring, see \cite{FR98}, \cite{FM01}. In the trivial multiplicity case that allows to compute all poles and zeroes of $f_k(z)$. Keeping in mind that $\check{R}(1)=I$ which implies $f_k(1)=1$, this completely determines these functions provided that zeroes and poles are simple. We give an easy general argument that this is the case,  see Proposition \ref{order of zeroes prp}. This argument does not apply for types C$_r$, F$_4$ and G$_2$. Another argument which applies to all cases uses the knowledge of $\check{R}(0)$ in terms of values of Casimir element, see Lemma \ref{R0} and Theorem \ref{order of zeroes thm}. 

All $q$-characters we use in this paper can be computed using an algorithm described in \cite{FM01}. We use Theorem \ref{q char arg} to show that all the participating $q$-characters have only one dominant monomial, and therefore the algorithm is justified.

\medskip

Here we give an example of G$_2$ for $V=\tl{L}_1$, the $7$-dimensional first fundamental module. 
In this case 
$$\underbracket[0.1ex]{L_{\omega_1}}_{7}\otimes \underbracket[0.1ex]{L_{\omega_1}}_{7}\cong \underbracket[0.1ex]{L_{2\omega_1}}_{27}\oplus \underbracket[0.1ex]{L_{\omega_2}}_{14}\oplus \underbracket[0.1ex]{L_{\omega_1}}_{7}\oplus \underbracket[0.1ex]{L_{\omega_0}}_{1}\ ,\qquad \check{R}(z)=P_{2\omega_1}^q+f_1(z)\,P_{\omega_2}^q+f_2(z)\,P_{\omega_1}^q+f_3(z)\,P_{\omega_0}^q\ ,$$ 
where $L_\lambda$ are irreducible $U_q\fk{g}$-modules of highest weight $\lambda$ and  $P_{\lambda}^q$ projectors onto $L_{\lambda}$.
The module $V$ is isomorphic to $L_{\omega_1}$ as $U_q{\fk{g}}$-module, and its $q$-character reads  
$$
\chi_q(\mr{1}_{0})=\mr{1}_{0} + \ul{\mr{1}_{2}^{-1}\mr{2}_{1}} + \mr{1}_{4}\mr{1}_{6}\mr{2}_{7}^{-1} + \ul{\mr{1}_{8}^{-1}\mr{1}_{4}} + \mr{1}_{6}^{-1}\mr{1}_{8}^{-1}\mr{2}_{5} + \mr{1}_{10}\mr{2}_{11}^{-1} + \ul{\mr{1}_{12}^{-1}}\ .$$
The $q$-character $\chi_q(\mr{1}_{a})$ of $V(q^a)$ is obtained from $\chi_q(\mr{1}_0)$ by adding $a$ to all indices.
The product $V(q^a)\otimes V$ is irreducible unless $\chi_q(\mr{1}_a)\chi_q(\mr{1}_0)$ has a dominant monomial (the one which has no $\mr{1}_{a}^{-1}$ and no $\mr{2}_{a}^{-1}$) different from $\mr{1}_{a}\mr{1}_{0}$. Clearly, such a monomial occurs only at $a=\pm 2$, $a=\pm 8$, $a=\pm 12$. For $a=2,8,12$ this monomial is a product of $\mr{1}_a$ to one of the underlined monomials in $\chi_q(\mr {1}_0)$.  For $a=-2,-8,-12$ such a monomial is a product of $\mr{1}_0$ to a monomial in $\chi_q(\mr {1}_a)$ corresponding to an underlined monomial. For each such case the product $\chi_q(\mr{1}_a)\chi_q(\mr{1}_0)$ is written as a sum of two $q$-characters. For example, 
\eq{\label{qchar arg for G2}
\chi_q(\mr{1}_{-8})\chi_q(\mr{1}_{0})=\chi_q(\mr{1}_{-8}\mr{1}_{0})+\chi_q(\mr{1}_{-4}).
}
We claim that $\chi_q(\mr{1}_{-8}\mr{1}_{0})$ has only one dominant monomial and therefore can be computed by the algorithm of \cite{FM01}. In fact, the product $\chi_q(\mr{1}_{-8})\chi_q(\mr{1}_{0})$ has two dominant monomials: $\mr{1}_{-8}\mr{1}_{0}$ and $\mr{1}_{-4}$. We use Theorem \ref{q char arg} to show that $\mr{1}_{-4}$ is not in $\chi_q(\mr{1}_{-8}\mr{1}_{0})$.
Then using the algorithm, we see that $\chi_q(\mr{1}_{-8}\mr{1}_{0})$ 
has 42 terms, and corresponds to the direct sum $L_{2\omega_1}\oplus L_{\omega_2}\oplus L_{\omega_0}$.
The summand $\chi_q(\mr{1}_{-4})$ has 7 terms and it corresponds to the remaining summand $L_{\omega_1}$. A $U_q\tl{\fk{g}}$-submodule which does not contain the product of highest weight vectors occurs only for $a<0$, see \cite{C00}, \cite{Ka02}, and it becomes the kernel of $\check{R}(z)$.
Thus $f_2$ has a zero when $z=q^{-8}$ and a pole when $z=q^{8}$. 

Similarly, we obtain that $z=q^{-2}$ is a zero of $f_1$ and $f_3$ while $z=q^{-12}$ is a zero of $f_3$. That way we find all zeroes and all poles of $f_i(z)$. Since zeroes and poles are simple, see Theorem \ref{order of zeroes thm}, we determine $f_i$ up to a constant which is obtained from $f_i(1)=1$. So
$$
f_1(z)= -q^{-2}\frac{1-q^{2}z}{1-q^{-2}z},\qquad f_2(z)=-q^{-8}\frac{1-q^{8}z}{1-q^{-8}z} , \qquad f_3(z)= q^{-14}\frac{(1-q^{2}z)(1-q^{12}z)}{(1-q^{-2}z)(1-q^{-12}z)}.
$$

\medskip

Finally, let us discuss the cases with non-trivial multiplicities.
After a choice of a basis of singular vectors, $f_k(z)$ become $m_k\times m_k$ matrices whose entries are rational functions. The $q$-characters tell us for which $z$ the matrices $f_k(z)$ are degenerate or have a pole and describe the ranks of these matrices. We have additional equations $f_k(1)=\id$, $f_k(z) f_k(z^{-1})=\operatorname{Id}$, and we also know $f_k(0)$ and $f_k(\infty)$, see \cite{R88} or equation (3.16) in \cite{DGZ94}. We also know how $f_k(z)$ commute with the flip operator $P$, see Lemma \ref{R and flip commute}. Finally, the $R$-matrix is self-adjoint, see part (\ref{R is self adjoint}) of Lemma \ref{lemR}.
In the cases we consider, this information determines the matrices $f_k(z)$ up to a sign, provided that the poles of the $R$-matrix are simple, see the proof of Theorem \ref{thm:R E8}. It is easy to guess the remaining signs but for the proof, we resort to checking (partly) the commutativity with $E_0$.

Alternatively, some examples of $R$-matrices with non-trivial multiplicities can be computed using the well-known fusion procedure for $R$-matrices. We give an example of the evaluation adjoint module $V=L_{\omega_1+\omega_2}$ for $\fk{sl}_3$, where we have a $2\times 2$ matrix, see Section \ref{mul cases A2}, and of the second fundamental module for G$_2$,
 $V=L_{\omega_2}\oplus L_{\omega_0}$, where we have a $2\times 2$ and a $3\times 3$ matrices, see Section \ref{mul cases G2}.

The new, most challenging and interesting case is E$_8$ where the $R$-matrix is of size $62001\times 62001$. In terms of projectors it has a $2\times 2$ matrix and a $3\times 3$ matrix. There is a one parameter freedom in these matrices due to the choice of rescaling of the basis. After using our techniques, we have only a sign in each matrix to fix. For that we use a computer calculation. This is the only result in the paper which we could not do by hand. The answer and details are given in Section \ref{E8 sec}. For E$_8$ we do not give an answer in terms of matrix units. However, for the final computer assisted calculation we are forced to choose a basis in the $L_{\omega_1}$ for $E_8$ which presents some interest on its own. The essential information about the basis is given in picture in Section \ref{E8 app}.

In all examples we computed, matrices $f_k(z)$ have some remarkable similarities, we plan to address this issue in the future publications.

\medskip

Rational $R$-matrices are easily obtained by the appropriate limit $q\to1$ of trigonometric ones. We give the answers in all cases.

\medskip 

The zoo of all possible $R$-matrices coming from quantum affine algebras is too large to give explicit formulas for all cases. However, on demand, one can make such computations using fusion process and the matrices given in this paper. 
This paper paves a way to compute examples related to the twisted and supersymmetric cases. The twisted cases we are discussing in \cite{DM25}. The supersymmetric cases we plan to treat in subsequent publications.

\medskip

The structure of the paper is as follows. In Section \ref{prel sec} we recall the quantum affine algebras, $R$-matrices, representations, and the $q$-characters.  In Section \ref{meth sec} we describe the details of our approach to computation of explicit formulas of $R$-matrices. In Section \ref{cla sec} we present the $R$-matrices for the first fundamental modules in the classical types. In Section \ref{ex sec} we give $R$-matrices  for the first fundamental modules in exceptional types. In particular, Section \ref{E8 sec} contains the E$_8$ matrix. In Section \ref{mul cases} we write examples of $R$-matrices of types A$_2$ and G$_2$ which contain non-trivial multiplicities. In Section \ref{app} we collect the various data about the choices of bases and expressions for projectors in terms of these bases.

\section{Preliminaries}\label{prel sec}

In this section, we recall well known facts about quantum affine algebras and their representations.\\ 
See \cite{CP1}, \cite{FM01} for details.

\subsection{Quantum affine algebras}

We use the following general notations.

\begin{enumerate}
\item Let $\mr{I}=\{1,\dots,r\}$ and $\tl{\mr{I}}=\{0,1,\dots,r\}$.

\item Let $\fk{g}$ be a simple finite-dimensional Lie algebra of rank $r$ with Cartan matrix $C=(C_{ij})_{i,j\in\mr{I}}$ and $D=\text{diag}(d_1,\dots,d_r)$ be such that $B=DC$ is symmetric and $d_i\in\Z_{>0}$ are minimal possible. The matrix $B$ is called the symmetrized Cartan matrix.

\item Let $\A_1,\dots,\A_r$ be simple roots, $\omega_1,\dots,\omega_r$ fundamental weights, ${\cl{P}}=\oplus_{i\in\mr{I}}\Z\omega_i$ the corresponding weight lattice and $\cl{P}_+=\oplus_{i\in\mr{I}}\Z_{\ge0}\omega_i$ the cone of dominant weights.
We set $\omega_0=0\in\cl{P}_+$.

\item Let $\tl{\fk{g}}=\fk{g}\otimes \C[t,t^{-1}]$ be the loop Lie algebra associated to $\fk{g}$. Let  $\tl{C}=(\tl{C}_{ij})_{i,j\in\tl{\mr{I}}}$ and $\tl{B}=(\tl{B}_{ij})_{i,j\in\tl{\mr{I}}}$ be the corresponding affine Cartan and symmetrized Cartan matrices.

\item Let $a=(a_0,\dots,a_r)$ be the sequence of positive integers such that $\tl Ca^t=0$ and such that $a_0,\dots, a_r$ are relatively prime.

\item Let $q\in\C^\times$ be such that $q$ is not a root of unity. We fix a square root $q^{1/2}$.
Let $q_j=q^{d_j}$, $j\in \tl{\mr{I}}$. For $k\in\frac{1}{2}\Z$ and $n\in\Z$, set
$$[n]_k=\frac{q^{kn}-q^{-kn}}{q^k-q^{-k}}\,\,,\quad [n]_k^{\mr{i}}=\frac{q^{kn}+(-1)^{n-1}q^{-kn}}{q^k+q^{-k}}.$$ 
Both $[n]_k$ and $[n]_k^{\mr{i}}$ are Laurent polynomials in $q^{1/2}$. We write $[n]_1$ as $[n]$ and $[n]_1^{\mr{i}}$ as $[n]^{\mr{i}}$.\\
Note that $\lim_{q\to 1}\,[n]_k=n$, $\lim_{q\to 1}[n]_k^{\mr{i}}=1$ if $n$ is odd, and $\lim_{q\to 1}[n]_k^{\mr{i}}=0$ if $n$ is even.

\item All representations are assumed to be finite-dimensional. We consider quantum affine algebras of level zero only. All representations of quantum affine algebras are assumed to be of type 1.

\item  For $n\in\Z_{>0}$ let $\kappa_n(q)=q^{-\phi(n)}\Phi_n(q^2)$ be the symmetric form of the $n$-th cyclotomic polynomial $\Phi_n(q)$, where $\phi(n)$ is the Euler function. We have $\kappa(q^{-1})=\kappa(q)$. For example, for $n=2^{i} \cdot 3^{j}$, $i,j\in\Z_{\geq 0}$, we have $\kappa_{6n}(q)=[3]_{n/2}^{\mr{i}}=q^{n}-1+q^{-n}$. 
\end{enumerate}

\begin{dfn}[Drinfeld-Jimbo realization] 
The quantum affine algebra $U_q\tl{\fk{g}}$ of level zero associated to $\fk{g}$ is an associative algebra over $\C$ with generators $E_i$, $F_i$, $K_i^{\pm1}$, $i\in\tl{\mr{I}}$, and relations: 
$$K_iK_i^{-1}=K_i^{-1}K_i=1\,\,,\quad K_iK_j=K_jK_i\,\,,\quad K_0^{a_0}K_1^{a_1}\cdots K_r^{a_r}=1\,\,,$$
$$K_iE_jK_i^{-1}=q^{\tl{B}_{ij}}E_j\,\,,\quad K_iF_jK_i^{-1}=q^{-\tl{B}_{ij}}F_j\,\,,\quad [E_i,F_j]=\D_{ij}\frac{K_i-K_i^{-1}}{q_i-q_i^{-1}}\,\,,$$
$$\sum_{m=0}^{1-\tl{C}_{ij}}(-1)^m\binom{1-\tl{C}_{ij}}{m}_{q_i}E_i^mE_jE_i^{1-\tl{C}_{ij}-m}=0\,\,,\quad\sum_{m=0}^{1-\tl{C}_{ij}}(-1)^m\binom{1-\tl{C}_{ij}}{m}_{q_i}F_i^mF_jF_i^{1-\tl{C}_{ij}-m}=0\ ,\quad i\ne j\,\,.$$
\end{dfn}

The algebra $U_q\tl{\fk{g}}$ has a Hopf algebra structure with comultiplication $\Delta$ given on the generators by 
\eq{\label{coproduct}
\Delta(K_i)=K_i\otimes K_i\,\,,\quad\Delta(E_i)=E_i\otimes K_i^{1/2}+K_i^{-1/2}\otimes E_i\,\,,\quad\Delta(F_i)=F_i\otimes K_i^{1/2}+K_i^{-1/2}\otimes F_i\,\,,\,\,\,\,i\in\tl{\mr{I}}.
}

The Hopf subalgebra of $U_q\tl{\fk{g}}$ generated by $K_i^{\pm1}$, $E_i$, $F_i$, $i\in\mr{I}$, is isomorphic to the quantum algebra $U_q\fk{g}$ associated to $\fk{g}$.

In what follows we also use the notation $U_q($A$_r)$,   $U_q($ E$_7)$, $U_q($A$_r^{(1)})$, etc., for quantum algebras $U_q\fk{g}$ of type A$_r$, E$_7$, quantum affine algebra $U_q\tl{\fk{g}}$ of type  A$_r^{(1)}$, etc.

\begin{thm}[Drinfeld's new realization]
The algebra $U_q\tl{\fk{g}}$ is isomorphic to the algebra with generators\\ $X_{i,n}^{\pm}$ $(i\in\mr{I}, n\in\Z)$, $K_i^{\pm1}$ $(i\in{\mr{I}})$, $H_{i,m}$ $(i\in\mr{I},m\in\Z\setminus\{0\})$, and relations: 
$$K_iK_i^{-1}=K_i^{-1}K_i=1\,\,,\quad [\Phi_i^\pm(z),\Phi_j^\pm(w)]=[\Phi_i^\pm(z),\Phi_j^\mp(w)]=0\,\,,$$
$$(q^{\pm B_{ij}}z-w)\,\Phi_i^\epsilon(z)X_j^\pm(w)=(z-q^{\pm B_{ij}}w)\,X_j^\pm(w)\Phi_i^\epsilon(z)\,\,\text{ for }\epsilon=\pm\,,$$
$$(q^{\pm B_{ij}}z-w)\,X_i^{\pm}(z)X_j^\pm(w)=(z-q^{\pm B_{ij}}w)\,X_j^\pm(w)X_i^\pm(z)\,\,,$$
$$[X_{i}^+(z),X_{j}^-(w)]=\D_{ij}\,\D\bigg(\frac{z}{w}\bigg)\,\frac{\Phi_{i}^+(z)-\Phi_{i}^-(z)}{q_i-q_i^{-1}}\,\,,\text{ where }\D(t)=\sum_{i\in\Z}t^i\,\in\C[[t,t^{-1}]]\,\,,$$ 
$$\sum_{\pi\in S_{1-C_{ij}}}\sum_{k=0}^{1-C_{ij}}(-1)^k\binom{1-C_{ij}}{k}_{q_i}X_{i,n_{\pi(1)}}^{\pm}\cdots X_{i,n_{\pi(k)}}^{\pm}X_{j,m}^{\pm}X_{i,n_{\pi(k+1)}}^{\pm}\cdots X_{i,n_{\pi(1-C_{ij})}}^{\pm}=0$$
for all sequences of integers $m,n_1,\dots, n_{1-C_{ij}}$ and $i\ne j$, where $S_{1-C_{ij}}$ is the symmetric group on $1-C_{ij}$ letters.

Here: 
$$\Phi_i^{\pm}(z)=K_i^{\pm1}\exp\bigg(\pm(q_i-q_i^{-1})\sum_{m=1}^\infty H_{i,\pm m}\,z^{\pm m}\bigg)\in U_q\tl{\fk{g}}[[z^{\pm1}]]\,,$$
$$X_i^\pm(z)=\sum_{n\in\Z}X_{i,n}^\pm z^{n}\in U_q\tl{\fk{g}}[[z,z^{-1}]]\,\,.$$

\qed
\end{thm}

\begin{prp}[The shift of spectral parameter automorphism $\tau_a$]
For any $a\in\C^\times$, there is a Hopf algebra automorphism $\tau_a$ of $U_q\tl{\fk{g}}$ defined by: 
$$\tau_a(X_i^\pm(z))=X_{i}^\pm(az)\,\,,\quad\tau_a(\Phi_{i}^\pm(z))=\Phi_{i}^\pm(az)\,\,,\quad i\in{\mr{I}}\,.$$ 
\qed
\end{prp}

Given a $U_q\tl{\fk{g}}$-module $V$ and $a\in\C^\times$, we denote by $V(a)$ the pull-back of $V$ by $\tau_a$.

\begin{dfn}[Weight space]
Given a $U_q\fk{g}$-module $V$ and $\lambda=\sum_{i\in \mr{I}}\lambda_i\omega_i\in\cl{P}$, define the subspace of weight $\lambda$ to be
$$V_\lambda=\{v\in V:K_iv=q_i^{\lambda_i}v, \ i\in\mr{I}\}.$$
If $V_\lambda\ne0$, $\lambda$ is called a weight of $V$. A nonzero vector $v\in V_\lambda$ is called a vector of weight $\lambda$.
\end{dfn}
For every representation $V$ of $U_q\fk{g}$ we have $V=\oplus_{\lambda}V_{\lambda}$.

\begin{dfn}[$\ell$-weight]
Given a $U_q\tl{\fk{g}}$-module $V$ and $\G=(\G_i^\pm(z))_{ i\in \mr{I}}$, $\G_i^\pm(z)\in\C[[z^{\pm1}]]$, a sequence of formal power series in $z^{\pm1}$, define the subspace of generalized eigenvectors of $\ell$-weight $\G$ to be
$$V[\G]=\{v\in V: (\Phi_i^\pm(z)-\G_i^\pm(z))^{\dim(V)}\, v=0, \ i\in\mr{I}\}.$$
If $V[\G]\ne0$, $\G$ is called an $\ell$-weight of $V$.

For every representation $V$ of $U_q\tl{\fk{g}}$ we have $V=\oplus_{\G}V[\G]$ and for every $\lambda\in\cl{P}$, $V_\lambda=\oplus_{\G} (V_\lambda\cap V[{\G}])$.

A non-zero vector $v$ is a vector of $\ell$-weight $\G$ if 
$$(\Phi_i^\pm(z)-\G_i^\pm(z))\ v=0, \ i\in\mr{I}.$$

\end{dfn}

\begin{dfn}[Highest $\ell$-weight representations]
A nonzero vector $v$ of $\ell$-weight $\G$ in some $U_q\tl{\fk{g}}$-module $V$ is called an $\ell$-singular vector if $$X_i^+(z)\,v=0\,,\ i\in{\mr{I}}\,.$$ 
A representation $V$ of $U_q\tl{\fk{g}}$ is called a highest $\ell$-weight representation if $V=U_q\tl{\fk{g}}\,v$ for some $\ell$-singular vector $v$. In such case $v$ is called the highest $\ell$-weight vector.
\end{dfn}

Let $\cl{U}$ be the set of all $\mr{I}$-tuples $ p=(p_i)_{i\in\mr{I}}$ of polynomials $p_i\in\C[z]$, with constant term $1$.

\begin{thm}
\leavevmode
\begin{enumerate}
\item Every irreducible representation of $U_q\tl{\fk{g}}$ is a highest $\ell$-weight representation. 
\item Let $V$ be an irreducible representation of $U_q\tl{\fk{g}}$ of highest $\ell$-weight $\big(\G_i^\pm(z)\big)_{i\in\mr{I}}$. Then there exists $p=(p_i)_{i\in\mr{I}}\in\cl{U}$ such that $$\G_i^\pm(z)=q_i^{\deg(p_i)}\frac{p_i(zq_i^{-1})}{p_i(zq_i)}\in\C[[z^{\pm1}]]\,.$$
\item Assigning to $V$ the $\mr{I}$-tuple $p\in\cl{U}$ defines a bijection between $\cl{U}$ and the set of isomorphism classes of irreducible representations of $U_q\tl{\fk{g}}$.
\end{enumerate}
\qed
\end{thm}

The polynomials $p_i(z)$ are called \ti{Drinfeld polynomials}. We denote the irreducible 
$U_q\tl{\fk{g}}$-module with  Drinfeld polynomials $p$ by $\tl{L}_p$. 



\begin{dfn}[Fundamental representations]
For each $i\in\mr{I}$, let $\tl{L}_{i}=\tl{L}_{p^{(i)}}$ be the irreducible $U_q\tl{\fk{g}}$-module corresponding to the Drinfeld polynomials given by: $$p^{(i)}=(1-\D_{ij}z)_{j\in\mr{I}}.$$
We call $\tl{L}_{i}(a)$ the $i^\text{th}$ fundamental representation of $U_q\tl{\fk{g}}$.
\end{dfn}

The category $\fk{Rep}\big(U_q\tl{\fk{g}}\big)$ of representations of $U_q\tl{\fk{g}}$ is an abelian monoidal category.  Denote by $\rep U_q\tl{\fk{g}}$ the
 Grothendieck ring of $\fk{Rep}\big(U_q\tl{\fk{g}}\big)$.
 
The category $\fk{Rep}\big(U_q\fk{g}\big)$ of representations of $U_q\fk{g}$ is an abelian monoidal semi-simple category. We denote the corresponding Grothendieck ring by $\rep U_q\fk{g}$. Irreducible modules in $\fk{Rep}\big(U_q\fk{g}\big)$ are parameterized by integral dominant weights. For $\lambda\in \cl{P}_+$, 
denote the corresponding irreducible $U_q\fk{g}$-module by $L_\lambda$.


The module $L_\la$ has a unique (up to a scalar) symmetric bilinear form $(\ ,\ )$, called Shapovalov form, such that $E_i^*=F_i$, $i\in \mr{I}$. The Shapovalov form is non-degenerate. 

We use Shapovalov form on factors to define the form on $L_\la\otimes L_\mu$. We call this form tensor Shapovalov form. The tensor Shapovalov form is non-degenerate, and
because of our symmetric choice of coproduct \eqref{coproduct}, we have 
\eq{\label{ad coproduct}
\big(\Delta(E_i)\big)^*=\Delta(F_i) \quad \text{and}\quad \big(\Delta(F_i)\big)^*=\Delta(E_i)\ ,\quad i\in \mr{I}\ .
}

In what follows we will choose a weighted basis of $L_{\omega_1}$ such that $E_i^T=F_i$, $i\in \mr{I}$, where $T$ stands for transposition. This basis is automatically orthonormal with respect to the Shapovalov form (for an appropriate choice of normalization of the latter) due to the following simple lemma of linear algebra.

\begin{lem}\label{lin alg lemma} 
Let $V$ be a vector space with a non-zero symmetric bilinear form $(\ ,\ )$. Let $\{v_1,\dots,v_d\}$ be a basis of $V$. Let $F_1,\dots, F_r$ be linear operators on $V$ which are strictly lower triangular in the basis of $v_i$. Assume that $V$ is cyclic with respect to the algebra generated by $F_1,\dots, F_r$ with cyclic vector $v_1$. Then if $F_i^*=F_i^T$ for all $i$, then $(v_i,v_j)=c\delta_{ij}$ for some nonzero constant $c$. \qed
\end{lem}

\subsection{\texorpdfstring{$q$}--characters}
For each $i\in\mr{I}$, $a\in\C^\times$, let $Y_{i,a}$ be an $r$-tuple of rational functions given by: $$Y_{i,a}(z)=\bigg(\underbrace{1,\dots,1}_{i-1},\,q_i\frac{1-q_i^{-1}za}{1-q_iza}\,,\underbrace{1,\dots,1}_{r-i}\bigg).$$
The $r$-tuple $Y_{i,a}$ is the highest $\ell$-weight of $\tl{L}_{i}(a)$.

Let $\cl{Y}$ be the abelian group of $r$-tuples of rational functions generated by $\{Y_{i,a}^{\pm1}\}_{i\in\mr{I},\,a\in\C^\times}$ with component-wise multiplication. It is well-known that the $\ell$-weights of representations of $U_q\tl{\fk{g}}$ belong to $\cl{Y}$.

\begin{dfn}[$q$-character]
The $q$-character of a $U_q\tl{\fk{g}}$-module $V$ is the formal sum $$\chi_q(V)=\sum_{\G\in\cl{Y}}\dim(V[\G])\,\G\,\,\in\Z[\cl{Y}].$$
\end{dfn}

\begin{thm}
The $q$-character map $\chi_q:\rep U_q\tl{\fk{g}}\to\Z[\cl{Y}],$ sending $V\mapsto \chi_q(V)$, is an injective ring homomorphism.
\qed
\end{thm}

\begin{dfn}[Dominant $\ell$-weights]
For an $i\in\mr{I}$, an $\ell$-weight is called $i$-dominant if the $\ell$-weight is a monomial in variables $\{Y_{i,a}, Y_{j,a}^{\pm}\}_{j\in\mr{I}, j\neq i\,,\,a\in\C^\times}$. An $\ell$-weight is called dominant if it is $i$-dominant for all $i\in \mr{I}$.

The set of dominant $\ell$-weights will be denoted by $\cl{Y}_+$.
\end{dfn}

A $U_q\tl{\fk{g}}$-module $V$ is called special if $\chi_q(V)$ contains a unique dominant monomial.

The semi-group $\cl{Y}_+$ is naturally identified with $\cl{U}$. For $m_+\in \cl{Y}_+$, let $p(m_+)\in \cl{U}$ be the corresponding set of Drinfeld polynomials.

\begin{dfn}[Simple $\ell$-roots]
For each $i\in\mr{I}$ and $a\in\C^\times$, let $A_{i,a}\in\cl{Y}$ be given by $$A_{i,a}(z)=\bigg(q^{B_{ij}}\frac{1-q^{-B_{ij}}za}{1-q^{B_{ij}}za}\bigg)_{j\in\mr{I}}.$$
We call $A_{i,a}$ a simple $\ell$-root of color $i$.
\end{dfn}

Denote $Y_{1,q^k}$ by $\mr{1}_k$ , $Y_{2,q^k}$ by $\mr{2}_k$ and so on.  
For $m_+\in\cl{Y}_+$, denote $\tl{L}_{p(m_+)}$ by $\tl L_{m_+}$ and $\chi_q(\tl{L}_{p(m_+)})$ by $\chi_q(m_+)$.

If $V$ is a special $U_q\tl{\fk{g}}$-module then the $q$-character can be computed by a recursive algorithm, see \cite{FM01}.

We prepare a theorem which allows us to eliminate some monomials from $\chi_q(V)$ and to show that $V$ is special.

\begin{thm}\label{q char arg} 
Let $V$ be an irreducible  $U_q\tl{\fk{g}}$-module. Let $m$ be an $i$-dominant monomial in $\chi_q(V)$ of multiplicity one  for some $i\in\mr{I}$. Let $b\in\C^\times$ and $m_-=m A^{-1}_{i,b}$. Suppose 
\begin{enumerate}
    \item \label{q char arg 0} The power of $Y_{i,bq_i^{-1}}$ in $m$ is not greater than the power of $Y_{i,b q_i}$ in $m$.

    \item\label{q char arg 1} $m A_{i,c}\not\in\chi_q(V)$ for all $c\in\C^\times$. 


     \item\label{q char arg 2} $m_-A_{j,c}\not\in\chi_q(V)$ for all $j\in \mr{I}$, $c\in\C^\times$ unless $(j,c)=(i,b)$.

     \item\label{q char arg 3} The multiplicity of $m_-$ in $\chi_q(V)$ is not greater than one.
\end{enumerate}
Then  multiplicity of $m_-$ in $\chi_q(V)$ is zero,  $m_-\notin \chi_q(V)$.
\end{thm}
\begin{proof} Assume $m_-\in\chi_q(V)$. Then by \eqref{q char arg 3}, the multiplicity of $m_-$ is exactly one.

Let $v,v_{-}\in V$ be non-zero vectors of $\ell$-weight $m,m_-$, respectively. 

Then by Lemma 3.1 in \cite{Y14},  the matrix coefficients of the action of $X_i^-(w)$ are linear combinations of derivatives of delta functions. These coefficients are non-zero only if the $\ell$-weights differ by $A_{i,c}^{-1}$ for some nonzero $c\in \C$ (cf. also Proposition 3.8 in \cite{MY14}), in which case the support of delta functions is at $c^{-1}$.
Thus the  action of $X_i^-(w)$  on $v$ takes the form 
$$
X^-_i(w)\,v=c_-\delta(b w)\,v_-+ \sum_s c_s(\delta (b_s w))\, v_s,
$$
where the sum is over some finite set of values of $s$, $ c_-, b_s\in \C$ with $b_s\neq 0$, $c_s=\sum_j c_{s,j}\partial_w^j \in \C[\partial_w]$,  and  $v_s$ are generalized $\ell$-weight vectors  of weight  $mA^{-1}_{i,b_s}$. By \eqref{q char arg 3},  we have $b_s\neq b$ for all $s$.

If $c_{-}=0$  then $v_-$ is not in the sum of images of $X^-_{j}(z)$. Indeed, by (\ref{q char arg 2}), if $u$ is a generalized $\ell$-vector,  then the vector  $X^-_{j}(w)u$ does not have an $m_-$ $\ell$-weight component unless maybe for $j=i$ and $u=c v$ for some $c\in\C$. But the latter is also zero if $c_-=0$.
Since $V$ is irreducible, all $\ell$ weighted vectors in $V$ except the highest $\ell$-weight vector, are obtained by the action of $X^-_{j}(w)$, therefore such a vector $v_-$ does not exist, and the theorem follows.

Let $c_{-}\neq 0$.
Using Lemma 3.1 in \cite{Y14} once again, we obtain
$$
X^+_i(z)\,v_-=\tilde c_-\delta(bz)\,v, \qquad X^+_i(z)\,v_s=\tilde c_s(\delta(b_sz))\,v+\dots,
$$
where $\tilde c_-\in\C$, $\tilde c_s=\sum_j \tilde c_{s,j}\partial_z^j \in \C[\partial_z]$, and the dots denote sum of vectors of $\ell$-weights different from $m$. In the first equation such terms are absent by the assumption (\ref{q char arg 2}). 
By (\ref{q char arg 1}), $X_i^+(z)\,v=0$.   Then we compute 
$$
[X^+_i(z),X^-_i(w)]v= X^+_i(z) X^-_i(w)v= \Big(
c_-\tilde c_-\delta(bz)\delta(bw)+\sum_s c_s(\delta(b_sw) \tilde c_s(\delta(b_sz) \Big)\ v + \dots\ .
$$
On the other hand from the relation in the algebra and $m$ we have
$$
[X^+_i(z),X^-_i(w)]\,v=\delta(z/w)\frac{\Phi_i^+(z)-\Phi_i^-(z)}{q_i-q_i^{-1}}\,v.
$$ 
The vector $v$ is of $\ell$-weight $m$, therefore it is an eigenvector of $\Phi_i^\pm(z)$ with eigenfunction which is a rational function. By \eqref{q char arg 0}, that eigenfunction has no pole at $z=b^{-1}$. 

It follows that $\tilde c_-=0$ (moreover, all terms with $b_s$ which are not poles of the eigenfunction should cancel out).
Then $X_i^+(z)\,v_-=0$. By \eqref{q char arg 2},  we also have $X_j^+(z)\,v_-=0$. Thus, $v_-$ is a highest $\ell$-weight vector. Since $V$ is irreducible,  such a vector $v_-$ does not exist, and the theorem follows.
\end{proof}

We apply Theorem \ref{q char arg} to extract $\chi_q(V)$ from a known tensor product. In all our cases this tensor product has two dominant monomials and we use Theorem \ref{q char arg} to show that one of them is not in $\chi_q(V)$.  That allows us to easily identify $\chi_q(V)$.
Note that the conditions in Theorem \ref{q char arg} are completely combinatorial and therefore can be easily checked.

\subsection{\texorpdfstring{$R$}c-matrices}

There is a quasitriangular structure on the Hopf algebra $U_q\tl{\fk{g}}$.

\begin{prp}\label{R prop}
The Hopf algebra $U_q\tl{\fk{g}}$ is almost cocommutative and quasitriangular, that is, there exists an invertible element $\fk{R}\in U_q\tl{\fk{g}}\,\hat{\otimes}\,U_q\tl{\fk{g}}$ of a completion of the tensor product, such that $$\Delta^{\op}(a)=\fk{R}\,\Delta(a)\,\fk{R}^{-1}\,,\,\,\,a\in U_q\tl{\fk{g}}\,,$$
where $\Delta^{\op}(a)=P\circ\Delta(a)$, $P$ is the flip operator, and
$$(\Delta\otimes\id)(\fk{R})=\fk{R}_{13}\fk{R}_{23}\,,\quad(\id\otimes\Delta)(\fk{R})=\fk{R}_{13}\fk{R}_{12}\,,\quad\fk{R}_{12}\fk{R}_{13}\fk{R}_{23}=\fk{R}_{23}\fk{R}_{13}\fk{R}_{12}\,.$$
\qed
\end{prp}

The element $\fk{R}$ is called the universal $R$-matrix of $U_q\tl{\fk{g}}$.

The universal $R$-matrix has weight zero and homogeneous degree zero:
$$
(K_i\otimes K_i) \fk{R}=\fk{R} (K_i\otimes K_i), \qquad (\tau_z\otimes \tau_z) \fk{R}=\fk{R} (\tau_z\otimes \tau_z),\qquad i\in \tl{\mr{I}},\ z\in\C^\times.
$$

\begin{dfn}[Trigonometric $R$-matrix]
Let $V$ and $W$ be two representations of $U_q\tl{\fk{g}}$ and $\pi_V$, $\pi_W$ be the respective representations maps. The map $$\tilde R^{V,W}(z)=(\pi_{V(z)}\otimes\pi_{W})(\fk{R}):V(z)\otimes W\to V(z)\otimes W$$ is called the $R$-matrix of $U_q\tl{\fk{g}}$ evaluated in $V(z)\otimes W$. 
\end{dfn}

\begin{dfn}[Normalized $R$-Matrix]
Let $V$, $W$ be representations of $U_q\tl{\fk{g}}$ with highest $\ell$-weight vectors $v$ and $w$ respectively. Denote by ${R}^{V,W}(z)\in\en(V\otimes W)$ the normalized $R$-matrix satisfying: 
$${R}^{V,W}(z)=f_{V,W}^{-1}(z)\,\tilde R^{V,W}(z)\,,$$ where $f_{V,W}(z)$ is the scalar function defined by $\tilde {R}^{V,W}(z)(v\otimes w)=f_{V,W}(z)\,v\otimes w.$ 
\end{dfn}

The map 
\eq{\label{R check}
\check{R}^{V,W}(z)=P\circ R^{V,W}(z):V(z)\otimes W\to W\otimes V(z)
}
(if it exists) is an intertwiner (or a homomorphism) of $U_q\tl{\fk{g}}$-modules. If $V$, $W$ are irreducible, then the module $V(z)\otimes W$ is irreducible for all but finitely many $z\in\C^\times$. If for some $z$, the module $V(z)\otimes W$ is irreducible, then $W\otimes V(z)$ is also irreducible and the intertwiner is unique up to a constant. 

\begin{lem}
    Let $V_i$, $i=1,2,3$, be representations of $U_q\tl{\fk{g}}$. 
\begin{enumerate}
    \item \label{tqybe1}$\displaystyle R^{V_1,V_2}_{12}(z)\,R^{V_1,V_3}_{13}(zw)\,R^{V_2,V_3}_{23}(w)=R^{V_2,V_3}_{23}(w)\,R^{V_1,V_3}_{13}(zw)\,R^{V_1,V_2}_{12}(z)\,.$
    \item \label{tqybe} $\displaystyle\check{R}^{V_1,V_2}_{23}(z)\,\check{R}^{V_1,V_3}_{12}(zw)\,\check{R}^{V_2,V_3}_{23}(w)=\check{R}^{V_2,V_3}_{12}(w)\,\check{R}^{V_1,V_3}_{23}(zw)\,\check{R}^{V_1,V_2}_{12}(z)\,.$
\end{enumerate}
\qed
\end{lem}
The above two properties are called trigonometric QYBE.

The $R$-matrix $\check{R}^{V,W}(z)$ depends on the choice of the coproduct. In this paper we use coproduct $\Delta$ given by \eqref{coproduct}. Let $\mathfrak{R}_{\text{op}}$ be the universal $R$ matrix  corresponding to coproduct $\Delta^{\text{op}}$ and $R^{V,W}_{\text{op}}(z)$ be that $R$-matrix evaluated in $V(z)\otimes W$. Then $\mathfrak{R}_{\text{op}}=P\,\mathfrak{R}P$ and
\eq{\label{Rop}
\check{R}^{V,W}_{\text{op}}(z)=P(\pi_V\otimes \pi_W)\big((\tau_{z}\otimes 1)(\mathfrak{R}_{\text{op}})\big) = P\check{R}^{W,V}(z^{-1})P.
}
We collect a few properties of the $R$-matrices.
\begin{lem} \label{lemR}
Let $V_i$, $i=1,2$, be representations of $U_q\tl{\fk{g}}$. 
\begin{enumerate}
    \item  The normalized intertwiner $\check{R}^{V_1,V_2}(z)$ is a rational function of $z$. \label{try}
    \item If $V_1=\tl {L}_i(a)$ is fundamental, then $\check{R}^{V_1,V_1}(1)=\id$.
    \item $\check{R}^{V_1,V_2}(z;q)=P\check{R}^{V_2,V_1}(z^{-1};q^{-1})P$.
    \item\label{inversion relation} $\check{R}^{V_1,V_2}(z)\,\check{R}^{V_2,V_1}(z^{-1})=\id$.
    \item\label{R is self adjoint} $\check{R}^{V_1,V_2}(z)$ is self-adjoint with respect to the tensor Shapovalov form.
\end{enumerate}
\end{lem}
\begin{proof}
   The intertwiner $\check{R}^{V_1,V_2}(z)$ is uniquely determined by commuting with $E_i, F_i$, $i\in\tilde I$. The action of these operators is given by Laurent polynomials in $z$. The first property follows.
   
   The second property follows from the well-known fact that the module $\tl{L}_i(a)\otimes \tl{L}_i(a)$ is irreducible.
 
 We provide a proof of the third property. 
Let $\nu: (U_q\tl{\fk{g}},\Delta)\to (U_{q^{-1}}\tl{\fk{g}},\Delta^{\text{op}})$ be an isomorphism of Hopf algebras sending $E_i\mapsto E_i$, $F_i\mapsto F_i$, $K_i\mapsto K_i^{-1}$, $i\in \tl{\mr{I}}$, and $q$ to $q^{-1}.$ Here we think of  $q$ as an extra variable.

For a $U_q\tl{\fk{g}}$-module $V^q$, let $V^q_\nu$ be the $U_{q^{-1}}\tl{\fk{g}}$-module obtained from $V^q$ by twisting with $\nu$. Then the identity map is an isomorphism of  $U_{q^{-1}}\tl{\fk{g}}$-modules $V^{q^{-1}}\xto{\smash{\raisebox{-0.65ex}{\ensuremath{\sim}}}} V^q_\nu$.

The $R$-matrix commutes with action of $g\in U_q\tl{\fk{g}}$, therefore it commutes with action of $\nu(g)$. Then the  $R$-matrix $\check{R}^{V_1,V_2}(z;q)$ maps the  $U_{q^{-1}}\tl{\fk{g}}$-modules
\bee{
(V^q_1(z)\otimes V^q_2)_\nu = (V^q_1(z))_\nu\mathop{\otimes}_{\text{op}}\, (V^q_2)_\nu=V^{q^{-1}}_1(z)\mathop{\otimes}_{\text{op}} V^{q^{-1}}_2 \to (V^q_2\otimes V^q_1(z))_\nu=V_2^{q^{-1}}\mathop{\otimes}_{\text{op}} V_1^{q^{-1}}(z).
}
Thus we obtain
\bee{
\check{R}^{V_1,V_2}(z;q)=\check{R}^{V_1,V_2}_{\text{op}}(z;q^{-1}).
}
Now the third property is obtained by combining this with \eqref{Rop}.

The fourth property is well-known and straightforward.

The fifth property follows by the uniqueness of the intertwiner, since by \eqref{ad coproduct}, we have $\big(\check{R}^{V_1,V_2}(z)\big)^*$ is an intertwiner. 
\end{proof}

\begin{lem}\label{R0}
Let $V_1$, $V_2$ be irreducible representations of $U_q\tl{\fk{g}}$ such that as $U_q\fk{g}$-modules, $V_1$, $V_2$ are irreducible of highest weights $\la$, $\mu$ respectively. Suppose that the tensor product $L_\la\otimes L_\mu=\oplus_{\nu}\,L_\nu$ has trivial multiplicities. Then
\eq{\label{R check 0}
\check{R}^{V_1,V_2}(0)=\sum_{\nu}(-1)^{\nu}\,q^{(C(\nu)-C(\la+\mu))/2}P_\nu\ ,
}
where $P_{\nu}$ are projectors onto $L_{\nu}$, $(-1)^\nu=\pm 1$ is the eigenvalue of the flip operator $P$ on the $q\to 1$ limit of $L_\nu$, and $C(\nu)=(\nu,\nu+2\rho)$, with $\rho$ being the half sum of all positive roots, and $(\ ,\ )$ be the standard scalar product given on simple roots by $(\A_i,\A_j)=B_{ij}$.
\end{lem}
\begin{proof}
The proof is same as in \cite{DGZ94}. We provide a few extra details here. At $z=0$, the intertwiner $\check{R}^{V_2,V_1}(z)$ (up to a normalization constant) reduces to $P\,R^{\,\mu,\la}$ where $R^{\,\mu,\la}$ is the $R$-matrix for $U_q\fk{g}$ evaluated in $L_\mu\otimes L_\la$.

The quasitriangular Hopf algebra $U_q\fk{g}$ has a distinct central element $v$ satisfying $\sr{R}_{\text{op}}\sr{R}=(v\otimes v)\Delta(v^{-1})$. Here $\sr{R}$ is the universal $R$-matrix for finite type quantum algebra $U_q\fk{g}$. On the irreducible representation $L_\la$ of $U_q\fk{g}$, $v$ acts as $q^{-C}$ where $C$ is the Casimir element for $U(\fk{g})$ (see \cite{CP1} Section 8.3, Proposition 8.3.14). 
The Casimir element $C$ acts in the irreducible representation of $U(\fk{g})$ of highest weight $\lambda$  by the constant $C(\la)=(\la,\la+2\rho)$.

Then as in \cite{DGZ94}, we have 
\eq{\label{R0 unnormalized}
\sum_\nu f_\nu(0)^2\,P_\nu &\ = PR^{\,\la,\mu}PR^{\,\mu,\la} = R_{\text{op}}^{\,\mu,\la}R^{\,\mu,\la} = (\pi_\mu\otimes\pi_\la)(\sr{R}_{\text{op}}\sr{R}) \\ 
& = (\pi_\mu(v)\otimes\pi_\la(v))(\pi_\mu\otimes\pi_\la)\big(\Delta(v^{-1})\big)=\sum_\nu q^{C(\nu)-C(\mu)-C(\la)}\,P_\nu\ .
}
Thus, $f_\nu(0)=\pm q^{(C(\nu)-C(\mu)-C(\la))/2}$. Now \eqref{R check 0} follows after normalization.
\end{proof}

It is known that the submodules of tensor products of fundamental modules correspond to to zeroes and poles of $R$-matrices.
\begin{thm}[\cite{FM01}]\label{poles thm}
The tensor product $\tl{L}_{s_1}(a_1)\otimes\cdots\otimes \tl{L}_{s_n}(a_n)$ of fundamental representations of $U_q\tl{\fk{g}}$, is reducible if and only if for some $i,j\in\{1,\dots,n\}$, $i\ne j$, the normalized $R$-matrix $R^{V,W}(z)$ has a pole at $z=a_i/a_j$ where $V=\tl{L}_{s_i}(1)$, $W=\tl{L}_{s_j}(1)$. In that case, $a_i/a_j$ is necessarily equal to $q^k$, where $k$ is an integer. \qed
\end{thm}

The following lemma is used for the computation of the $R$-matrix in the case of $E_8$.

Let $V$ be the first fundamental representation of $U_q\tl{\fk{g}}$ where $\fk{g}$ is not of type A or E$_6$. Then we choose a basis $\{v_i\}_{i=1}^d$ of $V$ with the following properties. Denote $\bar v_i=v_{\bar i}=v_{d+1-i}$ if weight of $v_i$ is not zero and $\bar v_i=v_i$ otherwise. Then we require
that the sum of weights of $v_i$ and $\bar v_i$ is zero and, moreover, 

\eq{\label{choice of basis}
E_j v_i=\sum_{k} a_{ik}^{(j)} v_k \quad \text{ if\ and\ only\ if }\quad  F_j \bar v_i=\sum_{k} a_{ik}^{(j)} \bar v_k\ ,\quad j\in\mr{I}\ .
}
We construct such a basis for each type by a direct computation.  
In fact, the basis we choose is also orthonormal with respect to the Shapovalov form, and in addition to \eqref{choice of basis} we have $E_j^T=F_j$, $j\in \mr{I}$.

Let $t:V\to V$ be a linear map such that $v_i\mapsto \bar v_i$. Note that $t^2=\id$.

\begin{lem}\label{R and flip commute}
Let $V$ be the first fundamental representation of $U_q\tl{\fk{g}}$ where $\fk{g}$ is not of type A or E$_6$. Then
\eq{\label{t and Rop}
\check{R}^{V,V}(z)=(t\otimes t)P\check{R}^{V,V}(z)P(t\otimes t)\ .
}
Here $P$ is the flip operator.
\p{
Let $\upsilon:(U_q\tl{\fk{g}},\Delta)\to (U_q\tl{\fk{g}},\Delta^{\text{op}})$ be an isomorphism of Hopf algebras sending $E_i\mapsto F_i$, $F_i\mapsto E_i$, $K_i\mapsto K_i^{-1}$, $i\in \tl{\mr{I}}$. Let $V_\upsilon$ be the $U_q\tl{\fk{g}}$-module obtained from $V$ by twisting with $\upsilon$.

Clearly  $t:\, V \xto{\smash{\raisebox{-0.65ex}{\ensuremath{\sim}}}} V_\upsilon$ is an isomorphism of $U_q\tl{\fk{g}}$-modules. Since $\tau_z\circ \upsilon=\upsilon\circ \tau_{z^{-1}}$, we have 
$$t:V(z^{-1}) \xto{\smash{\raisebox{-0.65ex}{\ensuremath{\sim}}}} V_\upsilon(z^{-1})=\big(V(z)\big)_\upsilon\ .$$

The $R$-matrix commutes with action of $g\in U_q\tl{\fk{g}}$, therefore it commutes with action of $\upsilon(g)$. Then we have a map of  $U_{q}\tl{\fk{g}}$-modules
\bee{
\check{R}^{V,V}(z):\ \big(V(z)\otimes V\big)_\upsilon  \to \big(V \otimes V (z)\big)_\upsilon \ ,
}
Moreover, 
\bee{
\big(V(z)\otimes V\big)_\upsilon = \big(V(z)\big)_\upsilon\mathop{\otimes}_{\text{op}}\, V_\upsilon = (t\otimes t) \big(V(z^{-1}) \mathop{\otimes}_{\text{op}} V \big),\quad \text{and}\quad  
\big(V \otimes V (z)\big)_\upsilon= (t\otimes t)\big(V \mathop{\otimes}_{\text{op}} V(z^{-1})\big).
}
Therefore, $\check{R}^{V,V}(z)=(t\otimes t)\check{R}_{\text{op}}^{V,V}(z^{-1})(t\otimes t)$. Now \eqref{t and Rop} follows by combining this with \eqref{Rop}.
}
\end{lem}

\subsection{Yangians} Yangians $Y(\fk{g})$ are well-known rational counterparts of (a half of) $U_q\tl{\fk{g}}$.

The categories of representations of $Y(\fk{g})$ and representations of $U_q\tl{\fk{g}}$ for generic $q$, are equivalent. Moreover, the dimensions of the corresponding irreducible $Y(\fk{g})$ and $U_q\tl{\fk{g}}$-modules coincide.

The Yangians also possess the $R$-matrices which
lead to solutions of the rational QYBE.
Namely, let $V_1$, $V_2$, $V_3$ be three representations of $Y(\fk{g})$, then
\eq{\label{rqybe1}
R^{V_1,V_2}_{12}(u)R^{V_1,V_3}_{13}(u+v)R^{V_2,V_3}_{23}(v)=R^{V_2,V_3}_{23}(v)R^{V_1,V_3}_{13}(u+v)R^{V_1,V_2}_{12}(u)\,,
}
\eq{\label{rqybe}
\check{R}^{V_1,V_2}_{23}(u)\check{R}^{V_1,V_3}_{12}(u+v)\check{R}^{V_2,V_3}_{23}(v)=\check{R}^{V_2,V_3}_{12}(v)\check{R}^{V_1,V_3}_{23}(u+v)\check{R}^{V_1,V_2}_{12}(u)\,.
}

The Yangian $R$-matrices $R^{V,W}(u)$ and rational solutions of the QYBE can be obtained from the $U_q\tl{\fk{g}}$ $R$-matrices $R^{V,W}(z)$ for corresponding representations by setting $z=q_1^{2u}$ and taking the limit $q\to 1$ (up to a constant change of parameter).

\section{The computation of the \texorpdfstring{$R$}{2}-matrices by the \texorpdfstring{$q$}{2}-characters}\label{meth sec}

We state an algorithm that finds the $R$-matrix $\check{R}(z)=\check{R}^{V,V}(z)$ for first fundamental representations $V=\tilde L_1$ of all Lie algebra types. 

\subsection{Cases of multiplicity one} 
We start with the multiplicity-free case which covers all types except for E$_8$.
In all types except for E$_8$, the module $\tilde L_1=L_{\omega_1}$ is irreducible as a representation of $U_q\fk{g}$ and the direct sum decomposition of the tensor product $L_{\omega_1}\otimes L_{\omega_1}$ is multiplicity-free. 

We expect that the same algorithm is applicable to all multiplicity-free cases. However, to justify it one needs to prove analogs of Theorem \ref{poles thm} and the applicability of the algorithm of the computation of the $q$-characters. 


\begin{alg}
\leavevmode
\begin{enumerate}
\item Find the decomposition $L_{\omega_{1}}\otimes L_{\omega_{1}}\cong L_{\lambda_{1}}\oplus\cdots\oplus L_{\lambda_{n}}$ of $U_q\fk{g}$-modules. Here $\lambda_{1}=2\omega_1.$ 
\item For $k=1,\dots,n$, let $P_{\lambda_{k}}$ be the projector onto $L_{\lambda_{k}}$ along other summands.
\item Then $\check{R}(z)=f_1(z)P_{\lambda_{1}}+\cdots+f_n(z)P_{\lambda_{n}}$ for some rational functions $f_k(u)$. We set $f_1(z)=1$.\label{rational}
\item Each $f_k(z)$ is determined up to a scalar multiple by finding its zeros and poles using $q$-characters.\label{zeros and poles}
\item Since $\check{R}(1)=\id$ , we get a unique expression for $\check{R}(z)$.
\end{enumerate}
\end{alg}

The part \eqref{rational} is based on Lemma \ref{lemR} \eqref{try}.
The part \eqref{zeros and poles} is based on Theorem \ref{poles thm} and the following theorem. 

\begin{thm}\label{order of zeroes thm}
The functions $f_k$ have no double poles nor double zeroes.    
\end{thm}
\begin{proof} 
Let $q^{a_1^{(k)}},\dots,q^{a_{l_k}^{(k)}}$ be poles of $f_k(z)$. Note that in all cases $a_i^{(k)}\in\Z_{>0}$.
In every case, after computing $a_i^{(k)}$, we check that
\bee{
\sum_{j=1}^{l_k} a_j^{(k)}=\frac12(C(\la_k)-C(2\omega_1)).
}
We also have $f_k(1)=1$ because of $\check{R}(1)=\id$.
Then the theorem follows from Lemma \ref{R0}. Here we use Theorem \ref{poles thm}(\cite{FM01}) to conclude that all poles have the form $z=q^k$ with $k>0$ (see \cite{C00}). The corresponding zeroes have the form $z=q^{-k}$ by property (\ref{inversion relation}) of Lemma \ref{lemR}.

\end{proof}
At least for types A, B, D, E$_6$, and E$_7$, Theorem \ref{order of zeroes thm} can be deduced without case by case checking of Casimir values from the following general proposition.

\begin{prp}\label{order of zeroes prp}
The rational functions $f_k(z)$ have numerators and denominators of degree at most $n-1$.     
\end{prp}
\begin{proof}
The matrix coefficients  of operator $F_0\in U_q\tl{\fk{g}}$ are linear functions of $z$. We have $f_1(z)=1$. To find $f_k(z)$, $k>1$, we need to solve an $(n-1)\times (n-1)$ non-homogeneous system with linear coefficients and  linear right hand sides. The proposition follows.
\end{proof}

The actual degrees of  numerators and denominators of functions $f_k(z)$ are given in the following table.

\begin{center}\begin{tabular}{c@{\hspace{1cm}} c@{\hspace{1.5cm}} c c c c c}
Type & $n$ & \multicolumn{5}{c}{Degrees} \\
 A & 2 & 0 & 1 & & &\\ 
 B & 3 & 0 & 1 & 2 & & \\
 C & 3 & 0 & 1 & 1 & & \\
 D & 3 & 0 & 1 & 2 & & \\
 E$_6$ & 3 & 0 & 1 & 2 & & \\
 E$_7$ & 4 & 0 & 1 & 2 & 3 & \\
 F$_4$ & 5 & 0 & 1 & 1 & 2 & 2 \\
 G$_2$ & 4 & 0 & 1 & 1 & 2 & 
\end{tabular} \end{center}

In the cases we consider here, the $R$-matrices are known and $\check R(z)$ computed by the algorithm simply match the known answers.

For every case, we give the $E_0$ and $F_0$ actions. The $R$-matrix can be directly checked to commute with the action of $E_0$. In \cite{J86} and \cite{DGZ94} the functions $f_k(z)$ were obtained from the commutativity with the $E_0$.

\medskip

The rational case can be obtained similarly. Alternatively, one can set $z=q^{2u}$ and take the limit of $q\to 1$.

\subsection{Cases with non-trivial multiplicities}  
In the case when the $U_q\fk{g}$-decomposition has multiplicity:
$$
V\otimes V\cong M_1\otimes L_{\lambda_{1}}\oplus\cdots\oplus M_n\otimes L_{\lambda_{n}}, \qquad m_k=\dim M_k,
$$
the functions $f_k(z)$ become $m_k\times m_k$ matrices after one chooses bases in the spaces of singular vectors.
The entries of $f_k(z)$ are rational functions. Then the computation with the $q$-characters produces zeros and poles of the determinants of these matrices and their rank when determinant is zero. In addition, we have
\eq{\label{propeties of f}
f_k(z)=Pf_k(z)P\ ,\quad f_k(1)=\id\ ,\quad f_k(z)f_k(z^{-1})=\id \ ,
}
where $P$ is the flip operator (acting on singular vectors). We also know $f_k(0)$ and $f_k(\infty)$. Finally, since the $R$-matrix is self-adjoint, and our basis is orthogonal, we know that the ratio of $ij$ and $ji$ entries of $f_k(z)$ with $i\neq j$ is the ratio of squares of the Shapovalov norms of the vectors corresponding to columns $j$ and $i$.

Finally, we use the following conjecture.
\begin{con}\label{conj:simple poles}
    Suppose $V(a)\otimes V$ has a single non-trivial submodule. Then the normalized $R$-matrix $\check{R}^{V,V}(z)$ has at most simple pole at $z=a$.
\end{con}
In general we expect that the order of the pole at $z=a$ is at most one less than the number of irreducible subfactors. 

Note that in the trivial multiplicity case, we have Theorem \ref{order of zeroes thm}. Such an argument computes the determinants of $f_k(z)$.
We also have a general Proposition \ref{order of zeroes prp} which can be extended to non-trivial multiplicity case, though a bound it provides is not sharp.

With Conjecture \ref{conj:simple poles}, the properties  we discussed fix $f_k(z)$ up to a sign. We use the commutation relation with $E_0$ to fix the sign and check the final answer. 

For the case of $E_8$, the first fundamental representation (249-dimensional), splits as a representation of $U_q\fk{g}$, into a direct sum of irreducible first fundamental representation (248-dimensional) of $U_q\fk{g}$ and the trivial one-dimensional representation. Due to this, the direct sum decomposition of the second tensor power of $\tl{L}_1(a)$ has multiplicities, so we have a $2\times 2$ matrix $f_{\omega_0}(z)$ and a $3\times 3$ matrix $f_{\omega_1}(z)$. See Section \ref{E8 sec} for details.

\comment{
As a $U_q(E_8)$-module, $\tl{L}_1(a)$ decomposes as $L_{\omega_1}\oplus L_{\omega_0}$ and we have: 
$$\big(\tl{L}_1(a)\big)^{\otimes 2}\cong\big(\underbracket[0.1ex]{L_{\omega_1}}_{248}\oplus \underbracket[0.1ex]{L_{\omega_0}}_{1}\big)^{\otimes 2}\cong \underbracket[0.1ex]{L_{2\omega_1}}_{27000}\oplus \underbracket[0.1ex]{L_{\omega_2}}_{30380}\oplus \underbracket[0.1ex]{L_{\omega_7}}_{3875}\oplus\,3 \underbracket[0.1ex]{L_{\omega_1}}_{248}\oplus\,2 \underbracket[0.1ex]{L_{\omega_0}}_{1},$$
see Section \ref{E8 sec}.

Our general approach allows us to compute the $R$-matrix up to constants $\beta$ in the $3\times 3$ matrix, and $\eta$ in the $2\times 2$ matrix. Note that $L_{\omega_1}^{\otimes 2}$ has trivial multiplicities so Lemma \ref{R0} applies.
We fix $\beta$ and $\eta$ by a computer computation, using the commutativity of $\check{R}(z)$ with the action of $E_0$.
}

\medskip

Another way to obtain $R$-matrices with non-trivial multiplicities and for the other representations is provided by the fusion process which makes use of properties   
$$(\Delta\otimes\id)(\fk{R})=\fk{R}_{13}\fk{R}_{23}\,,\quad(\id\otimes\Delta)(\fk{R})=\fk{R}_{13}\fk{R}_{12}\,,$$
see Proposition \ref{R prop}.

We provide  two such examples to get extra examples of matrices corresponding to non-trivial multiplicities. 

First, in the case of G$_2$, we have
$$
\tl{L}_2(a)\subset \tl{L}_1(aq)\otimes \tl{L}_1(aq^{-1})\ .
$$
Therefore the intertwiner
$$
\tl{L}_1(zq)\otimes \tl{L}_1(zq^{-1})\otimes \tl{L}_1(q)\otimes \tl{L}_1(q^{-1}) \to \tl{L}_1(q)\otimes \tl{L}_1(q^{-1})\otimes \tl{L}_1(zq)\otimes \tl{L}_1(zq^{-1})
$$
given by 
\eq{\label{G2fusion}
\check{R}_{23}(zq^{2})\,\check{R}_{34}(z)\,\check{R}_{12}(z)\,\check{R}_{23}(zq^{-2})\ ,
}
where $\check{R}(z)=\check{R}^{\tl{L}_1,\tl{L}_1}(z)$, has a $225\times 225$ block. That block is the $R$-matrix  $\check{R}^{\tl{L}_2,\tl{L}_2}(z)$. This matrix can be checked to commute with $E_0$. Similar to the case of E$_8$, we have one $3\times 3$ matrix and one $2\times 2$ matrix, see Section \ref{mul cases G2}. As in the case of E$_8$ these two matrices can be found using $q$-characters, the knowledge of $\check{R}(0)$, $\check{R}(\infty)$ and the properties in \eqref{propeties of f}, up to a sign. 

\medskip

Second, in the case of A$_2$, we have 
$$\tl{L}_{\mr{1}_0\mr{2}_3}(a)\subset \tl{L}_1(a)\otimes \tl{L}_2(aq^3)\ .$$

Therefore the intertwiner 
$$\tl{L}_1(z)\otimes \tl{L}_2(zq^3)\otimes\tl{L}_1(1)\otimes\tl{L}_2(q^3)\to\tl{L}_1(1)\otimes\tl{L}_2(q^3)\otimes \tl{L}_1(z)\otimes \tl{L}_2(zq^3)$$
given by 
\eq{\label{A2fusion}
\check{R}_{23}^{12}(zq^{-3})\,\check{R}_{34}^{22}(z)\,\check{R}_{12}^{11}(z)\,\check{R}_{23}^{21}(zq^{3})\ ,
}
where $\check{R}^{ij}(z)$ is the $R$-matrix $\check{R}^{\tl{L}_i,\tl{L}_j}(z)$, has a $64\times 64$ block. That block is the $R$-matrix $\check{R}^{\tl{L}_{\mr{1}_0\mr{2}_3},\tl{L}_{\mr{1}_0\mr{2}_3}}(z)$. This matrix can be checked to commute with $E_0$.
In this case we get a $2\times 2$ matrix, see Section \ref{mul cases A2}. In this case the information obtained from $q$-characters seems to be insufficient as some submodules are indecomposable and we have a double pole.

\medskip

We note that the three $2\times 2$ matrices and two $3\times 3$ matrices we produce here together with the matrices appearing in the twisted cases, see \cite{DM25}, look alike.
We plan to discuss this phenomenon in the future.

\section{The classical cases}\label{cla sec}
 The  matrices $\check{R}(z)$ in classical types have been computed in \cite{J86}. The rational versions are given in \cite{KS82}. This section has no new $R$-matrices and serves as an illustration for our methods.

From now on, $\check{R}(z)$ denotes the intertwiner $\check{R}^{\tl{L}_1,\tl{L}_1}(z):\tl{L}_1(az)\otimes \tl{L}_1(a)\to\tl{L}_1(a)\otimes\tl{L}_1(az)$. When it is necessary to emphasize the dependence on $q$ we write $\check{R}(z\,;q)$ in place of $\check{R}(z)$.

For a space $L$, we denote $\cl{S}^2(L), \Lambda^2(L)\subset L\otimes L$ the symmetric and skew-symmetric squares of $L$.

\subsection{Type A\texorpdfstring{$_r\ (r\ge1)$}{2}}

The Dynkin diagram is: 
\bigskip
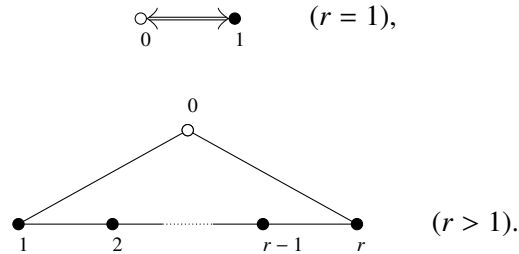
\begin{center}
\dynkin [extended, edge length=1.25cm, root radius=0.075cm, label macro/.code={\drlap{#1}}, labels={0, 1}, ordering=Kac] A[1]{1} \qquad $(r=1)$,
\end{center}
\medskip
\begin{center}
\dynkin [extended, edge length=1.25cm, root radius=0.075cm, label macro/.code={\drlap{#1}}, labels={0, 1, 2, r-1, r}, ordering=Kac] A[1]{} \qquad $(r>1)$.
\end{center}
The $(r+1)$-dimensional $U_q($A$_r^{(1)})$-module $\tl{L}_1(a)$ restricted to $U_q($A$_r)$ is isomorphic to $L_{\omega_1}$. As $U_q($A$_r)$-modules we have 
\eq{\label{tensorA} 
\underbracket[0.1ex]{L_{\omega_1}}_{r+1}\otimes \underbracket[0.1ex]{L_{\omega_1}}_{r+1}\cong \underbracket[0.1ex]{L_{2\omega_1}}_{\binom{r+2}{2}}\oplus \underbracket[0.1ex]{L_{\omega_2}}_{\binom{r+1}{2}}\ .
}
Here and in similar formulas, using under-brackets we show the dimensions of the modules.

In the $q\to 1$ limit, $L_{2\omega_1}\mapsto\cl{S}^2(L_{\omega_1})$ and $L_{\omega_2}\mapsto \Lambda^2(L_{\omega_1})$. For $r=1$, $L_{\omega_2}$ has to be replaced with $L_{\omega_0}$.

The $q$-character of  $\tl L_1=\tl{L}_{\mr{1}_0}$ has $r+1$ terms and there are no weight zero terms:
$$\chi_q(\mr{1}_0)=\mr{1}_0+\ul{\mr{1}_2^{-1}\mr{2}_1}+\mr{2}_3^{-1}\mr{3}_2+\mr{3}_4^{-1}\mr{4}_3+\cdots+(\mr{r}-\mr{1})_r^{-1}\mr{r}_{r-1}+\mr{r}_{r+1}^{-1}\,.$$
We underline monomials which may produce dominant monomials in the product $\chi_q(\mr{1}_0)\chi_q(\mr{1}_a)$.

Using the $q$-characters we compute the zeros and poles of $\check{R}(z)$ and the corresponding kernels and cokernels. Here we loosely say $z$ is a zero of an $R$-matrix if the $R$-matrix is a well defined but a degenerate operator (not totally zero operator). We repeatedly use Theorem \ref{q char arg}  to show that the participating $q$-characters have only one dominant monomial, see the discussion of \eqref{qchar arg for G2}. We can tell apart zeroes from poles since $z=q^k$ with $k<0$ corresponds to the cyclic tensor products by \cite{C00}, and therefore to zeroes of the $R$-matrix.
Here and below we do not give details of such standard computations with the $q$-characters and summarize the results in subsequent lemmas.  In the lemmas we show only poles of $\check{R}(z)$ and isomorphisms are isomorphisms of $U_q({\fk g})$-modules. The zeroes are obtained by changing $q\to q^{-1}$ and 
Quotient modules $\leftrightarrow$ Submodules.
\begin{lem}
The poles of the $R$-matrix $\check{R}(z)$, the corresponding submodules and quotient modules are given by
\begin{center}\begin{tabular}{c@{\hspace{1cm}} c c}
Poles & Submodules & Quotient modules  \\
$q^{2}$ & $\tl{L}_{\mr{1}_a\mr{1}_{aq^{-2}}}\cong L_{2\omega_1}$ & \hspace{10pt} $\tl{L}_{\mr{2}_{aq^{-1}}}\cong L_{\omega_2}\ (L_{\omega_0}\text{\ for\ }r=1)$ \\
\end{tabular}.\end{center}
\qed\end{lem}

We choose a basis $\{v_i:1\le i\le r+1\}$ for $L_{\omega_1}$ in the standard way, so that $v_1$ is a non-zero highest weight vector and $F_iv_i=v_{i+1}$. In the chosen basis, $v_1\otimes v_1$ is a singular vector of weight $2\omega_1$, and $q\,v_1\otimes v_2-v_2\otimes v_1$ is a singular vector of weight $\omega_2$. We generate respectively the modules $L_{2\omega_1}$ and $L_{\omega_2}$ using these singular vectors. 

For $\lambda=2\omega_1,\omega_2$, let $P_\lambda^q$ be the projector onto the $U_q($A$_r)$-module $L_\lambda$ in the decomposition \eqref{tensorA}, and let $E_{ij}$ be matrix units corresponding to the chosen basis, that is, $E_{ij}(v_k)=\D_{jk}v_i$. 

\begin{thm}
In terms of projectors, we have
\eq{\label{R proj A}
\check{R}(z)=P_{2\omega_1}^q-q^{-2}\frac{1-q^{2}z}{1-q^{-2}z}P_{\omega_2}^q\ .
}
In terms of matrix units, we have
\eq{\label{RqA}
\check{R}(z)=\sum_{i=1}^{r+1} E_{ii}\otimes E_{ii}\,+\,\frac{z(q-q^{-1})}{q-q^{-1}z}\,\sum_{i<j}E_{ii}\otimes E_{jj}\,+\,\frac{q-q^{-1}}{q-q^{-1}z}\,\sum_{i>j}E_{ii}\otimes E_{jj}\,+\,\frac{1-z}{q-q^{-1}z}\,\sum_{i\ne j}E_{ij}\otimes E_{ji}\ .
}
\qed
\end{thm}
One can directly check that the $R$-matrix commutes with the action of $E_0$ and $F_0$. Namely,
\eq{\label{R and E0 commutation relation}
\check{R}(a/b)\,\Delta E_0(a,b)=\Delta E_0(b,a)\,\check{R}(a/b)\quad\text{and}\quad \check{R}(a/b)\,\Delta F_0(a,b)=\Delta F_0(b,a)\,\check{R}(a/b)\ ,
}
where 
$\Delta E_0(a,b)=E_0(a)\otimes K_0^{1/2}+K_0^{-1/2}\otimes E_0(b)$, $\Delta F_0(a,b)=F_0(a)\otimes K_0^{1/2}+K_0^{-1/2}\otimes F_0(b)$,
$$K_0=q^{-1}E_{11}+\sum_{i=2}^rE_{ii}+qE_{r+1,r+1}\ ,\quad E_0(a)=a E_{r+1,1}\ ,$$
and $F_0(a)=a^{-1}E_{1,r+1}$ is the transpose of $a^{-2}E_0(a)$.

Let $
\displaystyle P_\lambda=\lim_{q\to 1}P_{\lambda}^q\,\,$ be the $U($A$_r)$ projector, let $I$ be the identity operator, and let $P$ be the flip operator.
\begin{cor}
In the rational case, the corresponding rational $R$-matrix is given by 
\eq{\label{RuA}
\check{R}(u)=P_{2\omega_1}+\frac{1+u}{1-u}P_{\omega_2}=\frac{1}{1-u}(I-uP)\ .
}
\end{cor}
\begin{proof}
We substitute $z=q^{2u}$ in \eqref{R proj A} and \eqref{RqA} and take the limit $q\to 1$.
\end{proof}

\subsection{Type B\texorpdfstring{$_r\ (r\ge 2)$}{2}}

The Dynkin diagrams are: 
\bigskip
\begin{center}
\dynkin [extended, edge length=1.25cm, root radius=0.075cm, label macro/.code={\drlap{#1}}, labels={0, 1, 2}, ordering=Kac] B[1]{2} \qquad $(r=2)$,
\end{center}
\medskip
\begin{center}
\dynkin [extended, edge length=1.25cm, root radius=0.075cm, label macro/.code={\drlap{#1}}, labels={0, 1, 2, 3, r-2, r-1, r}, ordering=Kac] B[1]{} \qquad $(r>2)$.
\end{center}

The $(2r+1)$-dimensional $U_q($B$_r^{(1)})$-module $\tl{L}_1(a)$ restricted to $U_q($B$_r)$ is isomorphic to $L_{\omega_1}$. For $r>2$, as  $U_q($B$_r)$-modules we have
\eq{\label{tensorB}
\underbracket[0.1ex]{L_{\omega_1}}_{2r+1}\otimes \underbracket[0.1ex]{L_{\omega_1}}_{2r+1}\cong \underbracket[0.1ex]{L_{2\omega_1}}_{r(2r+3)}\oplus \underbracket[0.1ex]{L_{\omega_2}}_{\binom{2r+1}{2}}\oplus \underbracket[0.1ex]{L_{\omega_0}}_{1}\ .
}
In the $q\to 1$ limit, $L_{2\omega_1}\oplus L_{\omega_0}\mapsto \cl{S}^2(L_{\omega_1})$ and $L_{\omega_2}\mapsto \Lambda^2(L_{\omega_1})$. For $r=2$, $L_{\omega_2}$ has to be replaced with $L_{2\omega_2}$.

For $r=2$, the $q$-character of $\tl L_1=\tl{L}_{{\mr{1}_0}}$ has $5$ terms and there is 1 weight zero term (shown in box): 
$$\chi_q(\mr{1}_0)=\mr{1}_{0} + \ul{\mr{1}_{4}^{-1}\mr{2}_{1}\mr{2}_{3}} + \boxed{\mr{2}_{5}^{-1}\mr{2}_{1}} + \mr{1}_{2}\mr{2}_{3}^{-1}\mr{2}_{5}^{-1} + \ul{\mr{1}_{6}^{-1}}\ .$$
Using the $q$-characters, we compute the zeros and poles of $\check{R}(z)$ and the corresponding kernels and cokernels.
\begin{lem}
The poles of the $R$-matrix $\check{R}(z)$, the corresponding submodules and quotient modules are given by
\begin{center}\begin{tabular}{c@{\hspace{1cm}} c c}
Poles &  Submodules & Quotient modules  \\

$q^{4}$ & $\tl{L}_{\mr{1}_a\mr{1}_{aq^{-4}}}\cong L_{2\omega_1}$ & $\hspace{25pt} \tl{L}_{\mr{2}_{aq^{-1}}\mr{2}_{aq^{-3}}}\cong L_{2\omega_2}\oplus L_{\omega_0}$ 
\vspace{2pt}\\

$q^{6}$ & $\hspace{33pt} \tl{L}_{\mr{1}_a\mr{1}_{aq^{-6}}}\cong L_{2\omega_1}\oplus L_{2\omega_2}$ & $\hspace{23pt}\tl{L}_{\scriptscriptstyle{1}}\cong L_{\omega_0}$ \\
\end{tabular}\ .\end{center}
\qed
\end{lem}
For $r>2$, the $q$-character of $\tl L_1=\tl{L}_{\mr{1}_0}$ has $2r+1$ terms and there is 1 weight zero term (shown in box):
$$\chi_q(\mr{1}_0)=\mr{1}_0+\ul{\mr{1}_4^{-1}\mr{2}_2}+\cdots+(\mr{r}-\mr{2})_{2r-2}^{-1}(\mr{r}-\mr{1})_{2r-4}$$
$$+\,(\mr{r}-\mr{1})_{2r}^{-1}\mr{r}_{2r-3}\mr{r}_{2r-1}+\boxed{\mr{r}_{2r+1}^{-1}\mr{r}_{2r-3}}+(\mr{r}-\mr{1})_{2r-2}\mr{r}_{2r-1}^{-1}\mr{r}_{2r+1}^{-1}$$
$$+\,(\mr{r}-\mr{2})_{2r}(\mr{r}-\mr{1})_{2r+2}^{-1}+\cdots+\mr{1}_{4r-6}\mr{2}_{4r-4}^{-1} + \ul{\mr{1}_{4r-2}^{-1}}\ .$$
Using the $q$-characters, we compute the zeros and poles of $\check{R}(z)$ and the corresponding kernels and cokernels.
\begin{lem}
The poles of the $R$-matrix $\check{R}(z)$, the corresponding submodules and quotient modules are given by
\begin{center}\begin{tabular}{c@{\hspace{1cm}} c c}
Poles & Submodules & Quotient modules  \\

$q^{4}$ & $\tl{L}_{\mr{1}_a\mr{1}_{aq^{-4}}}\cong L_{2\omega_1}$ & $\hspace{20pt} \tl{L}_{\mr{2}_{aq^{-2}}}\cong L_{\omega_2}\oplus L_{\omega_0}$ 
\vspace{2pt}\\

$q^{4r-2}$ & $\hspace{18pt} \tl{L}_{\mr{1}_a\mr{1}_{aq^{-4r+2}}}\cong L_{2\omega_1}\oplus L_{\omega_2}$ & $\hspace{4pt} \tl{L}_{\scriptscriptstyle{1}}\cong L_{\omega_0}$ \\
\end{tabular}\ .\end{center}
\qed
\end{lem}
We choose a basis $\{v_i:1\le i\le 2r+1\}$ for $L_{\omega_1}$ in the standard way so that $v_1$ is a non-zero highest weight vector, $F_iv_i=v_{i+1}$,  $F_iv_{\ol{i+1}}=v_{\ol{i}}$, $i=1,\dots, r-1$, $\ol{i}=2r+2-i$, and for $i=r$, $F_r.v_r=\sqrt{[2]}v_{r+1}$, $F_r.v_{\ol{r+1}}=\sqrt{[2]}v_{\ol{r}}$. 
In the chosen basis, $v_1\otimes v_1$ is a singular vector of weight $2\omega_1$, and $q^2\,v_1\otimes v_2-v_2\otimes v_1$ is a singular vector of weight $\omega_2$. 
We generate respectively the modules $L_{2\omega_1}$ and $L_{\omega_2}$ using these singular vectors. 

Let $\ve_i^q=(-1)^{r+1-i}q^{2r-2i+1}$, $\ve_{\ol{i}}^q=\ve_i^{q^{-1}}$, $1\le i\le r$, $\ve_{r+1}^q=1$. A singular vector $v_0\in L_{\omega_1}^{\otimes2}$ of weight $\omega_0$ is given by 
$$v_0=\sum_{i=1}^{2r+1}\ve_i^qv_i\otimes v_{\ol{i}}\ .$$ 

For $\lambda=2\omega_1,\omega_2\, (2\omega_2\text{ when }r=2),\omega_0$, let $P_\lambda^q$ be the projector onto the $U_q($B$_r)$-module $L_\lambda$ in the decomposition \eqref{tensorB}, and let $E_{ij}$ be matrix units corresponding to the chosen basis, that is, $E_{ij}v_k=\D_{jk}v_i$. 

\begin{thm}
In terms of projectors, we have
\eq{\label{R proj B}
\check{R}(z)=P_{2\omega_1}^q-q^{-4}\frac{1-q^{4}z}{1-q^{-4}z}P_{\omega_2}^q+q^{-4r-2}\frac{(1-q^{4}z)(1-q^{4r-2}z)}{(1-q^{-4}z)(1-q^{-4r+2}z)}P_{\omega_0}^q\ .
}
Here, in the case of $r=2$, $P_{\omega_2}^q$ is replaced by $P_{2\omega_2}^q$.

In terms of matrix units, we have
\eq{\label{RqB}
\check{R}(z)=\big(\check{R}(z\,;q^2)\big)_{\fk{sl}_{2r+1}}-\frac{(q^2-q^{-2})(1-z)}{(q^2-q^{-2}z)(q^{2r-1}-q^{-2r+1}z)}Q(z)\ ,
}
where $\big(\check{R}(z)\big)_{\fk{sl}_{2r+1}}$ is the A$_{2r}$ (or $\fk{sl}_{2r+1}$) trigonometric $R$-matrix in \eqref{RqA} and $Q(z)$ is given by
\bee{Q(z)=z\sum_{i+j<2r+2}\frac{\ve_i^q\ve_j^q}{q^{2r-1}}\,E_{ij}\otimes E_{\ol{i}\,\ol{j}}+\sum_{i+j>2r+2}\frac{\ve_i^q\ve_j^q}{q^{-2r+1}}\,E_{ij}\otimes E_{\ol{i}\,\ol{j}} +\,\frac{q^{2r-2}+q^{-2r+2}z}{q+q^{-1}}\,\sum_{\substack{i+j=2r+2\\ i\ne r+1}}E_{ij}\otimes E_{\ol{i}\,\ol{j}}\, \\
+\,\frac{q^{2r}+q^{-2r}z}{q+q^{-1}}\,E_{r+1,r+1}\otimes E_{r+1,r+1}\ .}
\qed
\end{thm}

One can directly check that the $R$-matrix commutes with the action of $E_0$ and $F_0$, where 
$$K_0=q^{-2}\big(E_{11}+E_{22}\big)+\sum_{i=3}^{2r-1}E_{ii}+q^2\big(E_{2r,2r}+E_{2r+1,2r+1}\big)\ ,\quad E_0(a)=a\big(E_{2r,1}+E_{2r+1,2}\big)\ ,$$
and $F_0(a)$ is the transpose of $a^{-2}E_0(a)$.

Let $\displaystyle P_\lambda=\lim_{q\to1}P_\lambda^q$ be the $U($B$_r)$ projector, let $I$ be the identity operator, let $P$ be the flip operator, and let $Q$ be given by
$$Q=\sum_{i,j=1}^{2r+1}(-1)^{i+j}E_{ij}\otimes E_{\ol{i}\,\ol{j}}=(2r+1)\,P_{\omega_0}\ .$$

\begin{cor}
In the rational case, the corresponding $R$-matrix is given by
\eq{\label{RuB}
\check{R}(u)=P_{2\omega_1}+\frac{1+u}{1-u}P_{\omega_2}+\frac{(1+u)(2r-1+2u)}{(1-u)(2r-1-2u)}P_{\omega_0}=\frac{1}{1-u}\bigg(I-uP+\frac{2u}{2r-1-2u}Q\bigg)\ .
}
In the case of $r=2$, $P_{\omega_2}$ is replaced by $P_{2\omega_2}$.
\end{cor}
\begin{proof}
We substitute $z=q^{4u}$ in \eqref{R proj B} and \eqref{RqB} and take the limit $q\to 1$.
\end{proof}

\subsection{Type C\texorpdfstring{$_r\ (r\ge 2)$}{2}}

The Dynkin diagram is:
\medskip
\begin{center}
\dynkin [extended, edge length=1.25cm, root radius=0.075cm, label macro/.code={\drlap{#1}}, labels={0, 1, 2, r-2, r-1, r}, ordering=Kac] C[1]{}.
\end{center}

\medskip
 
The $2r$-dimensional $U_q\big($C$_r^{(1)}\big)$-module $\tl{L}_{1}(a)$ when restricted to $U_q($C$_r)$ is isomorphic to $L_{\omega_1}$. As $U_q($C$_r)$-modules we have
\eq{\label{tensorC}
\underbracket[0.1ex]{L_{\omega_1}}_{2r}\otimes \underbracket[0.1ex]{L_{\omega_1}}_{2r}\cong \underbracket[0.1ex]{L_{2\omega_1}}_{r(2r+1)}\oplus \underbracket[0.1ex]{L_{\omega_2}}_{(r-1)(2r+1)}\oplus \underbracket[0.1ex]{L_{\omega_0}}_{1}\ .
}
In the $q\to 1$ limit, $L_{2\omega_1}\mapsto \cl{S}^2(L_{\omega_1})$ and $L_{\omega_2}\oplus L_{\omega_0}\mapsto \Lambda^2(L_{\omega_1})$.

The $q$-character of $\tl L_1=\tl{L}_{\mr{1}_0}$ has $2r$ terms and there are no weight zero terms:
$$\chi_q(\mr{1}_0)=\mr{1}_0+\ul{\mr{1}_2^{-1}\mr{2}_1}+\cdots+(\mr{r}-\mr{1})_r^{-1}\mr{r}_{r-1}+(\mr{r}-\mr{1})_{r+2}\mr{r}_{r+3}^{-1}+\cdots+\mr{1}_{2r}\mr{2}_{2r+1}^{-1}+\ul{\mr{1}_{2r+2}^{-1}}\ .$$
Using the $q$-characters we compute the zeros and poles of $\check{R}(z)$ and the corresponding kernels and cokernels.
\begin{lem}
The poles of the $R$-matrix $\check{R}(z)$, the corresponding submodules and quotient modules are given by
\begin{center}\begin{tabular}{c@{\hspace{1cm}} c c}
Poles & Submodules & Quotient modules  \\

$q^{2}$ & $\hspace{27pt} \tl{L}_{\mr{1}_a\mr{1}_{aq^{-2}}}\cong L_{2\omega_1}\oplus L_{\omega_0}$ & $\hspace{10pt} \tl{L}_{\mr{2}_{aq^{-1}}}\cong L_{\omega_2}$ 
\vspace{2pt}\\

$q^{2r+2}$ & $\hspace{19pt} \tl{L}_{\mr{1}_a\mr{1}_{aq^{-2r-2}}}\cong L_{2\omega_1}\oplus L_{\omega_2}$ & $\hspace{23pt} \tl{L}_{\scriptscriptstyle{1}}\cong L_{\omega_0}$
\end{tabular}\ .\end{center}
\qed
\end{lem}

We choose a basis $\{v_i:1\le i\le 2r\}$ for $L_{\omega_1}$ in the standard way so that $v_1$ is a non-zero highest weight vector,
$F_iv_i=v_{i+1}$ and $F_iv_{\ol{i+1}}=v_{\ol{i}}$,  where $\ol{i}=2r+1-i$, and $i=1,\dots,r$.
In the chosen basis, $v_1\otimes v_1$ is a singular vector of weight $2\omega_1$, and $q\,v_1\otimes v_2-v_2\otimes v_1$ is a singular vector of weight $\omega_2$. We generate respectively the modules $L_{2\omega_1}$ and $L_{\omega_2}$ using these singular vectors. 

Let 
$\ve_i^q=(-q)^{r+1-i}$, $\ve_{\ol{i}}^q=-\ve_i^{q^{-1}}$, $1\le i\le r$. A singular vector $v_0\in L_{\omega_1}^{\otimes2}$ of weight $\omega_0$ is given by 
$$v_0=\sum_{i=1}^{2r}\ve_i^qv_i\otimes v_{\ol{i}}\ .$$

For $\lambda=2\omega_1, \omega_2, \omega_0$, let $P_\lambda^q$ be the projector onto the $U_q($C$_r)$-module $L_{\lambda}$ in the decomposition \eqref{tensorC}, and let $E_{ij}$ be matrix units corresponding to the chosen basis, that is, $E_{ij}v_k=\D_{jk}v_i$.  

\begin{thm}
In terms of projectors, we have
\eq{\label{R proj C}
\check{R}(z)=P_{2\omega_1}^q-q^{-2}\frac{1-q^{2}z}{1-q^{-2}z}P_{\omega_2}^q-q^{-2r-2}\frac{1-q^{2r+2}z}{1-q^{-2r-2}z}P_{\omega_0}^q\ .
}
In terms of matrix units, we have
\eq{\label{RqC}
\check{R}(z)=\big(\check{R}(z)\big)_{\fk{sl}_{2r}}+\frac{(q-q^{-1})(1-z)}{(q-q^{-1}z)(q^{r+1}-q^{-r-1}z)}Q(z)\ ,
}
where $\big(\check{R}(z)\big)_{\fk{sl}_{2r}}$ is the A$_{2r-1}$ (or $\fk{sl}_{2r}$) trigonometric $R$-matrix in \eqref{RqA} and $Q(z)$ is given by
$$Q(z)=z\sum_{i+j<2r+1}\frac{\ve_i^q\ve_j^q}{q^{r+1}}\,E_{ij}\otimes E_{\ol{i}\,\ol{j}}
+\sum_{i+j>2r+1}\frac{\ve_i^q\ve_j^q}{q^{-r-1}}\,E_{ij}\otimes E_{\ol{i}\,\ol{j}}\,-\,\frac{q^{r+\frac{1}{2}}+q^{-r-\frac{1}{2}}z}{q^{\frac{1}{2}}+q^{-\frac{1}{2}}}\,\sum_{i+j=2r+1}E_{ij}\otimes E_{\ol{i}\,\ol{j}}\ .$$
\qed
\end{thm}

One can directly check that the $R$-matrix commutes with the action of $E_0$ and $F_0$, where 
$$K_0=q^{-2}E_{11}+\sum_{i=2}^{2r-1}E_{ii}+q^2E_{2r,2r}\ ,\quad E_0(a)=aE_{2r,1}\ ,$$
and $F_0(a)$ is the transpose of $a^{-2}E_0(a)$.

Let $\displaystyle P_\lambda=\lim_{q\to1}P_\lambda^q$ be the $U($C$_r)$ projector, let $I$ be the identity operator, let $P$ be the flip operator, and let $Q$ be given by 
\bee{
Q=\sum_{i,j=1}^{2r}(-1)^{i+j}E_{ij}\otimes E_{\ol{i}\,\ol{j}}=2r\,P_{\omega_0}\ .
}

\begin{cor}
In the rational case, the corresponding $R$-matrix is given by
\eq{\label{RuC}
\check{R}(u)=P_{2\omega_1}+\frac{1+u}{1-u}P_{\omega_2}+\frac{r+1+u}{r+1-u}P_{\omega_0}=\frac{1}{1-u}\bigg(I-uP-\frac{u}{r+1-u}Q\bigg)\ .
}
\end{cor}
\begin{proof}
We substitute $z=q^{2u}$ in \eqref{R proj C} and \eqref{RqC} and take the limit $q\to 1$. 
\end{proof}

\subsection{Type D\texorpdfstring{$_r\ (r\ge 4)$}{2}}

The Dynkin diagram is:

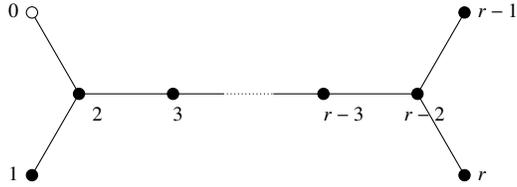
\begin{center}
\dynkin [extended, edge length=1.25cm, root radius=0.075cm, label macro/.code={\drlap{#1}}, labels={0, 1, 2, 3, r-3, r-2, r-1, r}, ordering=Kac] D[1]{}.
\end{center}

The $2r$-dimensional $U_q\big($D$_r^{(1)}\big)$-module $\tl{L}_1(a)$ when restricted to $U_q($D$_r)$ is isomorphic to $L_{\omega_1}$. As $U_q($D$_r)$-modules we have
\eq{\label{tensorD}
\underbracket[0.1ex]{L_{\omega_1}}_{2r}\otimes \underbracket[0.1ex]{L_{\omega_1}}_{2r}\cong \underbracket[0.1ex]{L_{2\omega_1}}_{(r+1)(2r-1)}\oplus \underbracket[0.1ex]{L_{\omega_2}}_{r(2r-1)}\oplus \underbracket[0.1ex]{L_{\omega_0}}_{1}\ .
}
In the $q\to 1$ limit, $L_{2\omega_1}\oplus L_{\omega_0}\mapsto \cl{S}^2(L_{\omega_1})$ and $L_{\omega_2}\mapsto \Lambda^2(L_{\omega_1})$.

The $q$-character of $\tl L_1=\tl{L}_{\mr{1}_0}$ has $2r$ terms and there are no weight zero terms:
$$\chi_q(\mr{1}_0)=\mr{1}_0+\ul{\mr{1}_2^{-1}\mr{2}_1}+\cdots+(\mr{r}-\mr{3})_{r-2}^{-1}(\mr{r}-\mr{2})_{r-3}+(\mr{r}-\mr{2})_{r-1}^{-1}(\mr{r}-\mr{1})_{r-2}\mr{r}_{r-2}+(\mr{r}-\mr{1})_r^{-1}\mr{r}_{r-2}$$
$$+\,(\mr{r}-\mr{1})_{r-2}\mr{r}_{r}^{-1}+(\mr{r}-\mr{2})_{r-1}(\mr{r}-\mr{1})_{r}^{-1}\mr{r}_{r}^{-1}+(\mr{r}-\mr{3})_{r}(\mr{r}-\mr{2})_{r+1}^{-1}+\cdots+\mr{1}_{2r-4}\mr{2}_{2r-3}^{-1}+\ul{\mr{1}_{2r-2}^{-1}}\ .$$
 Using the $q$-characters we compute the zeros and poles of $\check{R}(z)$ and the corresponding kernels and cokernels. 
\begin{lem}
The poles of the $R$-matrix $\check{R}(z)$, the corresponding submodules and quotient modules are given by
\begin{center}\begin{tabular}{c@{\hspace{1cm}} c c}
Poles & Submodules & Quotient modules  \\

$q^{2}$ & $\tl{L}_{\mr{1}_a\mr{1}_{aq^{-2}}}\cong L_{2\omega_1}$ & $\hspace{20pt} \tl{L}_{\mr{2}_{aq^{-1}}}\cong L_{\omega_2}\oplus L_{\omega_0}$ 
\vspace{2pt}\\

$q^{2r-2}$ & $\hspace{18pt} \tl{L}_{\mr{1}_a\mr{1}_{aq^{-2r+2}}}\cong L_{2\omega_1}\oplus L_{\omega_2}$ & $\hspace{5pt}\tl{L}_{\scriptscriptstyle{1}}\cong L_{\omega_0}$ \\
\end{tabular}\ .\end{center}
\qed
\end{lem}

We choose a basis $\{v_i:1\le i\le 2r\}$ for $L_{\omega_1}$ in the standard way so that $v_1$ is a non-zero highest weight vector, $F_iv_i=v_{i+1}$, $F_iv_{\ol{i+1}}=v_{\ol{i}}$, where $\ol{i}=2r+1-i$ and $i=1,\dots,r-1$, and $F_rv_{r-1}=v_{r+1}$, $F_rv_{\ol{r+1}}=v_{\ol{r-1}}$. In the chosen basis, $v_1\otimes v_1$ is a singular vector of weight $2\omega_1$, and $q\,v_1\otimes v_2-v_2\otimes v_1$ is a singular vector of weight $\omega_2$. We generate respectively the modules $L_{2\omega_1}$ and $L_{\omega_2}$ using these singular vectors. 

Let 
$\ve_i^q=(-q)^{r-i}$, $\ve_{\ol{i}}^q=\ve_i^{q^{-1}}$, $1\le i\le r$. A singular vector $v_0\in L_{\omega_1}^{\otimes2}$ of weight $\omega_0$ is given by 
$$v_0=\sum_{i=1}^{2r}\ve_i^qv_i\otimes v_{\ol{i}}\ .$$

For $\lambda=2\omega_1,\omega_2,\omega_0$, let $P_\lambda^q$ be the projector onto the $U_q($D$_r)$-module $L_{\lambda}$ in the decomposition \eqref{tensorD}, and let $E_{ij}$ be matrix units corresponding to the chosen basis, that is, $E_{ij}v_k=\D_{jk}v_i$.

\begin{thm}
In terms of projectors, we have
\eq{\label{R proj D}
\check{R}(z)=P_{2\omega_1}^q-q^{-2}\frac{1-q^{2}z}{1-q^{-2}z}P_{\omega_2}^q+q^{-2r}\frac{(1-q^{2}z)(1-q^{2r-2}z)}{(1-q^{-2}z)(1-q^{-2r+2}z)}P_{\omega_0}^q\ .
}
In terms of matrix units, we have
\eq{\label{RqD}
\check{R}(z)=\big(\check{R}(z)\big)_{\fk{sl}_{2r}}-\frac{(q-q^{-1})(1-z)}{(q-q^{-1}z)(q^{r-1}-q^{-r+1}z)}Q(z)\ ,
}
where $\big(\check{R}(z)\big)_{\fk{sl}_{2r}}$ is the A$_{2r-1}$ or $\fk{sl}_{2r}$ trigonometric $R$-matrix in \eqref{RqA} and $Q(z)$ is given by
$$Q(z)=z\sum_{i+j<2r+1}\frac{\ve_i^q\ve_j^q}{q^{r-1}}\,E_{ij}\otimes E_{\ol{i}\,\ol{j}}+\sum_{i+j>2r+1}\frac{\ve_i^q\ve_j^q}{q^{-r+1}}\,E_{ij}\otimes E_{\ol{i}\,\ol{j}}\,+\,\frac{q^{r-\frac{3}{2}}+q^{-r+\frac{3}{2}}z}{q^{\frac{1}{2}}+q^{-\frac{1}{2}}}\,\sum_{i+j=2r+1}E_{ij}\otimes E_{\ol{i}\,\ol{j}}\ .$$
\qed
\end{thm}

One can directly check that the $R$-matrix commutes with the action of $E_0$ and $F_0$, where 
$$K_0=q^{-1}\big(E_{11}+E_{22}\big)+\sum_{i=3}^{2r-2}E_{ii}+q\big(E_{2r-1,2r-1}+E_{2r,2r}\big)\ ,\quad E_0(a)=a \big(E_{2r-1,1}+E_{2r,2}\big)\ ,$$
and $F_0(a)$ is the transpose of $a^{-2}E_0(a)$.

Let $\displaystyle P_\lambda=\lim_{q\to1}P_\lambda^q$ be the $U($D$_r)$ projector, let $I$ be the identity operator, let $P$ be the flip operator, and let $Q$ be given by
\bee{
Q=\sum_{i,j=1}^{2r}(-1)^{i+j} E_{ij}\otimes E_{\ol{i}\,\ol{j}}=2r\,P_{\omega_0}\ .
}

\begin{cor}
In the rational case, the corresponding $R$-matrix is given by
\eq{\label{RuD}
\check{R}(u)=P_{2\omega_1}+\frac{1+u}{1-u}P_{\omega_2}+\frac{(1+u)(r-1+u)}{(1-u)(r-1-u)}P_{\omega_0}=\frac{1}{1-u}\bigg(I-uP+\frac{u}{r-1-u}Q\bigg)\ .
}
\end{cor}
\begin{proof}
We substitute $z=q^{2u}$ in \eqref{R proj D} and \eqref{RqD} and take the limit $q\to 1$.
\end{proof}

\section{The exceptional cases}\label{ex sec}
In this section we present the formulas for $\check R(z)$ for exceptional types. We give formulas in terms of projectors and in terms of matrix units. In terms of projectors, the formulas in all cases except for E$_8$ are not new. Formulas \eqref{R proj E6}, \eqref{R proj E7} can be found in \cite{M90}, \cite{BGZD94}, formula \eqref{R proj F4} in \cite{M91}, \cite{BGZD94}, formula \eqref{R proj G} in \cite{K90}. The corresponding rational formulas \eqref{RuE6}, \eqref{RuE7} can be found in \cite{M90}, formula \eqref{RuF4} in \cite{M91}, formula \eqref{RuG} in \cite{O86}.

To describe the $R$-matrix in terms of matrix units for exceptional types (we omit E$_8$ here) we will use the following universal formula. In fact the same formula could be used for classical types but we choose not to do that. 

We choose an orthonormal basis (with respect to properly normalized Shapovalov form) $v_i$ for $L_{\omega_1}$ labeled by numbers $i=1,\dots, d$, $d=\dim(L_{\omega_1})$, described in Section \ref{app}. Such a basis is easy to describe since all weight spaces are one-dimensional.  The only exception is the case of F$_4$ where we have a two dimensional zero weight space, which also can be handled, see \cite{DGZ94}. (Again, we do not give $R$-matrix in matrix unit form for E$_8$, though we do give such a basis for that case, see Section \ref{E8 app}.) 

Let $J=\{1,\dots,d\}$.
Given a highest weight $\lambda$ of a submodule in $L_{\omega_1}^{\otimes 2}$ we give a basis $w_s$ of $L_\lambda$, $s=1,\dots, \tilde d_\la$, $ \tl d_\lambda=\dim(L_{\lambda})$,  
of the form $w_s=\sum_{(i,j)\in I_s^\la} \sigma_{ij}^{q,s} v_i\otimes v_j$, where $I_s^\la\subset J\times J$. We list $I_s^\la$ and  $\sigma_{ij}^{q,s}$ in Section \ref{app}.
Importantly, the basis $w_s$ we choose is orthogonal and $w_s$ all have the same length with respect to the tensor product of Shapovalov forms in $L_{\omega_1}$. In addition, the different submodules in $L_{\omega_1}^{\otimes 2}$ are automatically orthogonal to each other, as $E_i^T=F_i$ for $i\in\mr I$, cf. Lemma \ref{lin alg lemma}.


Then many formulas for $R$-matrices in terms of matrix units have the following general form depending only on at most four coefficients $a_{\pm}(z)$ and $a_0^{(1)}(z)$, $a_0^{(2)}(z)$:
\eq{\label{fp}
G_\lambda(a_-,a_+,a_0^{(1)},a_0^{(2)}\,;\sigma)=\sum_{s=1}^{\tl{d}}\bigg(a_-(z)\sum_{<}\sigma_{ik}^{q,s}\sigma_{jl}^{q,s} E_{ij}\otimes E_{kl} +a_+(z)\sum_{>}\sigma_{ik}^{q,s}\sigma_{jl}^{q,s} E_{ij}\otimes E_{kl} \\ 
+\,a_0^{(1)}(z)\sum_{=,1}\sigma_{ik}^{q,s}\sigma_{jl}^{q,s} E_{ij}\otimes E_{kl} +\,a_0^{(2)}(z)\sum_{=,2}\sigma_{ik}^{q,s}\sigma_{jl}^{q,s} E_{ij}\otimes E_{kl}\bigg)\ ,
}
where for a given $s$ the sum is over pairs $(i,k),(j,l)\in I_s^\la$ such that $|(i,k)|+|(j,l)|$ is either smaller (in the $<$ sum), greater (in the $>$ sum) or equal (in the $=,1$ and $=,2$ sum) than $|I_s^\la|+1$. Here $|(i,k)|\in\{1,\dots,|I_s^\la|\}$ is the position of the pair $(i,k)$ in the list $I_s^\la$, and $|I_s^\la|$ is the cardinality of $I_s^\la$. 

The $=,1$ and $=,2$ sums in \eqref{fp}, corresponding to $|(i,k)|+|(j,l)|=|I_s^\la|+1$, are taken as follows. In the case of F$_4$ and G$_2$, when $\la=\om_0$, the $=,1$ sum is taken over those $(i,k)$ and $(j,l)$ for which both $v_i$, $v_k$ have weight $0$. The $=,2$ sum is taken over those $(i,k)$ and $(j,l)$ for which none of $v_i$, $v_k$ has weight $0$. In the case of F$_4$ and G$_2$, when $\la=\om_1$, the $=,1$ sum is taken over those $(i,k)$ and $(j,l)$ for which one of $v_i$, $v_k$ have weight $0$ or $v_i\otimes v_k$ has weight $0$. The $=,2$ sum is taken over those $(i,k)$ and $(j,l)$ for which none of $v_i$, $v_k$, $v_i\otimes v_k$ has weight $0$. 

In the case of F$_4$, when $\la=\om_4$, $a_0^{(1)}(z)=a_0^{(2)}(z)$ and the $=,1$ and $=,2$ sums combine to the sum over all $(i,k)$, $(j,l)$ such that $|(i,k)|+|(j,l)|=|I_s^\la|+1$. In addition, for $25\le s\le 28$, the $=,1$ and $=,2$ sums are absent. We write this as $G_{\om_4}(a_-,a_+,a_0\,;\sigma)$.

In the case of E$_6$ and E$_7$, there are no weight zero vectors in $L_{\omega_1}$ and the $=,1$ sum is declared empty.  Then we write $G_\la(a_-,a_+,a_0^{(1)},a_0^{(2)}\,;\sigma)$ as $G_\la(a_-,a_+,a_0\,;\sigma)$. In addition, in the case of E$_7$, when $\la=\omega_6$, the $=,2$ sum is absent for $64\le s\le 70$. We still write this as $G_{\om_6}(a_-,a_+,a_0\,;\sigma)$.

As always, $E_{ij}$ is the matrix unit corresponding to the chosen basis - a matrix of size $d\times d$ with $i,j$ entry $1$ and all other entries zero.

\subsection{Type E\texorpdfstring{$_6$}{2}}

The Dynkin diagram is:

\bigskip 

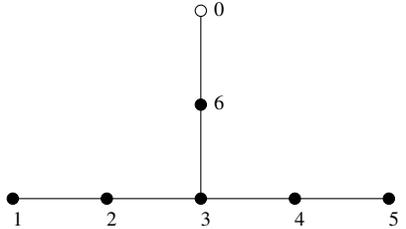
\begin{center}
\dynkin [extended, edge length=1.25cm, root radius=0.075cm, label macro/.code={\drlap{#1}}, labels={0, 1, 2, 3, 4, 5, 6}, ordering=Kac] E[1]{6}\ .
\end{center}

\medskip

The 27-dimensional $U_q\big($E$_6^{(1)}\big)$-module $\tl{L}_{1}(a)$ restricted to $U_q($E$_6)$ is isomorphic to $L_{\omega_1}$.\\ As $U_q($E$_6)$-modules we have
\eq{\label{tensorE6}
\underbracket[0.1ex]{L_{\omega_1}}_{27}\otimes \underbracket[0.1ex]{L_{\omega_1}}_{27}\cong \underbracket[0.1ex]{L_{2\omega_1}}_{351}\oplus \underbracket[0.1ex]{L_{\omega_2}}_{351}\oplus \underbracket[0.1ex]{L_{\omega_5}}_{27}\ .
}
In the $q\to 1$ limit, $L_{2\omega_1}\oplus L_{\omega_5}\mapsto \cl{S}^2(L_{\omega_1})$ and $L_{\omega_2}\mapsto \Lambda^2(L_{\omega_1})$.

The $q$-character of $\tl L_1=\tl{L}_{\mr{1}_0}$ has $27$ terms and there are no weight zero terms:
\eq{\label{qchar E6}
& \chi_q(\mr{1}_0)=\mr{1}_{0} + \ul{\mr{1}_{2}^{-1}\mr{2}_{1}} + \mr{2}_{3}^{-1}\mr{3}_{2} + \mr{3}_{4}^{-1}\mr{4}_{3}\mr{6}_{3} + \mr{4}_{3}\mr{6}_{5}^{-1} + \mr{4}_{5}^{-1}\mr{5}_{4}\mr{6}_{3} + \mr{3}_{4}\mr{4}_{5}^{-1}\mr{5}_{4}\mr{6}_{5}^{-1} + \mr{5}_{6}^{-1}\mr{6}_{3} + \mr{3}_{4}\mr{5}_{6}^{-1}\mr{6}_{5}^{-1} \\ 
& + \mr{2}_{5}\mr{3}_{6}^{-1}\mr{5}_{4} + \mr{2}_{5}\mr{3}_{6}^{-1}\mr{4}_{5}\mr{5}_{6}^{-1} + \mr{1}_{6}\mr{2}_{7}^{-1}\mr{5}_{4} + \mr{2}_{5}\mr{4}_{7}^{-1} + \mr{1}_{6}\mr{2}_{7}^{-1}\mr{4}_{5}\mr{5}_{6}^{-1} + \ul{\mr{1}_{8}^{-1}\mr{5}_{4}} + \mr{1}_{8}^{-1}\mr{4}_{5}\mr{5}_{6}^{-1} + \mr{1}_{6}\mr{2}_{7}^{-1}\mr{3}_{6}\mr{4}_{7}^{-1} \\ 
& + \mr{1}_{8}^{-1}\mr{3}_{6}\mr{4}_{7}^{-1} + \mr{1}_{6}\mr{3}_{8}^{-1}\mr{6}_{7} + \mr{1}_{6}\mr{6}_{9}^{-1} + \mr{1}_{8}^{-1}\mr{2}_{7}\mr{3}_{8}^{-1}\mr{6}_{7} + \mr{1}_{8}^{-1}\mr{2}_{7}\mr{6}_{9}^{-1} + \mr{2}_{9}^{-1}\mr{6}_{7} + \mr{2}_{9}^{-1}\mr{3}_{8}\mr{6}_{9}^{-1} + \mr{3}_{10}^{-1}\mr{4}_{9} + \mr{4}_{11}^{-1}\mr{5}_{10} + \mr{5}_{12}^{-1}\ .
}
Using the $q$-characters we compute the zeros and poles of $\check{R}(z)$ and the corresponding kernels and the cokernels.
\begin{lem}
The poles of the $R$-matrix $\check{R}(z)$, the corresponding submodules and quotient modules are given by
\begin{center}\begin{tabular}{c@{\hspace{1cm}} c c}
Poles & Submodules & Quotient modules  \\

$q^{2}$ & $\tl{L}_{\mr{1}_a\mr{1}_{aq^{-2}}}\cong L_{2\omega_1}$ & $\hspace{33pt} \tl{L}_{\mr{2}_{aq^{-1}}}\cong L_{\omega_2}\oplus L_{\omega_5}$ \\

$q^{8}$ & $\hspace{29pt} \tl{L}_{\mr{1}_a\mr{1}_{aq^{-8}}}\cong L_{2\omega_1}\oplus L_{\omega_2}$ & $\hspace{5pt} \tl{L}_{\mr{5}_{aq^{-4}}}\cong L_{\omega_5}$\\
\end{tabular} \ .\end{center}
\qed
\end{lem}

We choose a basis $\{v_i:1\le i\le 27\}$ for $L_{\omega_1}$ so that $v_1$ is a non-zero highest weight vector, see a diagram of $L_{\omega_1}$ in Section \ref{E6 app}. The vectors $v_i$ are ordered as their $\ell$-weights appear in the $q$-character \eqref{qchar E6}.

The $U_q($E$_6)$-submodule $L_{\omega_5}\se L_{\omega_1}^{\otimes 2}$ has a basis $\{u_s\}_{s=1}^{27}$ of the form 
$$u_s=\sum_{(i,j)\in I_s^{\omega_5}}\ve^{q,s}_{ij}v_i\otimes v_j,\qquad $$
where the sets $I_s^{\omega_5}$ are  given in Section \ref{E6 app} and have cardinality 10, and $\ve=\{\ve_{ij}^{q,s}\}_{s=1}^{27}$ are given by 
$$\ve^{q,s}_{ij}=(-q)^{5-|(i,j)|}\  \text{ for }i<j \text{ (or equivalently for }|(i,j)|\leq 5),\quad \ve^{q,s}_{ij}=\ve^{q^{-1},s}_{ji}\ \text{ for }i>j, \quad 1\le s\le 27\ .$$ 
We always have $i\ne j$ in this case. The vector $\ve$ will replace $\sigma$ in the expression of $G_{\omega_5}$ in \eqref{fp}, see \eqref{E6Tz}.

For $\lambda=2\omega_1,\omega_2,\omega_5$, let $P_\lambda^q$ be the projector onto the $U_q($E$_6)$-module $L_{\lambda}$ in the decomposition \eqref{tensorE6}.

\begin{thm}
In terms of projectors, we have 
\eq{\label{R proj E6}
\check{R}(z)=P_{2\omega_1}^q-q^{-2}\frac{1-q^{2}z}{1-q^{-2}z}\,P_{\omega_2}^q+q^{-10}\frac{(1-q^{2}z)(1-q^{8}z)}{(1-q^{-2}z)(1-q^{-8}z)}\,P_{\omega_5}^q\ .
}
In terms of matrix units, we have 
\eq{\label{RqE6}
\check{R}(z)=\big(\check{R}(z)\big)_{\fk{sl}_{27}}-\frac{(q-q^{-1})(1-z)}{(q-q^{-1}z)(q^4-q^{-4}z)}T(z)\ ,
}
where $\big(\check{R}(z)\big)_{\fk{sl}_{27}}$ is the A$_{26}$ (or $\fk{sl}_{27}$) trigonometric $R$-matrix in \eqref{RqA} and $T(z)$ is given by
\eq{\label{E6Tz}
T(z)=G_{\omega_5}\bigg(zq^{-4}, q^4, \frac{q^{\frac{7}{2}}+q^{-\frac{7}{2}}z}{q^{\frac{1}{2}}+q^{-\frac{1}{2}}}\,;\ve\bigg)\ .
}
\qed
\end{thm}

One can directly check that the $R$-matrix commutes with the action of $E_0$ and $F_0$, where 
$$K_0=\sum_{i\in\{1,2,3,4,6,8\}}(q^{-1}E_{ii}+qE_{\ol{i}\,\ol{i}})+\sum_{i,\ol{i}\notin\{1,2,3,4,6,8\}}E_{ii}\ ,\quad E_0(a)=a\,\big(E_{\ol{8}1}+E_{\ol{6}2}+E_{\ol{4}3}+E_{\ol{3}4}+E_{\ol{2}6}+E_{\ol{1}8}\big)\ ,$$
and $F_0(a)$ is the transpose of $a^{-2}E_0(a)$. Here $\ol{i}=28-i$.

Let $\displaystyle P_\lambda=\lim_{q\to1}P_\lambda^q$ be the  $U($E$_6)$ projector. For $(i,j)\in I_s^{\omega_5}$, let $\ve_{ij}^s=(-1)^{|(i,j)|}$ if $i<j$ and $\ve_{ij}^s=\ve_{ji}^s$ if $i>j$. Let $T$ be given by
$$T=\sum_{s=1}^{27}\ \sum_{(i,k),(j,l)\in I_s^{\omega_5}}\  \ve_{ik}^s\ve_{jl}^s\, E_{ij}\otimes E_{kl}=10\,P_{\omega_5}\ .$$

\begin{cor}
In the rational case, the corresponding $R$-matrix is given by:
\eq{\label{RuE6}
\check{R}(u) & =P_{2\omega_1}+\frac{1+u}{1-u}P_{\omega_2}+\frac{(1+u)(4+u)}{(1-u)(4-u)}P_{\omega_5} = \frac{1}{1-u}\bigg(I-uP+\frac{u}{4-u}T\bigg).
}
\end{cor}
\begin{proof}
We substitute $z=q^{2u}$ in \eqref{R proj E6} and \eqref{RqE6} and take limit $q\to 1$.
\end{proof}

\subsection{Type E\texorpdfstring{$_7$}{2}}

We consider the Dynkin diagram: 

\medskip 

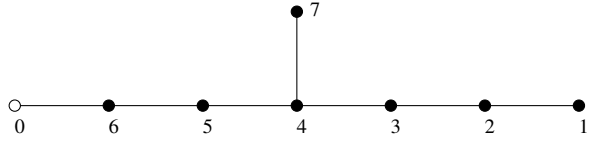
\begin{center}
\dynkin [extended, edge length=1.25cm, root radius=0.075cm, label macro/.code={\drlap{#1}}, labels={0, 6, 5, 4, 3, 2, 1, 7}, ordering=Kac] E[1]{7}\ .
\end{center}

\medskip 

The 56-dimensional $U_q\big($E$_7^{(1)}\big)$-module $\tl{L}_{1}(a)$ restricted to $U_q($E$_7)$ is isomorphic to $L_{\omega_1}$.\\ 
As $U_q($E$_7)$-modules we have
\eq{\label{tensorE7}
\underbracket[0.1ex]{L_{\omega_1}}_{56}\otimes \underbracket[0.1ex]{L_{\omega_1}}_{56}\cong \underbracket[0.1ex]{L_{2\omega_1}}_{1463}\oplus \underbracket[0.1ex]{L_{\omega_2}}_{1539}\oplus \underbracket[0.1ex]{L_{\omega_6}}_{133}\oplus \underbracket[0.1ex]{L_{\omega_0}}_{1}\ .
}
In the $q\to 1$ limit, $L_{2\omega_1}\oplus L_{\omega_6}\mapsto \cl{S}^2(L_{\omega_1})$ and $L_{\omega_2}\oplus L_{\omega_0}\mapsto \Lambda^2(L_{\omega_1})$.

The $q$-character of $\tl L_1=\tl{L}_{\mr{1}_0}$ has $56$ terms and there are no weight zero terms:
\eq{\label{qchar E7}
& \chi_q(\mr{1}_0) = \mr{1}_{0} + \Big(\, \ul{\mr{1}_{2}^{-1}\mr{2}_{1}} + \mr{2}_{3}^{-1}\mr{3}_{2} + \mr{3}_{4}^{-1}\mr{4}_{3} + \mr{4}_{5}^{-1}\mr{5}_{4}\mr{7}_{4} + \mr{5}_{4}\mr{7}_{6}^{-1} + \mr{5}_{6}^{-1}\mr{6}_{5}\mr{7}_{4} + \mr{4}_{5}\mr{5}_{6}^{-1}\mr{6}_{5}\mr{7}_{6}^{-1} + \mr{6}_{7}^{-1}\mr{7}_{4} \\
& + \mr{4}_{5}\mr{6}_{7}^{-1}\mr{7}_{6}^{-1} + \mr{3}_{6}\mr{4}_{7}^{-1}\mr{6}_{5} + \mr{3}_{6}\mr{4}_{7}^{-1}\mr{5}_{6}\mr{6}_{7}^{-1} + \mr{2}_{7}\mr{3}_{8}^{-1}\mr{6}_{5} + \mr{3}_{6}\mr{5}_{8}^{-1} + \mr{2}_{7}\mr{3}_{8}^{-1}\mr{5}_{6}\mr{6}_{7}^{-1} + \mr{1}_{8}\mr{2}_{9}^{-1}\mr{6}_{5} + \mr{1}_{8}\mr{2}_{9}^{-1}\mr{5}_{6}\mr{6}_{7}^{-1} \\
& + \mr{2}_{7}\mr{3}_{8}^{-1}\mr{4}_{7}\mr{5}_{8}^{-1} + \mr{1}_{8}\mr{2}_{9}^{-1}\mr{4}_{7}\mr{5}_{8}^{-1} + \mr{2}_{7}\mr{4}_{9}^{-1}\mr{7}_{8} + \mr{2}_{7}\mr{7}_{10}^{-1} + \mr{1}_{8}\mr{2}_{9}^{-1}\mr{3}_{8}\mr{4}_{9}^{-1}\mr{7}_{8} + \mr{1}_{8}\mr{2}_{9}^{-1}\mr{3}_{8}\mr{7}_{10}^{-1} + \mr{1}_{8}\mr{3}_{10}^{-1}\mr{7}_{8} + \mr{1}_{8}\mr{3}_{10}^{-1}\mr{4}_{9}\mr{7}_{10}^{-1} \\ 
& + \mr{1}_{8}\mr{4}_{11}^{-1}\mr{5}_{10} + \mr{1}_{8}\mr{5}_{12}^{-1}\mr{6}_{11} + \mr{1}_{8}\mr{6}_{13}^{-1} \Big) + \Big(\, \ul{\mr{1}_{10}^{-1}\mr{6}_{5}} + \mr{1}_{10}^{-1}\mr{5}_{6}\mr{6}_{7}^{-1} + \mr{1}_{10}^{-1}\mr{4}_{7}\mr{5}_{8}^{-1} + \mr{1}_{10}^{-1}\mr{3}_{8}\mr{4}_{9}^{-1}\mr{7}_{8} + \mr{1}_{10}^{-1}\mr{3}_{8}\mr{7}_{10}^{-1} \\
& + \mr{1}_{10}^{-1}\mr{2}_{9}\mr{3}_{10}^{-1}\mr{7}_{8} + \mr{1}_{10}^{-1}\mr{2}_{9}\mr{3}_{10}^{-1}\mr{4}_{9}\mr{7}_{10}^{-1} + \mr{2}_{11}^{-1}\mr{7}_{8} + \mr{2}_{11}^{-1}\mr{4}_{9}\mr{7}_{10}^{-1} + \mr{1}_{10}^{-1}\mr{2}_{9}\mr{4}_{11}^{-1}\mr{5}_{10} + \mr{2}_{11}^{-1}\mr{3}_{10}\mr{4}_{11}^{-1}\mr{5}_{10} + \mr{1}_{10}^{-1}\mr{2}_{9}\mr{5}_{12}^{-1}\mr{6}_{11} \\
& + \mr{1}_{10}^{-1}\mr{2}_{9}\mr{6}_{13}^{-1} + \mr{2}_{11}^{-1}\mr{3}_{10}\mr{5}_{12}^{-1}\mr{6}_{11} + \mr{3}_{12}^{-1}\mr{5}_{10} + \mr{2}_{11}^{-1}\mr{3}_{10}\mr{6}_{13}^{-1} + \mr{3}_{12}^{-1}\mr{4}_{11}\mr{5}_{12}^{-1}\mr{6}_{11} + \mr{3}_{12}^{-1}\mr{4}_{11}\mr{6}_{13}^{-1} + \mr{4}_{13}^{-1}\mr{6}_{11}\mr{7}_{12} \\
& + \mr{6}_{11}\mr{7}_{14}^{-1} + \mr{4}_{13}^{-1}\mr{5}_{12}\mr{6}_{13}^{-1}\mr{7}_{12} + \mr{5}_{12}\mr{6}_{13}^{-1}\mr{7}_{14}^{-1} + \mr{5}_{14}^{-1}\mr{7}_{12} + \mr{4}_{13}\mr{5}_{14}^{-1}\mr{7}_{14}^{-1} + \mr{3}_{14}\mr{4}_{15}^{-1} + \mr{2}_{15}\mr{3}_{16}^{-1} + \mr{1}_{16}\mr{2}_{17}^{-1} \Big) + \ul{\mr{1}_{18}^{-1}}.
}
Here we group the monomials according to the restriction of $U_q(\hat{\text{E}}_7)$-module $\tl L_{1}$ to $U_q(\hat{\text{E}}_6)$ subalgebra. On the level of $q$-characters, this restriction amounts to $\mr{1}_a\mapsto 1$, $\mr{i}_a\mapsto\mr{(i-1)}_a$, $2\le i\le 7$. Then the restriction of $\chi_q^{\text{E}_7}(\mr{1}_0)$ is $1+\chi_q^{\text{E}_6}(\mr{1}_1)+\chi_q^{\text{E}_6}(\mr{5}_5)+1$.

Using the $q$-characters we compute the zeros and poles of $\check{R}(z)$ and the corresponding kerenls and cokernels. 

\begin{lem}
The poles of the $R$-matrix $\check{R}(z)$, the corresponding submodules and quotient modules are given by
\begin{center}\begin{tabular}{c@{\hspace{0.5cm}} c c}
Poles & Submodules & \hspace{10pt} Quotient modules  \\

$q^{2}$ & $\tl{L}_{\mr{1}_a\mr{1}_{aq^{-2}}}\cong L_{2\omega_1}$ & $\hspace{56pt} \tl{L}_{\mr{2}_{aq^{-1}}}\cong L_{\omega_2}\oplus L_{\omega_6}\oplus L_{\omega_0}$ \\

$q^{10}$ & $\hspace{26pt} \tl{L}_{\mr{1}_a\mr{1}_{aq^{-10}}}\cong L_{2\omega_1}\oplus L_{\omega_2}$ & $\hspace{31pt} \tl{L}_{\mr{6}_{aq^{-5}}}\cong L_{\omega_6}\oplus L_{\omega_0}$ 
\vspace{2pt}\\

$q^{18}$ & $\hspace{55pt} \tl{L}_{\mr{1}_a\mr{1}_{aq^{-18}}}\cong L_{2\omega_1}\oplus L_{\omega_2}\oplus L_{\omega_6}$ & $\hspace{15pt} \tl{L}_{\scriptscriptstyle{1}}\cong L_{\omega_0}$ \\
\end{tabular}\ .\end{center}
\qed
\end{lem}

We choose a basis $\{v_i:1\le i\le 56\}$ for $L_{\omega_1}$ so that $v_1$ is a non-zero highest weight vector, see a diagram of $L_{\omega_1}$ in Section \ref{E7 app}. The vectors $v_i$ are ordered as their $\ell$-weights appear in the $q$-character \eqref{qchar E7}.

The $U_q($E$_7)$-submodule $L_{\omega_6}\se L_{\omega_1}^{\otimes 2}$ has a basis $\{u_s\}_{s=1}^{133}$ of the form 
$$u_s=\sum_{(i,j)\in I_s^{\omega_6}}\ve^{q,s}_{ij}v_i\otimes v_j,\qquad $$
where the sets $I_s^{\omega_6}$ are  given in Section \ref{E7 app} and have cardinality 56 for $64\le s\le 70$ and 12 otherwise, and $\ve=\{\ve_{ij}^{q,s}\}_{s=1}^{133}$ are given as follows for $1\le s\le 63$ or $71\le s\le 133$,
$$\ve^{q,s}_{ij}=-(-q)^{6-|(i,j)|}\  \text{ for }i<j \text{ (or equivalently for }|(i,j)|\leq 6),\quad \ve^{q,s}_{ij}=\ve^{q^{-1},s}_{ji}\ \text{ for }i>j\ ,$$
while for $64\le s\le 70$, $\ve_{ij}^{q,s}\in\C(q)$ are more complicated and are listed in Section \ref{E7 app}. We always have $i\ne j$ in this case. The vector $\ve$ will replace $\sigma$ in the expression of $G_{\omega_6}$ in \eqref{fp}, see \eqref{E7TQ}.

\comment{
For $i<j$ (or equivalently $|(i,j)|\le 6$), we set $\ve^{q,s}_{ij}=(-q)^{6-|(i,j)|}$ and for $i>j$, set $\ve^{q,s}_{ij}=\ve^{q^{-1},s}_{ji}$, $1\le s\le 63$ or $71\le s\le 133$. For $64\le s\le 70$, $\ve_{ij}^{q,s}$ are listed in Section \ref{E7 app}. The parities $\ve_{ij}^{q,s}$ correspond to $\sigma_{ij}^{q,s}$ which are to be used in the expression of $F_{\omega_6}$ in \eqref{fp}.
A singular vector $u_1\in L_{\omega_1}^{\otimes 2}$ having weight $\omega_6$ is given by $$\displaystyle u_1=\sum_{(i,j)\in I_1^{\omega_6}}\ve^{q,1}_{ij}v_i\otimes v_j\ ,$$
where $$I_1^{\omega_6}=\{(1,29),(2,16),(3,13),(4,11),(5,8),(6,7),(7,6),(8,5),(11,4),(13,3),(16,2),(29,1)\}\ .$$ 
We generate the $U_q($E$_7)$-submodule $L_{\omega_6}$ using this singular vector. Out of the resulting $133$ vectors of $L_{\omega_6}$, $126$ of nonzero weight, are all of the same form as $u_1$ but with different indices $(i,j)$ which are listed in sets $I_s^{\omega_6}$ (all of length $12$), $1\le s\le 63$ or $71\le s\le 133$, in Section \ref{E7 app}. 
The $7$ vectors of zero weight in $L_{\omega_6}$ are of the form 
$$\sum_{i=1}^{56}\ve_{i\,\ol{i}}^{q,s}\,v_i\otimes v_{\ol{i}}\ ,\quad 64\le s\le 70\ .$$
}

The $U_q($E$_7)$-submodule $L_{\omega_0}\se L_{\omega_1}^{\otimes 2}$ is one-dimensional with a singular vector $v_0\in L_{\omega_1}^{\otimes 2}$ of weight $\omega_0$ given by 
$$v_0=\sum_{(i,\ol{i})\in I^{\omega_0}}p_i^q\,v_i\otimes v_{\ol{i}}\ ,$$ 
where $I^{\omega_0}=\{(i,\ol{i}):1\le i\le 56\}$, $\ol{i}=57-i$ and $p_i^q=q^{k+\frac{1}{2}},\ k\in \Z$. The set $\{p_i^q:1\le i\le 28\}$ is given by
\eq{\label{E7 om0 par}\{q^{27/2}, -q^{25/2}, q^{23/2}, -q^{21/2}, q^{19/2}, -q^{17/2}, -q^{17/2}, q^{15/2}, q^{15/2}, -q^{13/2}, -q^{13/2}, q^{11/2}, q^{11/2}, -q^{9/2}, \\ -q^{9/2}, -q^{9/2}, q^{7/2}, q^{7/2}, -q^{5/2}, -q^{5/2}, q^{3/2}, q^{3/2}, -q^{1/2}, -q^{1/2}, q^{-1/2}, -q^{-3/2}, q^{-5/2}, -q^{-7/2}\}\ ,}
and $p_i^q=-p_{\ol{i}}^{q^{-1}}$, $29\le i\le 56$. The vector $p=\{p_i^q\}$ will replace $\sigma$ in the expression of $G_{\omega_0}$ in \eqref{fp}, see \eqref{E7TQ}.

For $\lambda=2\omega_1,\omega_2,\omega_6,\omega_0$, let $P_\lambda^q$ be the projector onto the $U_q($E$_7)$-module $L_{\lambda}$ in the decomposition \eqref{tensorE7}. 

\begin{thm}
In terms of projetors, we have 
\eq{\label{R proj E7}
\check{R}(z)=P_{2\omega_1}^q-q^{-2}\frac{1-q^{2}z}{1-q^{-2}z}\,P_{\omega_2}^q+q^{-12}\frac{(1-q^{2}z)(1-q^{10}z)}{(1-q^{-2}z)(1-q^{-10}z)}\,P_{\omega_6}^q-q^{-30}\frac{(1-q^{2}z)(1-q^{10}z)(1-q^{18}z)}{(1-q^{-2}z)(1-q^{-10}z)(1-q^{-18}z)}\,P_{\omega_0}^q\ .
}
In terms of matrix units, we have 
\eq{\label{RqE7}
\check{R}(z)=\big(\check{R}(z)\big)_{\fk{sl}_{56}}-\frac{(q-q^{-1})(1-z)}{(q-q^{-1}z)(q^5-q^{-5}z)}T(z)+\frac{(q-q^{-1})(1-z)}{(q-q^{-1}z)(q^5-q^{-5}z)(q^{9}-q^{-9}z)}Q(z)
}
where $\big(\check{R}(z)\big)_{\fk{sl}_{56}}$ is the A$_{55}$ (or $\fk{sl}_{56}$) trigonometric $R$-matrix in \eqref{RqA} and $T(z)$, $Q(z)$ are given by
\eq{\label{E7TQ}
T(z)=G_{\omega_6}\bigg(zq^{-5},q^5,\frac{q^{\frac{9}{2}}+q^{-\frac{9}{2}}z}{q^{\frac{1}{2}}+q^{-\frac{1}{2}}}\,;\ve\bigg)\ ,\quad Q(z)=G_{\omega_0}\bigg(z^2q^{-14}a_-(z),\ q^{14} a_+(z),\ q^{13}-q^{-13}z^2\,;p\bigg)\ .
}
Here $\displaystyle a_{\pm}(z)=\frac{1}{[2]\,[3]_3^{\mr{i}}}\Big(\mp q^{\mp 5}z^{\pm1}\big(q^{\pm 3}+q^{\pm1}-q^{\mp3}\big)\pm\big([2]_8+[2]_6-[3]\big)\Big)$.
\qed
\end{thm}

One can directly check that the $R$-matrix commutes with the action of $E_0$ and $F_0$, where 
$$K_0=\sum_{i=1}^6\big(q^{-1}E_{ii}+q^{-1}E_{i'i'}+qE_{\ol{i}\,\ol{i}}+qE_{\ol{i'}\,\ol{i'}}\,\big)+\sum_{i,\ol{i}\notin\{j,j':1\le j\le 6\}}E_{ii}\ ,\quad E_0(a)=a\,\sum_{i=1}^6\big(E_{\ol{i'}i}+E_{\ol{i}\,i'}\,\big)\ ,$$
and $F_0(a)$ is the transpose of $a^{-2}E_0(a)$. Here $1'=29,\ 2'=16,\ 3'=13,\ 4'=11,\ 5'=8,\ 6'=7$.

\medskip

Let $\displaystyle P_\lambda=\lim_{q\to1}P_\lambda^q$ be the  $U($E$_7)$ projector. For $(i,j)\in I_s^{\omega_6}$, let $\ve_{ij}^s$ be the $q\to 1$ limit of $\ve_{ij}^{q,s}$. For $1\le i\le 56$, let $p_i\in\{1,-1\}$ be the $q\to 1$ limit of $p_i^q$. Let $T$, $Q$ be given by
$$\displaystyle T=\sum_{s=1}^{133}\ \sum_{(i,k),(j,l)\in I_s^{\omega_6}}\  \ve_{ik}^s\ve_{jl}^s\, E_{ij}\otimes E_{kl}=12\,P_{\omega_6}\ , \quad Q=\frac{1}{2}\sum_{i,j=1}^{56}p_ip_jE_{ij}\otimes E_{\ol{i}\,\ol{j}}=28\,P_{\omega_0}\ .$$

\begin{cor}
In the rational case, the corresponding $R$-matrix is given by:
\eq{\label{RuE7}
\check{R}(u) &
= P_{2\omega_1}+\frac{1+u}{1-u}P_{\omega_2}+\frac{(1+u)(5+u)}{(1-u)(5-u)}P_{\omega_6}+\frac{(1+u)(5+u)(9+u)}{(1-u)(5-u)(9-u)}P_{\omega_0}\\ 
&=\frac{1}{1-u}\bigg(I-uP+\frac{u}{5-u}T+\frac{u(1+u)}{(5-u)(9-u)}Q\bigg).
}
\end{cor}
\begin{proof}
We substitute $z=q^{2u}$ in \eqref{R proj E7} and \eqref{RqE7} and take limit $q\to 1$.

\end{proof}

\subsection{Type F\texorpdfstring{$_4$}{2}}

We consider the Dynkin diagram: 

\medskip

\begin{center}
\dynkin [extended, edge length=1.25cm, root radius=0.075cm, label macro/.code={\drlap{#1}}, labels={0, 4, 3, 2, 1}, ordering=Kac] F[1]{4}\ .
\end{center}

\medskip

The $26$-dimensional $U_q\big($F$_4^{(1)}\big)$-module $\tl{L}_{1}(a)$ when restricted to $U_q($F$_4)$ is isomorphic to $L_{\omega_1}$.\\ 
As $U_q($F$_4)$-modules we have 
\eq{\label{tensorF4}
\underbracket[0.1ex]{L_{\omega_1}}_{26}\otimes \underbracket[0.1ex]{L_{\omega_1}}_{26}\cong \underbracket[0.1ex]{L_{2\omega_1}}_{324}\oplus \underbracket[0.1ex]{L_{\omega_2}}_{273}\oplus \underbracket[0.1ex]{L_{\omega_4}}_{52}\oplus \underbracket[0.1ex]{L_{\omega_1}}_{26}\oplus \underbracket[0.1ex]{L_{\omega_0}}_{1}\ .
}
In the $q\to1$ limit, $L_{2\omega_1}\oplus L_{\omega_1}\oplus L_{\omega_0}\mapsto \cl{S}^2(L_{\omega_1})$ and $L_{\omega_2}\oplus L_{\omega_4}\mapsto \Lambda^2(L_{\omega_1})$.

The $q$-character of $\tl L_1=\tl{L}_{\mr{1}_0}$ has $26$ terms and there are $2$ weight zero terms (shown in box):
\eq{\label{qchar F4}
& \chi_q(\mr{1}_0)= \mr{1}_{0} + \ul{\mr{1}_{2}^{-1}\mr{2}_{1}} + \mr{2}_{3}^{-1}\mr{3}_{2} + \mr{2}_{5}\mr{3}_{6}^{-1}\mr{4}_{4} + \mr{2}_{5}\mr{4}_{8}^{-1} + \mr{1}_{6}\mr{2}_{7}^{-1}\mr{4}_{4} + \ul{\mr{1}_{8}^{-1}\mr{4}_{4}} + \mr{1}_{6}\mr{2}_{7}^{-1}\mr{3}_{6}\mr{4}_{8}^{-1} + \mr{1}_{8}^{-1}\mr{3}_{6}\mr{4}_{8}^{-1} \\
& + \mr{1}_{6}\mr{2}_{9}\mr{3}_{10}^{-1} + \mr{1}_{8}^{-1}\mr{2}_{7}\mr{2}_{9}\mr{3}_{10}^{-1} + \mr{1}_{6}\mr{1}_{10}\mr{2}_{11}^{-1} + \boxed{\mr{1}_{8}^{-1}\mr{1}_{10}\mr{2}_{11}^{-1}\mr{2}_{7}} + \boxed{\ul{\mr{1}_{12}^{-1}\mr{1}_{6}}} + \mr{1}_{8}^{-1}\mr{1}_{12}^{-1}\mr{2}_{7} + \mr{1}_{10}\mr{2}_{9}^{-1}\mr{2}_{11}^{-1}\mr{3}_{8} + \mr{1}_{12}^{-1}\mr{2}_{9}^{-1}\mr{3}_{8} \\
& + \mr{1}_{10}\mr{3}_{12}^{-1}\mr{4}_{10} + \mr{1}_{12}^{-1}\mr{2}_{11}\mr{3}_{12}^{-1}\mr{4}_{10} + \mr{1}_{10}\mr{4}_{14}^{-1} + \mr{1}_{12}^{-1}\mr{2}_{11}\mr{4}_{14}^{-1} + \mr{2}_{13}^{-1}\mr{4}_{10} + \mr{2}_{13}^{-1}\mr{3}_{12}\mr{4}_{14}^{-1} + \mr{2}_{15}\mr{3}_{16}^{-1} + \mr{1}_{16}\mr{2}_{17}^{-1} + \ul{\mr{1}_{18}^{-1}}\ .
}
Using the $q$-characters we compute the zeros and poles of $\check{R}(z)$ and the corresponding kernels and cokernels. 

\begin{lem}
The poles of the $R$-matrix $\check{R}(z)$, the corresponding submodules and quotient modules are given by
\begin{center}\begin{tabular}{c@{\hspace{1cm}} c c}
Poles & Submodules & \hspace{10pt} Quotient modules  \\

$q^{2}$ & $\tl{L}_{\mr{1}_a\mr{1}_{aq^{-2}}}\cong L_{2\omega_1}\oplus L_{\omega_4}\oplus L_{\omega_0}$ & $\hspace{30pt} \tl{L}_{\mr{2}_{aq^{-1}}}\cong L_{\omega_2}\oplus L_{\omega_1}$ \\

$q^{8}$ & $\tl{L}_{\mr{1}_a\mr{1}_{aq^{-8}}}\cong L_{2\omega_1}\oplus L_{\omega_2}\oplus L_{\omega_1}$ & $\hspace{30pt} \tl{L}_{\mr{4}_{aq^{-4}}}\cong L_{\omega_4}\oplus L_{\omega_0}$ \\

$q^{12}$ & $\hspace{25pt} \tl{L}_{\mr{1}_a\mr{1}_{aq^{-12}}}\cong L_{2\omega_1}\oplus L_{\omega_2}\oplus L_{\omega_4}\oplus L_{\omega_0}$ & $\hspace{7pt} \tl{L}_{\mr{1}_{aq^{-6}}}\cong L_{\omega_1}$ 
\vspace{2pt}\\

$q^{18}$ & $\hspace{25pt} \tl{L}_{\mr{1}_a\mr{1}_{aq^{-18}}}\cong L_{2\omega_1}\oplus L_{\omega_2}\oplus L_{\omega_4}\oplus L_{\omega_1}$ & $\hspace{20pt} \tl{L}_{\scriptscriptstyle{1}}\cong L_{\omega_0}$\\
\end{tabular} \ .\end{center}
\qed
\end{lem}

We choose a basis $\{v_i:1\le i\le 26\}$ for $L_{\omega_1}$ so that $v_1$ is a non-zero highest weight vector, see a diagram of $L_{\omega_1}$ in Section \ref{F4 app}. The vectors $v_i$ are ordered as their $\ell$-weights appear in the $q$-character \eqref{qchar F4}.

The $U_q($F$_4)$-submodule $L_{\omega_4}\se L_{\omega_1}^{\otimes 2}$ has a basis $\{u_s\}_{s=1}^{52}$ of the form 
$$u_s=\sum_{(i,j)\in I_s^{\omega_4}}\ve^{q,s}_{ij}v_i\otimes v_j,\qquad $$
where the sets $I_s^{\omega_4}$ are  given in Section \ref{F4 app} and have cardinality 28 for $25\le s\le 28$, $12$ for $13\le s\le 24$ or $29\le s\le 40$ and 6 otherwise, and $\ve=\{\ve_{ij}^{q,s}\}_{s=1}^{52}$ are given as follows for $1\le s\le 12$ or $41\le s\le 52$,
$$\ve^{q,s}_{ij}=-(-q)^{4-|(i,j)|}\  \text{ for }i<j \text{ (or equivalently for }|(i,j)|\leq 3),\quad \ve^{q,s}_{ij}=\ve^{q^{-1},s}_{ji}\ \text{ for }i>j\ ,$$
while for $13\le s\le 40$, $\ve_{ij}^{q,s}\in\C(q)$ are more complicated and are listed in Section \ref{F4 app}. We have $i\ne j$ here except in the case of zero weight vectors. The vector $\ve$ will replace $\sigma$ in the expression of $G_{\omega_4}$ in \eqref{fp}, see \eqref{F4TSz}.

The $U_q($F$_4)$-submodule $L_{\omega_1}\se L_{\omega_1}^{\otimes 2}$ has a basis $\{w_s\}_{s=1}^{26}$ of the form 
$$w_s=\sum_{(i,j)\in I_s^{\omega_1}}\mu^{q,s}_{ij}v_i\otimes v_j,\qquad $$
where the sets $I_s^{\omega_1}$ are  given in Section \ref{F4 app} and have cardinality 28 for $13\le s\le 14$ and 12 otherwise, and $\mu=\{\mu_{ij}^{q,s}\}_{s=1}^{26}$ are given in Section \ref{F4 app}. The vector $\mu$ will replace $\sigma$ in the expression of $G_{\omega_1}$ in \eqref{fp}, see \eqref{F4TSz}.

\comment{
We list in the Section \ref{F4 app}, the sets $I_s^{\omega_1}$  $(1\le s\le 26)$, and the parities $\mu_{ij}^{q,s}\in\C(q)$, corresponding to $\sigma_{ij}^{q,s}$ which are to be used in the expression of $F_{\omega_1}$ in \eqref{fp}. 
A singular vector $w_1\in L_{\omega_1}^{\otimes 2}$ having weight $\omega_1$ is given by 
$$\displaystyle w_1=\sum_{(i,j)\in I_1^{\omega_1}}\mu^{q,1}_{ij}\, v_i\otimes v_j\ ,$$ 
where \bee{
I_1^{\omega_1}=\{(1,13),(1,14),(2,12),(3,10),(4,8),(5,6),(6,5),(8,4),(10,3),(12,2),(14,1),(13,1)\}\ ,\\ \mu^{q,1}_{ij}\in \frac{\sqrt{[3]}}{\sqrt{[2]}}\bigg\{0,q^6\frac{\sqrt{[2]}}{\sqrt{[3]}}, -q^{9/2}, q^{7/2}, -q^{3/2}, q^{-1/2}, q^{1/2}, -q^{-3/2}, q^{-7/2}, -q^{-9/2}, q^{-6}\frac{\sqrt{[2]}}{\sqrt{[3]}}, 0\bigg\}\ .
} 
}

The $U_q($F$_4)$-submodule $L_{\omega_0}\se L_{\omega_1}^{\otimes 2}$ is one-dimensional with a singular vector $v_0\in L_{\omega_1}^{\otimes 2}$ of weight $\omega_0$ given by 
$$v_0=\sum_{(i,j)\in I^{\omega_0}} p_{ij}^q\,v_i\otimes v_{j}\ .$$ where 
$I^{\omega_0}=\{(1,\ol{1}),\dots,(12,\ol{12}),(13,\ol{13}),(13,13), (14,14), (\ol{13},13),(\ol{12},12),\dots,(\ol{1},1)\}\ ,\ \ol{i}=27-i$, $1\le i\le 26$, and the set $\{p_{i\,\ol{i}}^q: 1\le i\le 13\}$ is given by 
\bee{\big\{q^{11}, -q^{10}, q^9, -q^7, q^5, q^6, -q^5, -q^4, q^3, q^2, -q, -q, 0\big\}\ ,}
and we have $p_{i\,\ol{i}}^q=p_{\ol{i}\,i}^{q^{-1}}$, for $14\le i\le 26$, $p_{13,13}^{q}=p_{14,14}^{q}=1$. The vector $p=\{p_{ij}^q\}$ will replace $\sigma$ in the expression of $G_{\omega_0}$ in \eqref{fp}, see \eqref{F4TSz}.

For $\lambda=2\omega_1,\omega_2,\omega_4,\omega_1,\omega_0$, let $P_\lambda^q$ be the projector onto the $U_q($F$_4)$-module $L_{\lambda}$ in the decomposition \eqref{tensorF4}.

\begin{thm}
In terms of projectors, we have 
\eq{\label{R proj F4}
\check{R}(z) = P_{2\omega_1}^q-q^{-2}\frac{1-q^{2}z}{1-q^{-2}z}\,P_{\omega_2}^q-q^{-8}\frac{1-q^{8}z}{1-q^{-8}z}P_{\omega_4}^q + q^{-14}\frac{(1-q^{2}z)(1-q^{12}z)}{(1-q^{-2}z)(1-q^{-12}z)}\,P_{\omega_1}^q \\
+ q^{-26}\frac{(1-q^{8}z)(1-q^{18}z)}{(1-q^{-8}z)(1-q^{-18}z)}\,P_{\omega_0}^q\ .
}
In terms of matrix units, we have
\eq{\label{RqF4}
\check{R}(z)=\ & \big(\check{R}(z)\big)_{\fk{sl}_{26}} + \frac{(q-q^{-1})(1-z)}{(q-q^{-1}z)(q^4-q^{-4}z)}T(z)-\frac{(q-q^{-1})(1-z)}{(q-q^{-1}z)(q^6-q^{-6}z)}S(z) \\
& -\frac{(q-q^{-1})(1-z)}{(q-q^{-1}z)(q^4-q^{-4}z)(q^9-q^{-9}z)}Q(z)-\frac{(q-q^{-1})(1-z)}{(q-q^{-1}z)}(E_{13,13}\otimes E_{13,13}-E_{14,14}\otimes E_{14,14})\ ,
}
where $\big(\check{R}(z)\big)_{\fk{sl}_{26}}$ is the A$_{25}$ (or $\fk{sl}_{26}$) trigonometric $R$-matrix in \eqref{RqA} and $T(z)$, $S(z)$, $Q(z)$ are given by
\eq{\label{F4TSz}
& T(z)=G_{\omega_4}\bigg(zq^{-4},q^4,\frac{q^{\frac{7}{2}}+q^{-\frac{7}{2}}z}{q^{\frac{1}{2}}+q^{-\frac{1}{2}}}\,;\ve\bigg)\ ,\\ 
& S(z)=G_{\omega_1}\bigg(zq^{-6},q^6,\frac{q^{\frac{11}{2}}+q^{-\frac{11}{2}}z}{q^{\frac{1}{2}}+q^{-\frac{1}{2}}},\frac{q^{\frac{13}{2}}+q^{\frac{9}{2}}+q^{\frac{7}{2}}+(q^{-\frac{7}{2}}+q^{-\frac{9}{2}}+q^{-\frac{13}{2}})z}{(q^2+1+q^{-2})\,(q^{\frac{1}{2}}+q^{-\frac{1}{2}})}\,;\mu\bigg)\ ,\\
& Q(z) = G_{\omega_0}\bigg(zq^{-12}a_-(z),q^{12}\,a_+(z), a_0(z) , (q^5-q^{-5}z)(q^6+q^{-6}z)\,;p\bigg)\ .
}

Here $\displaystyle a_{\pm}(z)=\frac{1}{[3]_3^\mr{i}}\big(q^9-q^{-9}z-q^{\mp3}[2]_4^\mr{i}\,(1+z)\big)$ and $\displaystyle a_0(z)=\frac{[3]^{\mr{i}}}{[2]_{\frac{1}{2}}}\big(q^{\frac{25}{2}}-z\,[2]_{\frac{7}{2}}^{\mr{i}}\,[3]_{\frac{1}{2}}-q^{-\frac{25}{2}}z^2\big).$
\qed
\end{thm}

One can directly check that the $R$-matrix commutes with the action of $E_0$ and $F_0$, where 
$$K_0=\sum_{i\in\{1,2,3,4,6,7\}}(q^{-2}E_{ii}+q^2E_{\ol{i}\,\ol{i}})+\sum_{i,\ol{i}\notin\{1,2,3,4,6,7\}}E_{ii}\ ,\quad E_0(a)=a\,\big(E_{\ol{7}1}+E_{\ol{6}2}+E_{\ol{4}3}+E_{\ol{3}4}+E_{\ol{2}6}+E_{\ol{1}7}\big)\ ,$$
and $F_0(a)$ is the transpose of $a^{-2}E_0(a)$.

Let $\displaystyle P_\lambda=\lim_{q\to1}P_\lambda^q$ be the  $U($F$_4)$ projector. For $(i,j)\in I_s^{\omega_4}$, let $\ve_{ij}^s$ be the $q\to1$ limit of $\ve_{ij}^{q,s}$. For $(i,j)\in I_s^{\omega_1}$, let $\mu_{ij}^s$ be the $q\to 1$ limit of $\mu_{ij}^{q,s}$. For $(i,j)\in I^{\omega_0}$, let $p_{ij}\in\{1,-1\}$ be the $q\to 1$ limit of $p_{ij}^q$. Let $T,S,Q$ be given by
$$\displaystyle T=\sum_{s=1}^{52}\sum_{(i,k),(j,l)\in I_s^{\omega_4}}\  \ve_{ik}^s\ve_{jl}^s\, E_{ij}\otimes E_{kl}=6\,P_{\omega_4}\ ,\quad
S=\sum_{s=1}^{26}\sum_{(i,k),(j,l)\in I_s^{\omega_1}}\mu_{ik}^s\,\mu_{jl}^s \,E_{ij}\otimes E_{kl}=14\,P_{\omega_1}\ ,$$
$$Q=\sum_{(i,k),(j,l)\in I^{\omega_0}}p_{ik} p_{jl}\,E_{ij}\otimes E_{kl}=26\,P_{\omega_0}\ .$$

\begin{cor}
In the rational case, the corresponding $R$-matrix is given by
\eq{\label{RuF4}
\check{R}(u) & =P_{2\omega_1}+\frac{1+u}{1-u}P_{\omega_2}+\frac{4+u}{4-u}P_{\omega_4}+\frac{(1+u)(6+u)}{(1-u)(6-u)}P_{\omega_1}+\frac{(4+u)(9+u)}{(4-u)(9-u)}P_{\omega_0} \\ 
& = \frac{1}{1-u}\bigg(I-uP-\frac{u}{4-u}T+\frac{u}{6-u}S+\frac{u(1-u)}{(4-u)(9-u)}Q\bigg).
}
\end{cor}

\begin{proof}
We substitute $z=q^{2u}$ in \eqref{R proj F4} and \eqref{RqF4} and take limit $q\to 1$.
\end{proof}

\subsection{Type G\texorpdfstring{$_2$}{2}}

We consider the Dynkin diagram: 

\medskip 

\begin{center}
\dynkin [extended, edge length=1.25cm, root radius=0.075cm, label macro/.code={\drlap{#1}}, labels={0, 2, 1}, ordering=Kac] G[1]{2}\ .
\end{center} 

\medskip

The $7$-dimensional $U_q\big($G$_2^{(1)}\big)$-module $\tl{L}_1(a)$ when restricted to $U_q($G$_2)$ is isomorphic to $L_{\omega_1}$.\\ As $U_q($G$_2)$-modules we have
\eq{\label{tensorG}
\underbracket[0.1ex]{L_{\omega_1}}_{7}\otimes \underbracket[0.1ex]{L_{\omega_1}}_{7}\cong \underbracket[0.1ex]{L_{2\omega_1}}_{27}\oplus \underbracket[0.1ex]{L_{\omega_2}}_{14}\oplus \underbracket[0.1ex]{L_{\omega_1}}_{7}\oplus \underbracket[0.1ex]{L_{\omega_0}}_{1}\ .}
In the $q\to1$ limit, $L_{2\omega_1}\oplus L_{\omega_0}\mapsto \cl{S}^2(L_{\omega_1})$ and $L_{\omega_2}\oplus L_{\omega_1}\mapsto \Lambda^2(L_{\omega_1})$.

The $q$-character of $\tl L_1=\tl{L}_{\mr{1}_0}$ has $7$ terms and there is $1$ weight zero term (shown in box):
\eq{\label{qchar G2}
\chi_q(\mr{1}_0)=\mr{1}_{0} + \ul{\mr{1}_{2}^{-1}\mr{2}_{1}} + \mr{1}_{4}\mr{1}_{6}\mr{2}_{7}^{-1} + \boxed{\ul{\mr{1}_{8}^{-1}\mr{1}_{4}}} + \mr{1}_{6}^{-1}\mr{1}_{8}^{-1}\mr{2}_{5} + \mr{1}_{10}\mr{2}_{11}^{-1} + \ul{\mr{1}_{12}^{-1}}\ .
}
Using the $q$-characters we compute the zeros and poles of $\check{R}(z)$ and the corresponding kernels and cokernels. 

\begin{lem}
The poles of the $R$-matrix $\check{R}(z)$, the corresponding submodules and quotient modules are given by
\begin{center}\begin{tabular}{c@{\hspace{1cm}} c c}
Poles & Submodules & \hspace{20pt} Quotient modules  \\

$q^{2}$ & $\tl{L}_{\mr{1}_a\mr{1}_{aq^{-2}}}\cong L_{2\omega_1}\oplus L_{\omega_1}$ & $\hspace{40pt} \tl{L}_{\mr{2}_{aq^{-1}}}\cong L_{\omega_2}\oplus L_{\omega_0}$ \\

$q^{8}$ & $\hspace{28pt} \tl{L}_{\mr{1}_a\mr{1}_{aq^{-8}}}\cong L_{2\omega_1}\oplus L_{\omega_2}\oplus L_{\omega_0}$ & $\hspace{16pt} \tl{L}_{\mr{1}_{aq^{-4}}}\cong L_{\omega_1}$ 
\vspace{2pt}\\

$q^{12}$ & $\hspace{26pt} \tl{L}_{\mr{1}_a\mr{1}_{aq^{-12}}}\cong L_{2\omega_1}\oplus L_{\omega_2}\oplus L_{\omega_1}$ & $\hspace{28pt} \tl{L}_{\scriptscriptstyle{1}}\cong L_{\omega_0}$ \\
\end{tabular}\ .\end{center}
\qed
\end{lem}

We choose a basis $\{v_i:1\le i\le 7\}$ for $L_{\omega_1}$ so that $v_1$ is a non-zero highest weight vector, see a diagram of $L_{\omega_1}$ in Section \ref{G2 app}. The vectors $v_i$ are ordered as their $\ell$-weights appear in the $q$-character \eqref{qchar G2}.

The $U_q($G$_2)$-submodule $L_{\omega_1}\se L_{\omega_1}^{\otimes 2}$ has a basis $\{w_s\}_{s=1}^{7}$ of the form 
$$w_s=\sum_{(i,j)\in I_s^{\omega_1}}\mu^{q,s}_{ij}v_i\otimes v_j,\qquad $$
where the sets $I_s^{\omega_1}$ are  given in Section \ref{G2 app} and have cardinality 7 for $s=4$ and cardinality 4 otherwise, and $\mu=\{\mu_{ij}^{q,s}\}_{s=1}^{7}$ are given in Section \ref{G2 app}. The vector $\mu$ will replace $\sigma$ in the expression of $G_{\omega_1}$ in \eqref{fp}, see \eqref{G2Sz}.

\comment{
We list in Section \ref{G2 app}, the sets $I_s^{\omega_1}\ (1\le s\le 7)\ $, and the parities $\mu_{ij}^{q,s} \in \C(q)$, corresponding to $\sigma_{ij}^{q,s}$ which are to be used in the expression of $F_{\omega_1}$ in \eqref{fp}. A singular vector $w_1\in L_{\omega_1}^{\otimes 2}$ having weight $\omega_1$ is given by 
$$w_1=\sum_{(i,j)\in I_1^{\omega_1}}\mu_{ij}^{q,1} v_i\otimes v_j\ ,\qquad I_1^{\omega_1}=\{(1,4), (2,3), (3,2), (4,1)\}\ ,\quad\mu_{ij}^{q,1}\in\{q^3,-q^{\frac{3}{2}}\sqrt{[2]}, q^{-\frac{3}{2}}\sqrt{[2]},-q^{-3}\}\ .$$
}

The $U_q($G$_2)$-submodule $L_{\omega_0}\se L_{\omega_1}^{\otimes 2}$ is one-dimensional with a singular vector $v_0\in L_{\omega_1}^{\otimes 2}$ of weight $\omega_0$ given by 
$$v_0=\sum_{(i,\ol{i})\in I^{\omega_0}}p_i^q v_i\otimes v_{\ol{i}}\ ,$$
where $I^{\omega_0}=\{(i,\ol{i}):1\le i\le 7\}$, $\ol{i}=8-i$ and the parities $\{p_i^q:1\le i\le 7\}$ are given by
$$\{q^5, -q^4, q, -1, q^{-1}, -q^{-4},q^{-5}\}\ .$$
The vector $p=\{p_i^q\}$ will replace $\sigma$ in the expression of $G_{\omega_0}$ in \eqref{fp}, see \eqref{G2Sz}.

For $\lambda=2\omega_1,\omega_2,\omega_1,\omega_0$, let $P_\lambda^q$ be the projector onto the $U_q($G$_2)$-module $L_{\lambda}$ in the decomposition \eqref{tensorG}.

\begin{thm}
In terms of projectors, we have 
\eq{\label{R proj G}
\check{R}(q,z)=P_{2\omega_1}^q-q^{-2}\frac{1-q^{2}z}{1-q^{-2}z}\,P_{\omega_2}^q-q^{-8}\frac{1-q^{8}z}{1-q^{-8}z}\,P_{\omega_1}^q+q^{-14}\frac{(1-q^{2}z)(1-q^{12}z)}{(1-q^{-2}z)(1-q^{-12}z)}\,P_{\omega_0}^q\ .}
In terms of matrix units, we have 
\eq{\label{RqG}
\check{R}(z)=\big(\check{R}(z)\big)_{\fk{sl}_{7}}+\frac{(q-q^{-1})(1-z)}{(q-q^{-1}z)(q^4-q^{-4}z)}S(z)-\frac{(q-q^{-1})(q^2+q^{-2})(1-z)}{(q-q^{-1}z)(q^6-q^{-6}z)}Q(z)\ ,
}
where $\big(\check{R}(z)\big)_{\fk{sl}_{7}}$ is the A$_6$ (or $\fk{sl}_{7}$) trigonometric $R$-matrix in \eqref{RqA} and $S(z)$, $Q(z)$ are given by
\eq{\label{G2Sz}
& S(z)=G_{\omega_1}\bigg(zq^{-4},q^4,\frac{q^{\frac{7}{2}}+q^{-\frac{7}{2}}z}{q^{\frac{1}{2}}+q^{-\frac{1}{2}}},\frac{q^3+q^{-3}z}{q+q^{-1}}\,;\mu\bigg)\ ,\\ & Q(z)=G_{\omega_0}\bigg(zq^{-6},q^6, \frac{q^7-q^5+q^4+(q^{-4}-q^{-5}+q^{-7})z}{q^2+q^{-2}}, \frac{q^4+q^{-4}z}{q^2+q^{-2}}\,;p\bigg)\ .
}
\qed
\end{thm}

One can directly check that the $R$-matrix commutes with the action of $E_0$ and $F_0$, where 
$$K_0=\sum_{i=1}^2\big(q^{-3}E_{ii}+q^3E_{\ol{i}\,\ol{i}}\big)+\sum_{i=3}^5 E_{ii}\ ,\quad E_0(a)=a\big(E_{\ol{1}1}+E_{\ol{2}2}\,\big)\ ,$$
and $F_0(a)$ is the transpose of $a^{-2}E_0(a)$.

Let $\displaystyle P_\lambda=\lim_{q\to1}P_\lambda^q$ be the  $U($G$_2)$ projectors. For $(i,j)\in I_s^{\omega_1}$, let $\mu_{ij}^s$ be the $q\to 1$ limit of $\mu_{ij}^{q,s}$, and let $S$, $Q$ be given by
\bee{S=\sum_{s=1}^7\sum_{(i,k),(j,l)\in I_s^{\omega_1}}\mu_{ik}^s\,\mu_{jl}^s\,E_{ij}\otimes E_{kl}=6\,P_{\omega_1}\ ,\quad Q=\sum_{i,j=1}^7(-1)^{i+j}\,E_{ij}\otimes E_{\ol{i}\,\ol{j}}=7\,P_{\omega_0}\ . 
} 

\begin{cor}
In the rational case, the corresponding $R$-matrix is given by \eq{\label{RuG}
\check{R}(u) & = P_{2\omega_1}+\frac{1+u}{1-u}P_{\omega_2}+\frac{4+u}{4-u}P_{\omega_1}+\frac{(1+u)(6+u)}{(1-u)(6-u)}P_{\omega_0} = \frac{1}{1-u}\bigg(I-uP-\frac{u}{4-u}S+\frac{2u}{6-u}Q\bigg)\ .
}
\end{cor}
\begin{proof}
We substitute $z=q^{2u}$ in \eqref{R proj G} and \eqref{RqG} and take limit $q\to 1$.
\end{proof}

\subsection{Type E\texorpdfstring{$_8$}{2}}
\label{E8 sec}

We consider the Dynkin diagram: 

\medskip

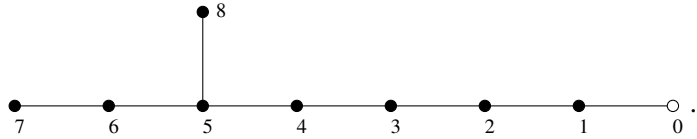
\begin{center}
\dynkin [extended, edge length=1.25cm, root radius=0.075cm, label macro/.code={\drlap{#1}}, labels={0, 1, 2, 3, 4, 5, 6, 7, 8}, ordering=Kac] E[1]{8}\ .
\end{center}

\medskip

The $249$-dimensional $U_q\big($E$_8^{(1)}\big)$-module $\tl{L}_1(a)$, when restricted to $U_q($E$_8)$, is isomorphic to $L_{\omega_1}\oplus L_{\omega_0}$.\\ As $U_q($E$_8)$-modules we have 
\eq{\label{tensorE8}
\big(\tl{L}_1(a)\big)^{\otimes 2}\cong\big(\underbracket[0.1ex]{L_{\omega_1}}_{248}\oplus \underbracket[0.1ex]{L_{\omega_0}}_{1}\big)^{\otimes 2}\cong \underbracket[0.1ex]{L_{2\omega_1}}_{27000}\oplus \underbracket[0.1ex]{L_{\omega_2}}_{30380}\oplus \underbracket[0.1ex]{L_{\omega_7}}_{3875}\oplus\,3 \underbracket[0.1ex]{L_{\omega_1}}_{248}\oplus\,2 \underbracket[0.1ex]{L_{\omega_0}}_{1}\ .
}
In the $q\to 1$ limit, $L_{2\omega_1}\oplus L_{\omega_7}\oplus L_{\omega_0}\mapsto \cl{S}^2(L_{\omega_1})$ and $L_{\omega_2}\oplus L_{\omega_1}\mapsto \Lambda^2(L_{\omega_1})$.

The $q$-character of $\tl L_1=\tl{L}_{\mr{1}_0}$ has $249$ terms with $9$ weight zero terms, and is given in Section \ref{E8 app}. Using the $q$-characters we compute the zeros and poles of the $R$-matrix.

\begin{lem}
The poles of the $R$-matrix $\check{R}(z)$, the corresponding submodules and quotient modules are given by
\begin{center}\begin{tabular}{c@{\hspace{1cm}} c c}
Poles & Submodules & Quotient modules  \\

$q^{2}$ & $\hspace{-50pt}\tl{L}_{\mr{1}_a\mr{1}_{aq^{-2}}}\cong L_{2\omega_1}\oplus L_{\omega_1}\oplus L_{\omega_0}$ & $\hspace{40pt}\tl{L}_{\mr{2}_{aq^{-1}}}\cong L_{\omega_2}\oplus L_{\omega_7}\oplus 2L_{\omega_1}\oplus L_{\omega_0}$ \\

$q^{12}$ & $\hspace{-19pt}\tl{L}_{\mr{1}_a\mr{1}_{aq^{-12}}}\cong L_{2\omega_1}\oplus L_{\omega_2}\oplus 2L_{\omega_1}\oplus L_{\omega_0}$ & $\hspace{8pt}\tl{L}_{\mr{7}_{aq^{-6}}}\cong L_{\omega_7}\oplus L_{\omega_1}\oplus L_{\omega_0}$ \\

$q^{20}$ & $\hspace{9pt}\tl{L}_{\mr{1}_a\mr{1}_{aq^{-20}}}\cong L_{2\omega_1}\oplus L_{\omega_2}\oplus L_{\omega_7}\oplus 2L_{\omega_1}\oplus L_{\omega_0}$ & $\hspace{-22pt}\tl{L}_{\mr{1}_{aq^{-10}}}\cong L_{\omega_1}\oplus L_{\omega_0}$ \\

$q^{30}$ & $\hspace{10pt}\tl{L}_{\mr{1}_a\mr{1}_{aq^{-30}}}\cong L_{2\omega_1}\oplus L_{\omega_2}\oplus L_{\omega_7}\oplus 3L_{\omega_1}\oplus L_{\omega_0}$ & $\hspace{-32pt}\tl{L}_{\scriptscriptstyle{1}}\cong L_{\omega_0}$\\
\end{tabular}\ .\end{center}
\qed
\end{lem}

We choose a basis $\{v_i:1\le i\le 248\}\cup\{v_{249}\}$ for $L_{\omega_1}\oplus L_{\omega_0}$, see Section \ref{E8 app}. In the chosen basis, the vectors $v_{121},\dots, v_{128},$ and $v_{249}$ are of weight zero.

A singular vector in $L_{\omega_1}^{\otimes 2}$ of weight $2\omega_1$, respectively $\omega_2$, is given by $v_1\otimes v_1$, respectively $q\,v_1\otimes v_2 - v_2\otimes v_1$. A singular vector in $L_{\omega_1}^{\otimes 2}$ of weight $\omega_7$ is given by 
\bee{
\sum_{(i,j)\in I^{\omega_7}}(-q)^{7-\min(i,j)}\,v_i\otimes v_j\ ,\quad I^{\omega_7}=\begin{matrix}\big\{(1,58), (2,30), (3,17), (4,14), (5,12), (6,9), (7,8), \\
(8,7), (9,6), (12,5), (14,4), (17,3), (30,2), (58,1)\big\}\end{matrix}\ .
}

For the last two summands in \eqref{tensorE8}, there is a natural choice of the three singular vectors $u_1\in L_{\omega_1}^{\otimes 2}$, $u_2\in L_{\omega_1}\otimes L_{\omega_0}$,  $u_3\in L_{\omega_0}\otimes L_{\omega_1}$ of weight $\omega_1$ and the two singular vectors $w_1\in L_{\omega_1}^{\otimes 2}$, $w_2\in L_{\omega_0}^{\otimes 2}$ of weight $\omega_0$. We choose $u_2$, $u_3$ to be $v_1\otimes v_{249}$ and $v_{249}\otimes v_1$ respectively, and $w_2$ to be $v_{249}\otimes v_{249}$. The singular vectors $u_1$ and $w_1$ are chosen such that the coordinate of $v_1\otimes v_{125}$ in $u_1$ is $q^{15}$, and that of $v_1\otimes v_{248}$ in $w_1$ is $q^{29}$. The vectors $u_1$, $u_2$, $u_3$, $w_1$, $w_2$ are all orthogonal to each other and their Shapovalov norms are given by 
\bee{
(u_1,u_1)=\frac{[2]_{16}\,[3]_3^{\mr{i}}\,[5]\,[15]}{[3]}\ ,\quad (w_1,w_1)=[2]_6\,[2]_{10}\,[2]_{12}\,[31]\ ,\quad (u_2,u_2)=(u_3,u_3)=(w_2,w_2)=1\ .
}

For $\la=2\omega_1,\omega_2,\omega_7,\omega_1,\omega_0$, let $P_\la^q$ be the projector onto the $U_q($E$_8)$-module $L_\la$ in the decomposition \eqref{tensorE8}.

\begin{thm}
\label{thm:R E8}
In terms of projectors, we have 
\eq{\label{RqE8}
\check{R}(z)=P_{2\omega_1}^q-q^{-2}\frac{1-q^{2}z}{1-q^{-2}z}\,P_{\omega_2}^q+q^{-14}\frac{(1-q^{2}z)(1-q^{12}z)}{(1-q^{-2}z)(1-q^{-12}z)}\,P_{\omega_7}^q+\frac{q^{-17}\,f_{\omega_1}(z)}{(1-q^{-2}z)(1-q^{-12}z)(1-q^{-20}z)}\otimes P_{\omega_1}^q \\
+\frac{q^{-32}\,f_{\omega_0}(z)}{(1-q^{-2}z)(1-q^{-12}z)(1-q^{-20}z)(1-q^{-30}z)}\otimes P_{\omega_0}^q\ ,
}
where the matrices $f_{\omega_1}(z)$ and $f_{\omega_0}(z)$ are given by
$$f_{\omega_1}(z)=\begin{bmatrix}-q^{-15}-q^{-6}\A_{q^{-1}}\,z+q^{6}\A_{q}\,z^2+q^{15}\,z^3 & \B_q\,z(1-z) & \B_q\,z(1-z)
\vspace{0.25cm}\\
\G_q\,z(1-z) & a_q\,z(q^{15}+q^{-15}z) & (1-z)(q^{15}-b_q\,z+q^{-15}z^2)
\vspace{0.25cm}\\
\G_q\,z(1-z) & (1-z)(q^{15}-b_q\,z+q^{-15}z^2) & a_q\,z(q^{15}+q^{-15}z)\end{bmatrix}\ ,$$
$$f_{\omega_0}(z)=\begin{bmatrix}q^{-30}-q^{-15}\,\zeta_q\,z+\xi_q\,z^2-q^{15}\,\zeta_q\,z^3+q^{30}z^4 & \eta_q\,z(1-z^2) \vspace{0.25cm}\\ \rho_q\,z(1-z^2) & q^{30}-q^{15}\,\zeta_q\,z+\xi_q\,z^2-q^{-15}\,\zeta_q\,z^3+q^{-30}z^4 \end{bmatrix}\ .$$

Here the constants $\A_q, \B_q, \G_q, a_q, b_q, \zeta_q, \xi_q, \eta_q, \rho_q\in \C(q)$ are given by
$$ \A_q=\frac{[2]_{19}^{\mr{i}}-[2]_{17}-[2]_{13}^{\mr{i}}-2q^{15}+q^{11}+q^9-q^{-1}}{[2]_{8}+[2]_{6}-[3]}\ ,\quad \B_q=\frac{[2]_{2}^{\mr{i}}\,[2]_3^{\mr{i}}}{[2]_{8}+[2]_{6}-[3]}\ ,\quad \G_q=\frac{[2]_2^{\mr{i}}\,[2]_{5}^{\mr{i}}\,[2]_{16}\,[3]_{3}^{\mr{i}}\,[15]}{[2]_{8}+[2]_{6}-[3]}\ ,$$
$$a_q=\frac{[2]_{2}^{\mr{i}}\,[2]_{3}^{\mr{i}}\,[2]_{5}^{\mr{i}}}{[2]_{8}+[2]_{6}-[3]}\ ,\quad b_q=\frac{[2]\,\big([2]_{12}-[7]^{\mr{i}}\big)}{[2]_{8}+[2]_{6}-[3]}\ ,\quad \xi_q=[2]_{32}-[2]_{30}+[2]_{18}+[2]_{10}+2\ , $$
$$ \zeta_q=\frac{[2]\,[2]_{16}\,[3]_{3}^{\mr{i}}}{[2]_{8}+[2]_{6}-[3]}\ ,\quad \eta_q=\frac{[2]_{2}^{\mr{i}}\,[2]_{3}^{\mr{i}}\,[2]_{5}^{\mr{i}}}{[2]_{8}+[2]_{6}-[3]}\ ,\quad \rho_q=\frac{[2]_{2}^{\mr{i}}\,[2]_{3}^{\mr{i}}\,[2]_{5}^{\mr{i}}\,[2]_{6}\,[2]_{10}\,[2]_{12}\,[31]}{[2]_{8}+[2]_{6}-[3]}\ .
$$
\end{thm}

We note that $[2]_8+[2]_6-[3]=\kappa_{60}(q)$ is the symmetric form of 60-th cyclotomic polynomial.

\begin{proof}
The rational functions corresponding to the first three summands in \eqref{tensorE8} are determined completely using $q$-characters. Let $g_1(z)$ and $g_2(z)$ be the $3\times 3$ and $2\times 2$ matrices corresponding to the last two summands respectively.

The $3\times 3$ matrix $g_1(z)$ is determined (up to a sign) as follows. Using Lemma \ref{R0}, we get 
\eq{\label{g10}
g_1(0)=\begin{bmatrix} -q^{-32} & 0 & 0\\ 0 & 0 & q^{-2} \\ 0 & q^{-2} & 0
\end{bmatrix}\ ,\quad g_1(\infty)=\begin{bmatrix} -q^{32} & 0 & 0\\ 0 & 0 & q^{2} \\ 0 & q^{2} & 0
\end{bmatrix}\ .
}
From $q$-characters we know the poles of $g_1(z)$. From Conjecture \ref{conj:simple poles}, we presume that the poles are simple. Combining this and \eqref{g10} with $g_1(1)$ being zero on off-diagonal entries and that $g_1(z)$ commutes with the flip operator acting on singular vectors, see Lemma \ref{R and flip commute}, we get $$g_1(z)=\frac{q^{-17}\,f_{\omega_1}(z)}{(1-q^{-2}z)(1-q^{-12}z)(1-q^{-20}z)}\ ,$$
where 
\bee{
f_{\omega_1}(z)=\begin{bmatrix}
-q^{-15}+\A_1 z+\A_2 z^2+q^{15}z^3 & \B z(1-z) & \B z(1-z) \\ 
\G z(1-z) & z(a_1+a_2z) & (1-z)(q^{15}+b z+q^{-15}z^2) \\
\G z(1-z) & (1-z)(q^{15}+b z+q^{-15} z^2) & z(a_1+a_2z)
\end{bmatrix}\ .
}
Since $g_1(1)$ is $1$ on the diagonal entries, we have 
\eq{\label{g1 at 1}
a_1+a_2=[2]^{\mr{i}}\,[2]_6^{\mr{i}}\,[2]_{10}^{\mr{i}}\ .
}
From $g_1(z)g_1(z^{-1})=\id$, we get
\eq{\label{a1 a2}
a_1=q^{30}a_2
}
and 
\eq{\label{A1 A2}
\A_1-a_2+b=q^{-15}\ ,\quad \A_2-a_1-b=-q^{15}\ .
}
The rank of $g_1(q^{-2})$ is $1$. This gives
\eq{\label{g1 at q2}
q\,a_1+q^{-1}\,a_2=[2]^{\mr{i}}\,\big(b+[2]_{17}\big)\ ,
}
and
\eq{\label{g1 at q2 bg}
\big([2]^{\mr{i}}\big)^2\,\B\G=(q\,a_1+q^{-1}\,a_2)\,\big(q\,\A_1+q^{-1}\,\A_2+[2]_{12}^{\mr{i}}\big)\ .
}
Now, using \eqref{g1 at 1} and \eqref{a1 a2} we get $a_1$ and $a_2$. Then \eqref{g1 at q2} gives $b$. Then $\A_1$ and $\A_2$ are obtained using \eqref{A1 A2}. Finally, the product $\B\G$ is obtained using \eqref{g1 at q2 bg}. From the choice of singular vectors $u_1\in L_{\omega
_1}^{\otimes2}$, $u_2\in L_{\omega_1}\otimes L_{\omega_0}$, we have 
\eq{\frac{\G}{\B}=\frac{(u_1,u_1)}{(u_2,u_2)}=\frac{[2]_{16}\,[3]_3^{\mr{i}}\,[5]\,[15]}{[3]}\ .}
Therefore, the matrix $f_{\omega_1}(z)$ is determined up to the sign of $\B$ (or $\G$).

\medskip

The $2\times 2$ matrix $g_2(z)$ is determined (up to a sign) as follows. Using Lemma \ref{R0}, we get
\eq{\label{g20}
g_2(0)=\begin{bmatrix}
q^{-62} & 0\\ 0 & q^{-2}
\end{bmatrix}\ ,\quad g_2(\infty)=\begin{bmatrix}
q^{62} & 0\\ 0 & q^2
\end{bmatrix}\ .
}
From $q$-characters we know the poles of $g_2(z)$. From Conjecture \ref{conj:simple poles} we presume that the poles are simple. Combining this and \eqref{g20} with $g_2(1)$ begin zero on off-diagonal entries we get $$g_2(z)=\frac{q^{-32}\,f_{\omega_0}(z)}{(1-q^{-2}z)(1-q^{-12}z)(1-q^{-20}z)(1-q^{-30}z)}\ ,$$
where 
\bee{
f_{\omega_0}(z)=\begin{bmatrix}
q^{-30}+\zeta_1 z+\xi_1 z^2+\zeta_2 z^3+q^{30}z^4 & z(1-z)(\eta_1+\eta_2 z) \\ z(1-z)(\rho_1+\rho_2 z) & q^{30} +\zeta_3 z+\xi_2 z^2 + \zeta_4 z^3 +q^{-30} z^4
\end{bmatrix}\ .
}
Using $g_2(z)g_2(z^{-1})=\id$, we get
\bee{
\zeta_1=\zeta_4\ ,\quad \zeta_2=\zeta_3\ ,\quad \xi_1=\xi_2\ ,\quad \eta_1=\eta_2\ ,\quad \rho_1=\rho_2\ ,
}
so that
\bee{
f_{\omega_0}(z)=\begin{bmatrix}
q^{-30}+\zeta_1 z+\xi z^2+\zeta_2 z^3+q^{30}z^4 & \eta\, z(1-z^2) \\ \rho\, z(1-z^2) & q^{30} +\zeta_2 z+\xi z^2 + \zeta_1 z^3 +q^{-30} z^4
\end{bmatrix}\ .
}
Since $g_2(1)$ is $1$ on the diagonal entries we have
\eq{\label{g2 at 1}
\zeta_1+\xi+\zeta_2+[2]_{30}=[2]^{\mr{i}}\,[2]^{\mr{i}}_6\,[2]^{\mr{i}}_{10}\,[2]^{\mr{i}}_{15}\ .
}
From $g_2(z)g_2(z^{-1})=\id$, now we get
\eq{\label{g2 inv a}
q^{30}\zeta_1+q^{-30}\zeta_2=-[2]_3\,[2]_5\,[2]_{16}\,[3]_3^{\mr{i}}\ ,
}
\eq{\label{g2 inv b}
q^{-30}\zeta_1+q^{30}\zeta_2+\xi(\zeta_1+\zeta_2)=-[2]_3\,[2]_5\,[2]_{16}\,[3]_3^{\mr{i}}\,\big([2]_{32}+[2]_{18}+[2]_{10}+1\big)\ ,
}
\eq{\label{g2 inv c}
\eta\rho=\zeta_1\zeta_2+\xi\,[2]_{30}-\big([2]_{50}+[2]_{42}+2[2]_{32}+[2]_{28}+[2]_{22}+2[2]_{18}+[2]_{14}+2[2]_{10}+[2]_8+4\big)\ .
}
Now, using \eqref{g2 at 1}, \eqref{g2 inv a} and \eqref{g2 inv b} we get two solutions for each of $\zeta_1$, $\zeta_2$ and $\xi$, out of which one is rejected because the $q\to 1$ limit of $g_2(z)$ does not exist in that case. After that we have a unique solution for $\zeta_1$, $\zeta_2$, $\xi$. Finally the product $\eta\,\rho$ is found using \eqref{g2 inv c}.
From the choice of singular vectors $w_1\in L_{\omega
_1}^{\otimes2}$, $w_2\in L_{\omega
_0}^{\otimes2}$, we have 
\eq{\frac{\rho}{\eta}=\frac{(w_1,w_1)}{(w_2,w_2)}=[2]_6\,[2]_{10}\,[2]_{12}\,[31]\ .}
Therefore, the matrix $f_{\omega_0}(z)$ is determined up to the sign of $\rho$ (or $\eta$).

To fix the signs of $\B$ in $f_{\omega_1}(z)$ and $\eta$ in $f_{\omega_0}(z)$, we use the $E_0$ action. Namely, to determine the sign of $\B$ we apply both sides of the commutation relation in \eqref{R and E0 commutation relation} to $v_1\otimes v_1$ and compare the coefficients of $v_1\otimes v_{249}$ on the two sides. To determine the sign of $\eta$ we apply both sides of \eqref{R and E0 commutation relation} to $v_1\otimes v_{249}$ and compare coefficients of $v_{249}\otimes v_{249}$ on the two sides.
\end{proof}

One can directly check that the $R$-matrix commutes with the action of $E_0$ and $F_0$, where 
$$K_0=q^{-2}E_{11}+q^2E_{\ol{1}\,\ol{1}}+\sum_{i=2}^{57}\big(q^{-1}E_{ii}+q E_{\ol{i}\,\ol{i}}\big) + \sum_{i=58}^{\ol{58}}E_{ii}+E_{249,249}\ ,$$
\bee{
E_0(a)= a\,\bigg( & \sum_{i=1}^4 \frac{(-1)^{i-1}}{\sqrt{[i][i+1]}} \big(E_{120+i,1} + E_{\ol{1},120+i}\,\big) + \frac{\sqrt{[2][3]}}{\sqrt{[5]([2]_8+[2]_6-[3])}} \big( E_{125,1} + E_{\ol{1},125} \big)  \\
& + \frac{[2]^{\mr{i}}\sqrt{[2][3][5]}}{\sqrt{[2]_8+[2]_6-[3]}} \big( E_{249,1} + E_{\ol{1},249} \big) + \sum_{i=2}^{57}E_{\ol{59-i},i}\bigg)
\ ,}
and $F_0(a)$ is the transpose of $a^{-2}E_0(a)$.
Here $\ol{i}=249-i$.

Let $\displaystyle P_\lambda=\lim_{q\to1}P_\lambda^q$ be the  $U($E$_8)$ projectors.

\begin{cor}
In the rational case, the corresponding $R$-matrix is given by
\eq{\label{RuE8}
\check{R}(u)=P_{2\omega_1}+\frac{1+u}{1-u}P_{\omega_2}+\frac{(1+u)(6+u)}{(1-u)(6-u)}P_{\omega_7}+\frac{f_{\omega_1}(u)}{(1-u)(6-u)(10-u)}\otimes P_{\omega_1} \\
+\frac{f_{\omega_0}(u)}{(1-u)(6-u)(10-u)(15-u)}\otimes P_{\omega_0}\ ,
}
where the matrices $f_{\omega_1}(u)$ and $f_{\omega_0}(u)$ are given by
$$f_{\omega_1}(u)=\begin{bmatrix} 60+44\,u+15\,u^2+u^3 & -6\,u & -6\,u
\vspace{0.25cm}\\
-300\,u & 60 & -u(4-u)(11-u)
\vspace{0.25cm}\\
-300\,u & -u(4-u)(11-u) & 60 \end{bmatrix}\ ,$$
$$f_{\omega_0}(u)=\begin{bmatrix} 900+660\,u+269\,u^2+30\,u^3+u^4 & -60\,u \vspace{0.25cm}\\ -14880\,u & 900-660\,u+269\,u^2-30\,u^3+u^4 \end{bmatrix}\ .$$
\end{cor}
\begin{proof}
We substitute $z=q^{2u}$ in \eqref{RqE8} and take limit $q\to 1$.
\end{proof}

\section{Other representations}
\label{mul cases}

\subsection{G\texorpdfstring{$_2$}{2} second fundamental representation}
\label{mul cases G2}

In this subsection, we write the $R$-matrix for the second fundamental module of G$_2$, obtained using fusion in \eqref{G2fusion}, in terms of projectors related to the tensor square decomposition.

As $U_q($G$_2)$-modules, we have 
\eq{\label{tensorG2}
\big(\tl{L}_2(a)\big)^{\otimes 2}\cong\big(\underbracket[0.1ex]{L_{\omega_2}}_{14}\oplus \underbracket[0.1ex]{L_{\omega_0}}_{1}\big)^{\otimes 2}\cong \underbracket[0.1ex]{L_{2\omega_2}}_{77}\oplus \underbracket[0.1ex]{L_{3\omega_1}}_{77}\oplus \underbracket[0.1ex]{L_{2\omega_1}}_{27}\oplus\,3\underbracket[0.1ex]{L_{\omega_2}}_{14}\oplus\,2\underbracket[0.1ex]{L_{\omega_0}}_{1}\ .
}

The $q$-character of $\tl L_2=\tl{L}_{\mr{2}_0}$ has $15$ terms with $3$ weight zero terms (shown in box): 
$$\chi_q(\mr{2}_0)=\mathrm{2}_{0} + \ul{\mathrm{1}_{1}\mathrm{1}_{3}\mathrm{1}_{5}\mathrm{2}_{6}^{-1}} + \mathrm{1}_{7}^{-1}\mathrm{1}_{1}\mathrm{1}_{3} + \mathrm{1}_{5}^{-1}\mathrm{1}_{7}^{-1}\mathrm{1}_{1}\mathrm{2}_{4} + \mathrm{1}_{3}^{-1}\mathrm{1}_{5}^{-1}\mathrm{1}_{7}^{-1}\mathrm{2}_{2}\mathrm{2}_{4} + \ul{\mathrm{1}_{1}\mathrm{1}_{9}\mathrm{2}_{10}^{-1}} + \boxed{\mathrm{1}_{11}^{-1}\mathrm{1}_{1}} + \boxed{\mathrm{1}_{3}^{-1}\mathrm{1}_{9}\mathrm{2}_{10}^{-1}\mathrm{2}_{2}}$$ 
$$+ \boxed{\ul{\mathrm{2}_{8}^{-1}\mathrm{2}_{4}}} + \mathrm{1}_{3}^{-1}\mathrm{1}_{11}^{-1}\mathrm{2}_{2} + \mathrm{1}_{5}\mathrm{1}_{7}\mathrm{1}_{9}\mathrm{2}_{8}^{-1}\mathrm{2}_{10}^{-1} + \mathrm{1}_{11}^{-1}\mathrm{1}_{5}\mathrm{1}_{7}\mathrm{2}_{8}^{-1} + \mathrm{1}_{9}^{-1}\mathrm{1}_{11}^{-1}\mathrm{1}_{5} + \mathrm{1}_{7}^{-1}\mathrm{1}_{9}^{-1}\mathrm{1}_{11}^{-1}\mathrm{2}_{6} + \ul{\mathrm{2}_{12}^{-1}}.$$
Using the $q$-characters we can find the zeros and poles of the $R$-matrix $\check{R}^{\tl{L}_2,\tl{L}_2}(z)$.
\begin{lem}
The poles of the $R$-matrix $\check{R}^{\tl{L}_2,\tl{L}_2}(z)$, the corresponding submodules and quotient modules are given by
\begin{center}\begin{tabular}{c@{\hspace{1cm}} c c}
Poles & Submodules & Quotient modules  \\

$q^{6}$ & $\hspace{-50pt} \tl{L}_{\mr{2}_a\mr{2}_{aq^{-6}}}\cong L_{2\omega_1}\oplus L_{\omega_2}\oplus L_{\omega_0}$ & $\hspace{7pt} \tl{L}_{\mr{1}_{aq^{-1}}\mr{1}_{aq^{-3}}\mr{1}_{aq^{-5}}}\cong L_{3\omega_1}\oplus L_{2\omega_1}\oplus 2L_{\omega_2}\oplus L_{\omega_0}$ \\

$q^{8}$ & $\hspace{20pt} \tl{L}_{\mr{2}_a\mr{2}_{aq^{-8}}}\cong L_{2\omega_2}\oplus L_{3\omega_1}\oplus L_{2\omega_1}\oplus 2L_{\omega_2}\oplus L_{\omega_0}$ & $\hspace{-25pt} \tl{L}_{\mr{2}_{aq^{-4}}}\cong L_{\omega_2}\oplus L_{\omega_0}$ \\

$q^{10}$ & $\hspace{-15pt} \tl{L}_{\mr{2}_a\mr{2}_{aq^{-10}}}\cong L_{2\omega_2}\oplus L_{3\omega_1}\oplus 2L_{\omega_2}\oplus L_{\omega_0}$ & $\hspace{-6pt} \tl{L}_{\mr{1}_{aq^{-1}}\mr{1}_{aq^{-9}}}\cong L_{2\omega_1}\oplus L_{\omega_2}\oplus L_{\omega_0}$ 
\vspace{2pt}\\

$q^{12}$ & $\hspace{18pt} \tl{L}_{\mr{2}_a\mr{2}_{aq^{-12}}}\cong L_{2\omega_2}\oplus L_{3\omega_1}\oplus L_{2\omega_1}\oplus 3L_{\omega_2}\oplus L_{\omega_0}$ & $\hspace{-35pt} \tl{L}_{\scriptscriptstyle{1}}\cong L_{\omega_0}$ \\
\end{tabular} \ .\end{center}
\qed
\end{lem}

For $\lambda=2\omega_2$, $3\omega_1$, $2\omega_1$, $\omega_2$, $\omega_0$, let $P_{\lambda}^q$ be the $U_q($G$_2)$ projector onto $L_{\lambda}$ in the decomposition \eqref{tensorG2}. 

\begin{thm}
In terms of projectors, we have
\eq{\label{RqG2fun2}
\check{R}^{\tl{L}_2,\tl{L}_2}(z)=\, & P_{2\omega_2}^q-q^{-6}\frac{1-q^{6}z}{1-q^{-6}z}\,P_{3\omega_1}^q+q^{-16}\frac{(1-q^{6}z)(1-q^{10}z)}{(1-q^{-6}z)(1-q^{-10}z)}\,P_{2\omega_1}^q\\ & +\frac{q^{-12}\,f_{\omega_2}(z)}{(1-q^{-6}z)(1-q^{-8}z)(1-q^{-10}z)}\otimes P_{\omega_2}^q +\frac{q^{-18}\,f_{\omega_0}(z)}{(1-q^{-6}z)(1-q^{-8}z)(1-q^{-10}z)(1-q^{-12}z)}\otimes P_{\omega_0}^q\,\,,
}
where the matrices $f_{\omega_1}(z)$ and $f_{\omega_0}(z)$ are given by
$$f_{\omega_2}(z)=\begin{bmatrix}-q^{-6}-q^{-4}\,\A_{q^{-1}}\,z+q^{4}\,\A_{q}\,z^2+q^{6}\,z^3 & \B_q\, z(1-z) & \B_q\, z(1-z) 
\vspace{0.25cm}\\ 
\G_q\, z(1-z) & a_q\,z(q^{6}+q^{-6}z) & (1-z)(q^{6}-b_q\,z+q^{-6}\,z^2)
\vspace{0.25cm}\\ 
\G_q\, z(1-z) & (1-z)(q^{6}-b_q\,z+q^{-6}\,z^2) & a_q\, z(q^{6}+q^{-6}z) \end{bmatrix}\ ,$$ 
$$f_{\omega_0}(z)=\begin{bmatrix}q^{-12}-q^{-6}\zeta_q\,z+\xi_q\,z^2-q^{6}\zeta_q\,z^3+q^{12}\,z^4 & \eta_q\,z(1-z^2)
\vspace{0.25cm}\\ 
\rho_q\,z(1-z^2) & q^{12}-q^{6}\zeta_q\,z+\xi_q\,z^2-q^{-6}\zeta_q\,z^3+q^{-12}\,z^4\end{bmatrix}\ .$$
Here the constants $\A_q$, $\B_q$, $\G_q$, $a_q$, $b_q$, $\zeta_q$, $\xi_q$, $\eta_q$, $\rho_q$ $\in$ $\C(q)$ are given by
$$\A_q=\frac{[3]\,\big([2]_{10}^\mr{i}-q^2[2]_6^\mr{i}-q^6\big)}{[3]_{2}^{\mr{i}}}\ ,\quad \B_q=\frac{[2]^{\mr{i}}\,[2]_{5}^{\mr{i}}}{[3]_{2}^{\mr{i}}}\ ,\quad \G_q=\frac{\big([2]\big)^2\,[2]_{9}\,[2]_{3}^{\mr{i}}\,[2]_{6}^{\mr{i}}}{[3]_{2}^{\mr{i}}}\ ,\quad a_q=\frac{[2]_{2}^{\mr{i}}\,[2]_{3}^{\mr{i}}\,[2]_{5}^{\mr{i}}}{[3]_{2}^{\mr{i}}}\ ,$$ 
$$b_q=\frac{[2]_{8}+[2]_{6}-[2]_{2}}{[3]_{2}^{\mr{i}}}\ ,\quad \zeta_q=\frac{[2]\,[2]_{9}}{[3]_{2}^{\mr{i}}}\,\,,\,\,\xi_q=[2]_{18}-[2]_{12}+[2]_4+[2]_2+2\ ,$$
$$\eta_q=\frac{[2]^{\mr{i}}\,\big([2]_{5}^{\mr{i}}\big)^2}{[3]_2^{\mr{i}}}\ ,\quad \rho_q=\frac{\big([2]\big)^2\,[2]_4\,[2]_{3}^{\mr{i}}\,[2]_{7}^{\mr{i}}\,\big([2]_{11}^{\mr{i}}-[2]_{9}^{\mr{i}}+[2]^{\mr{i}}\big)}{[3]_{2}^{\mr{i}}}\ .$$
\qed
\end{thm}

We note that $[3]_2^{\mr{i}}=\kappa_{24}(q)$ is the symmetric form of 24-th cyclotomic polynomial.

Let $P_\la=\lim_{q\to1}P_\la^q$ be the $U($G$_2)$ projectors.

\begin{cor}
In the rational case, the corresponding $R$-matrix is given by 
\eq{\label{RuG2fun2}
\check{R}^{\tl{L}_2,\tl{L}_2}(u)=P_{2\omega_2}+\frac{3+u}{3-u}P_{3\omega_1}^q+\frac{(3+u)(5+u)}{(3-u)(5-u)}P_{2\omega_1}+\frac{f_{\omega_2}(u)}{(3-u)(4-u)(5-u)}\otimes P_{\omega_2} \\
+\frac{f_{\omega_0}(u)}{(3-u)(4-u)(5-u)(6-u)}\otimes P_{\omega_0}\ ,
}
where the matrices $f_{\omega_2}(u)$, $f_{\omega_0}(u)$ are given by 
$$f_{\omega_2}(u)=\begin{bmatrix} 60-7\,u+6\,u^2+u^3 & -5\,u & -5\,u
\vspace{0.25cm}\\
-144\,u & 60 & -u(1+u)(7-u)
\vspace{0.25cm}\\
-144\,u & -u(1+u)(7-u) & 60 \end{bmatrix}\ ,$$
$$f_{\omega_0}(u)=\begin{bmatrix} 360-42\,u+29\,u^2+12\,u^3+u^4 & -150\,u \vspace{0.25cm}\\ -1008\,u & 360+42\,u+29\,u^2-12\,u^3+u^4 \end{bmatrix}\ .$$
\end{cor}
\begin{proof}
We substitute $z=q^{2u}$ in \eqref{RqG2fun2} and take limit $q\to 1$.
\end{proof}

\subsection{A\texorpdfstring{$_2$}{2} adjoint evaluation representation}
\label{mul cases A2}

In this subsection, we write the $R$-matrix for the evaluation adjoint representation of A$_2$, obtained using fusion in \eqref{A2fusion}, in terms of projectors related to the tensor square decomposition.

As $U_q($A$_2)$-modules, we have 
\eq{\label{tensorA2ad}
\big(\underbracket[0.1ex]{L_{\omega_1+\omega_2}}_{8}\big)^{\otimes 2}\cong \underbracket[0.1ex]{L_{2\omega_1+2\omega_2}}_{27}\oplus \underbracket[0.1ex]{L_{3\omega_1}}_{10}\oplus \underbracket[0.1ex]{L_{3\omega_2}}_{10}\oplus\, 2\,\underbracket[0.1ex]{L_{\omega_1+\omega_2}}_{8}\oplus\,\underbracket[0.1ex]{L_{\omega_0}}_{1}\ .
}

For $\lambda=2\omega_1+2\omega_2, 3\omega_1, 3\omega_2, \omega_1+\omega_2, \omega_0$, let $P_\la^q$ be the $U_q($A$_2)$-projector onto $L_\la$ in the decomposition \eqref{tensorA2ad}. 

\begin{thm}\label{thm:RqA2ad}
In terms of projectors, we have 
\eq{\label{RqA2ad}
\check{R}^{\text{ad},\text{ad}}(z)=P_{2\omega_1+2\omega_2}^q-q^{-2}\frac{1-q^{2}z}{1-q^{-2}z}\,\big(P_{3\omega_1}^q+P_{3\omega_2}^q\big)-\frac{q^{-5}\,f(z)}{(1-q^{-2}z)^2(1-q^{-6}z)}\otimes P_{\omega_1+\omega_2}^q \\ 
+\,q^{-8}\frac{(1-q^{2}z)(1-q^{6}z)}{(1-q^{-2}z)(1-q^{-6}z)}\,P_{\omega_0}^q\ ,
}
where the matrix $f(z)$ is given by $$f(z)=\begin{bmatrix}[3]\,\big([2]^{\mr{i}}\big)^2\,z(q^{-1}z-q) & (z-1)(qz-q^{-1})(q^{-1}z-q)
\vspace{0.25cm}\\ (z-1)\big(z^2-([2]_{6}+[2]_{2}-2)z+1\big) & -[3]\,\big([2]^{\mr{i}}\big)^2\,z(qz-q^{-1})\end{bmatrix}\ .$$ 
\qed
\end{thm}

In this case the $q$-characters are not sufficient to write the $R$-matrix. First, we lack Theorem \ref{poles thm} identifying the zeroes and poles of the $R$-matrix with the submodules and quotient modules. Second, some submodules and quotient modules are indecomposable, and we have a double pole of the $R$-matrix.

The $q$-character of $\tl{L}_{\mr{1}_0\mr{2}_3}$ has $8$ terms out of which $2$ are of weight zero (shown in box):
$$\chi_q(\mr{1}_0\mr{2}_3)=\mathrm{1}_{0}\mathrm{2}_{3} + \mathrm{1}_{2}^{-1}\mathrm{2}_{1}\mathrm{2}_{3} + \mathrm{1}_{0}\mathrm{1}_{4}\mathrm{2}_{5}^{-1} + \boxed{\mathrm{1}_{6}^{-1}\mathrm{1}_{0} + \mathrm{1}_{2}^{-1}\mathrm{1}_{4}\mathrm{2}_{5}^{-1}\mathrm{2}_{1}} + \mathrm{1}_{2}^{-1}\mathrm{1}_{6}^{-1}\mathrm{2}_{1} + \mathrm{1}_{4}\mathrm{2}_{3}^{-1}\mathrm{2}_{5}^{-1} + \mathrm{1}_{6}^{-1}\mathrm{2}_{3}^{-1}\ .$$

Then we compute the decomposition
$$
\chi_q(\mr{1}_0\mr{2}_3)\chi_q(\mr{1}_2\mr{2}_5)=\chi_q(\mr{1}_0\mr{1}_2\mr{2}_3\mr{2}_5)+\chi_q(\mr{1}_0\mr{1}_2\mr{1}_4)+\chi_q(\mr{2}_1\mr{2}_3\mr{2}_5)+\chi_q(\mr{1}_4\mr{2}_1)+1\ ,
$$
and as $U_q($A$_2)$-modules this corresponds to
$$L_{\omega_1+\omega_2}^{\otimes2}\cong \big(L_{2\omega_1+2\omega_2}\oplus L_{\omega_1+\omega_2}\big) \oplus L_{3\omega_1}\oplus L_{3\omega_2}\oplus L_{\omega_1+\omega_2}\oplus L_{\omega_0}\ .$$

As we see from Theorem \ref{thm:RqA2ad}, the last four summands correspond to poles of the $R$-matrix at $z=q^2$ and  one $L_{\omega_1+\omega_2}$ is a double pole. This means there is a submodule which contains all modules except for this  
$L_{\omega_1+\omega_2}$. We do not expect this submodule to be a direct some of all four summands.

\medskip 

Let $P_\la=\lim_{q\to1}P_\la^q$ be the $U($A$_2)$-projectors. 
\begin{cor}
In the rational case, the corresponding $R$-matrix is given by 
\eq{\label{RuA2ad}
\check{R}^{\text{ad},\text{ad}}(u)=P_{2\omega_1+2\omega_2}+\frac{1+u}{1-u}\big(P_{3\omega_1}+P_{3\omega_2}\big)+\frac{f(u)}{(1-u)^2(3-u)}\otimes P_{\omega_1+\omega_2}+\frac{(1+u)(3+u)}{(1-u)(3-u)}P_{\omega_0}\ ,
}
where the matrix $f(u)$ is given by 
\bee{
f(u)=\begin{bmatrix}
3(1-u) & u(1-u^2) \\
u(10-u^2) & 3(1+u)
\end{bmatrix}\ .
}
\p{
We substitute $z=q^{2u}$ in \eqref{RqA2ad} and take limit $q\to 1$.
}
\end{cor}


\section{Appendices}\label{app}

\subsection{Type E\texorpdfstring{$_6$}{2}}
\label{E6 app}

A diagram of the first fundamental module $L_{\omega_1}$ is shown in Figure \ref{fig:E6 first fundamental}.
\begin{figure}
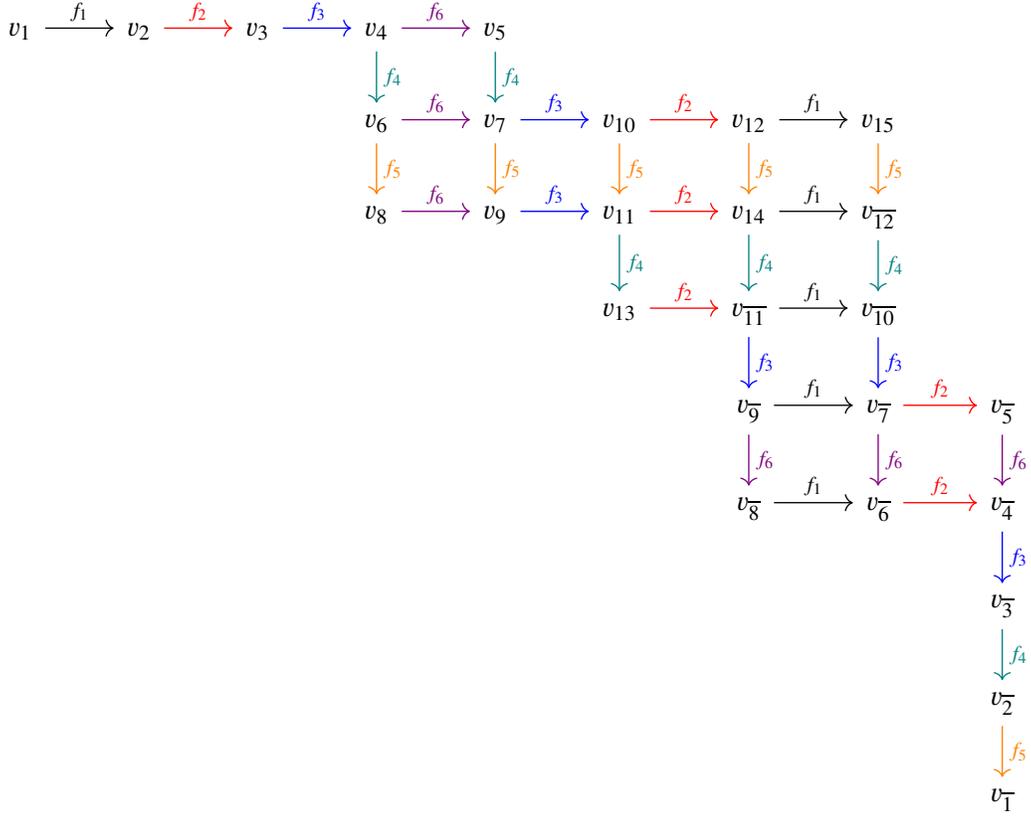

    \centering
\ddia{arrows={line width=0.5pt}}{v_1\ar[r,"f_1"] \& v_2\ar[r,red,"f_2"] \& v_3\ar[r,blue,"f_3"] \& v_4\ar[r,violet,"f_6"]\ar[d,teal,"f_4"] \& v_5\ar[d,teal,"f_4"] \& \& \& \& \\ \& \& \& v_6\ar[r,violet,"f_6"]\ar[d,orange,"f_5"] \& v_7\ar[d,orange,"f_5"]\ar[r,blue,"f_3"] \& v_{10}\ar[d,orange,"f_5"]\ar[r,red,"f_2"] \& v_{12}\ar[d,orange,"f_5"]\ar[r,"f_1"] \& v_{15}\ar[d,orange,"f_5"] \& \\ \& \& \& v_8\ar[r,violet,"f_6"] \& v_9\ar[r,blue,"f_3"] \& v_{11}\ar[r,red,"f_2"]\ar[d,teal,"f_4"] \& v_{14}\ar[r,"f_1"]\ar[d,teal,"f_4"] \& v_{\ol{12}}\ar[d,teal,"f_4"] \& \\ \& \& \& \& \& v_{13}\ar[r,red,"f_2"] \& v_{\ol{11}}\ar[r,"f_1"]\ar[d,blue,"f_3"] \& v_{\ol{10}}\ar[d,blue,"f_3"] \& \\ \& \& \& \& \& \& v_{\ol{9}}\ar[r,"f_1"]\ar[d,violet,"f_6"] \& v_{\ol{7}}\ar[d,violet,"f_6"]\ar[r,red,"f_2"] \& v_{\ol{5}}\ar[d,violet,"f_6"] \\ \& \& \& \& \& \& v_{\ol{8}}\ar[r,"f_1"] \& v_{\ol{6}}\ar[r,red,"f_2"] \& v_{\ol{4}}\ar[d,blue,"f_3"] \\ \& \& \& \& \& \& \& \& v_{\ol{3}}\ar[d,teal,"f_4"] \\ \& \& \& \& \& \& \& \& v_{\ol{2}}\ar[d,orange,"f_5"] \\ \& \& \& \& \& \& \& \& v_{\ol{1}}}
    \caption{First fundamental module for E$_6$.}
    \label{fig:E6 first fundamental}
\end{figure}
Here $v_j$ are ordered as their $\ell$-weights appear in the $q$-character of $\tl L_{\mr{1}_0}$ in \eqref{qchar E6}, and $\ol{i}=28-i$, $1\leq i\leq 12$.
The coefficients of all the arrows are one. The action of $f_i$'s is indicated in this diagram. The action of $e_i$'s is obtained by reversing all the arrows and keeping the same coefficient on each arrow.

The submodule $L_{\omega_5}$ forms a similar diagram as above but by switching $f_1\oto \textcolor{orange}{f_5}$, $\textcolor{red}{f_2}\oto \textcolor{teal}{f_4}$. We choose a basis $\{u_s\}_{s=1}^{27}$ for $L_{\omega_5}\se L_{\omega_1}^{\otimes2}$. The basis vectors are of the form
$$u_s=\sum_{(i,j)\in I_s^{\omega_5}}\ve_{ij}^{q,s}v_i\otimes v_j\ ,\quad 1\le s\le 27\ .$$

The sets $I_s^{\omega_5}$ and coordinates $\ve_{ij}^{q,s}$, $1\le s\le 27$, are used in the expression of $T(z)$ in $\eqref{E6Tz}$. The sets $I_s^{\omega_5}$ have cardinality 10 and the property $(i,j)\in I_s^{\omega_5}$ if and only if $(j,i)\in I_s^{\omega_5}$. The element $(j,i)$, $i<j$, is placed in $I_s^{\omega_5}$ such that the positions $|(i,j)|$ and $|(j,i)|$ of $(i,j)$ and $(j,i)$ respectively, satisfy $|(i,j)|+|(j,i)|=11$. We always have $i\ne j$ in this case. Therefore, we list only the 5 element subsets of $I_s^{\omega_5}$ for which $i<j$. See Figure \ref{fig:E6I}.

\begin{figure}
\bee{
& \scriptstyle\big\{(1,15),(2,12),(3,10),(4,7),(5,6)\big\},\,
\big\{(1,16),(2,14),(3,11),(4,9),(5,8)\big\},\,
\big\{(1,18),(2,17),(3,13),(6,9),(7,8)\big\},\, \\
& \scriptstyle \big\{(1,21),(2,19),(4,13),(6,11),(8,10)\big\},
\big\{(1,22),(2,20),(5,13),(7,11),(9,10)\big\},\,
\big\{(1,23),(3,19),(4,17),(6,14),(8,12)\big\},\, \\
& \scriptstyle \big\{(1,24),(3,20),(5,17),(7,14),(9,12)\big\},\,
\big\{(2,23),(3,21),(4,18),(6,16),(8,15)\big\},
\big\{(2,24),(3,22),(5,18),(7,16),(9,15)\big\},\, \\
& \scriptstyle \big\{(1,25),(4,20),(5,19),(10,14),(11,12)\big\},\,
\big\{(2,25),(4,22),(5,21),(10,16),(11,15)\big\},\,
\big\{(1,26),(6,20),(7,19),(10,17),(12,13)\big\}, \\
& \scriptstyle \big\{(3,25),(4,24),(5,23),(12,16),(14,15)\big\},\,
\big\{(2,26),(6,22),(7,21),(10,18),(13,15)\big\},\,
\big\{(1,27),(8,20),(9,19),(11,17),(13,14)\big\},\, \\ 
& \scriptstyle \big\{(2,27),(8,22),(9,21),(11,18),(13,16)\big\},
\big\{(3,26),(6,24),(7,23),(12,18),(15,17)\big\},\,
\big\{(3,27),(8,24),(9,23),(14,18),(17,16)\big\},\,\\
& \scriptstyle \big\{(4,26),(6,25),(10,23),(12,21),(15,19)\big\},\,
\big\{(5,26),(7,25),(10,24),(12,22),(15,20)\big\},
\big\{(4,27),(8,25),(11,23),(14,21),(16,19)\big\},\, \\
& \scriptstyle \big\{(5,27),(9,25),(11,24),(14,22),(16,20)\big\},\,
\big\{(6,27),(8,26),(13,23),(17,21),(19,18)\big\},\,
\big\{(7,27),(9,26),(13,24),(17,22),(18,20)\big\}, \\
& \scriptstyle\big\{(10,27),(11,26),(13,25),(19,22),(20,21)\big\},\,
\big\{(12,27),(14,26),(17,25),(19,24),(20,23)\big\},\,
\big\{(15,27),(16,26),(18,25),(21,24),(22,23)\big\}.
}
\caption{The index sets $I_s^{\omega_5}$ ($i<j$), $1\le s\le 27$, for E$_6$.}
\label{fig:E6I}
\end{figure}
The corresponding coordinates $\ve_{ij}^{q,s}$, $1\le s\le 27$, are given by $\ve_{ij}^{q,s}=(-q)^{5-|(i,j)|}$ if $i<j$, and $\ve_{ij}^{q,s}=\ve_{ji}^{q^{-1},s}$ if $i>j$. 

\subsection{Type E\texorpdfstring{$_7$}{2}}
\label{E7 app}

A diagram of the first fundamental module $L_{\omega_1}$ is shown in Figure \ref{fig:E7 first fundamental}.
\begin{figure}
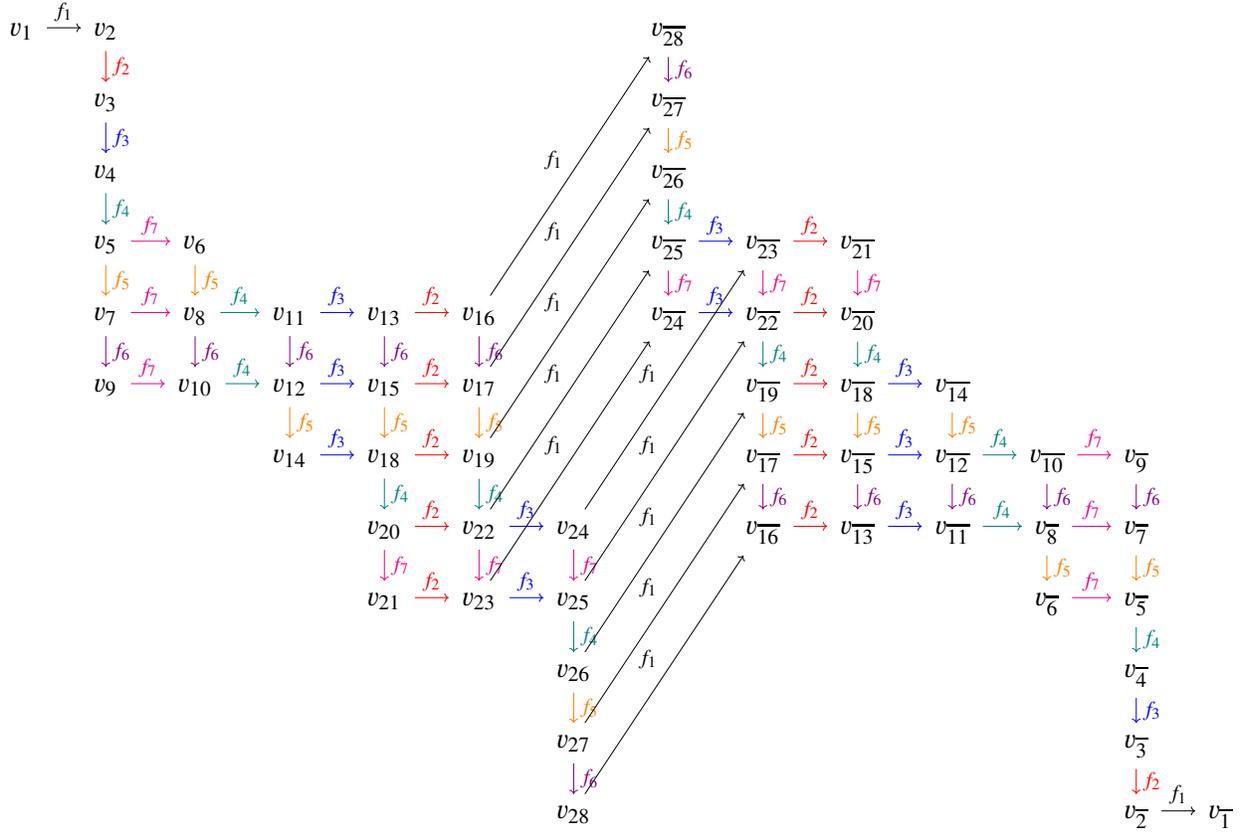

\ddia{sep=small,arrows={line width=0.25pt}}{
v_1\ar[r,"f_1"] \& v_2\ar[d,red,"f_2"] \& \& \& \& \& \& v_{\ol{28}}\ar[d,violet,"f_6"] \& \& \& \& \& \& \\
\& v_3\arrow[d,blue,"f_3"] \& \& \& \& \& \& v_{\ol{27}}\ar[d,orange,"f_5"] \& \& \& \& \& \& \\
\& v_4\arrow[d,teal,"f_4"] \& \& \& \& \& \& v_{\ol{26}}\ar[d,teal,"f_4"] \& \& \& \& \& \& \\ 
\& v_5\ar[r,magenta,"f_7"]\ar[d,orange,"f_5"] \& v_6\ar[d,orange,"f_5"] \& \& \& \& \& v_{\ol{25}}\ar[r,blue,"f_3"]\ar[d,magenta,"f_7"] \& v_{\ol{23}}\ar[r,red,"f_2"]\ar[d,magenta,"f_7"] \& v_{\ol{21}}\ar[d,magenta,"f_7"] \& \& \& \& \\
\& v_7\ar[d,violet,"f_6"]\ar[r,magenta,"f_7"] \& v_8\ar[r,teal,"f_4"]\ar[d,violet,"f_6"] \& v_{11}\ar[r,blue,"f_3"]\ar[d,violet,"f_6"] \& v_{13}\ar[r,red,"f_2"]\ar[d,violet,"f_6"] \& v_{16} \ar[rruuuu,"f_1"] \ar[d,violet,"f_6"] \& \& v_{\ol{24}}\ar[r,blue,"f_3"] \& v_{\ol{22}}\ar[r,red,"f_2"]\ar[d,teal,"f_4"] \& v_{\ol{20}}\ar[d,teal,"f_4"] \& \& \& \& \\
\& v_9 \ar[r,magenta,"f_7"] \& v_{10}\ar[r,teal,"f_4"] \& v_{12}\ar[r,blue,"f_3"]\ar[d,orange,"f_5"] \& v_{15}\ar[r,red,"f_2"]\ar[d,orange,"f_5"] \& v_{17}\ar[rruuuu,"f_1"]\ar[d,orange,"f_5"] \& \& \& v_{\ol{19}}\ar[r,red,"f_2"]\ar[d,orange,"f_5"] \& v_{\ol{18}}\ar[r,blue,"f_3"]\ar[d,orange,"f_5"] \& v_{\ol{14}}\ar[d,orange,"f_5"] \& \& \& \\
\& \& \& v_{14}\ar[r,blue,"f_3"] \& v_{18}\ar[r,red,"f_2"]\ar[d,teal,"f_4"] \& v_{19}\ar[rruuuu,"f_1"]\ar[d,teal,"f_4"] \& \& \& v_{\ol{17}}\ar[r,red,"f_2"]\ar[d,violet,"f_6"] \& v_{\ol{15}}\ar[r,blue,"f_3"]\ar[d,violet,"f_6"] \& v_{\ol{12}}\ar[r,teal,"f_4"]\ar[d,violet,"f_6"] \& v_{\ol{10}}\ar[r,magenta,"f_7"]\ar[d,violet,"f_6"] \& v_{\ol{9}}\ar[d,violet,"f_6"] \& \\ 
\& \& \& \& v_{20}\ar[r,red,"f_2"]\ar[d,magenta,"f_7"] \& v_{22} \ar[r,blue,"f_3"] \ar[d,magenta,"f_7"] \ar[rruuuu,"f_1"] \& v_{24}\ar[d,magenta,"f_7"]\ar[rruuuu,"f_1"] \&  \& v_{\ol{16}}\ar[r,red,"f_2"] \& v_{\ol{13}}\ar[r,blue,"f_3"] \& v_{\ol{11}}\ar[r,teal,"f_4"] \& v_{\ol{8}}\ar[r,magenta,"f_7"]\ar[d,orange,"f_5"] \& v_{\ol{7}}\ar[d,orange,"f_5"] \& \\
\& \& \& \& v_{21}\ar[r,red,"f_2"] \& v_{23}\ar[r,blue,"f_3"]\ar[rruuuu,"f_1"] \& v_{25} \ar[d,teal,"f_4"] \ar[rruuuu,"f_1"] \& \& \& \& \& v_{\ol{6}}\ar[r,magenta,"f_7"] \&  v_{\ol{5}}\ar[d,teal,"f_4"] \& \\
\& \& \& \& \& \& v_{26} \ar[rruuuu,"f_1"] \ar[d,orange,"f_5"] \& \& \& \& \& \& v_{\ol{4}}\ar[d,blue,"f_3"] \& \\
\& \& \& \& \& \& v_{27} \ar[rruuuu,"f_1"] \ar[d,violet,"f_6"] \& \& \& \& \& \& v_{\ol{3}}\ar[d,red,"f_2"] \& \\
\& \& \& \& \& \& v_{28} \ar[rruuuu,"f_1"] \& \& \& \& \& \& v_{\ol{2}}\ar[r,"f_1"] \& v_{\ol{1}}
}
    \caption{First fundamental module for E$_7$.}
    \label{fig:E7 first fundamental}
\end{figure}
Here $v_j$ are ordered as their $\ell$-weights appear in the $q$-character of $\tl L_{\mr{1}_0}$ in \eqref{qchar E7} and $\ol{i}=57-i$, $1\leq i\leq 28$. The coefficients of all the arrows are one. The action of $f_i$'s is indicated in this diagram. The action of $e_i$'s is obtained by reversing all the arrows and keeping the same coefficient on each arrow.

The subalgebra of $U_q(E_7)$ generated by $\{e_i,f_i,k_i^{\pm1}:2\le i\le 7\}$ is isomorphic to $U_q(E_6)$.
As a module over this $U_q(E_6)$ subalgebra, the vector representation shown in Figure \ref{fig:E7 first fundamental}, and the $133$-dimensional $U_q(E_7)$-adjoint representation $L_{\omega_6}\se L_{\omega_1}^{\otimes 2}$ (see Figure \ref{fig:E7 adjoint}) decompose respectively as 
\eq{\label{E7 om1 E6}
\underbracket[0.1ex]{L_{\omega_1}}_{56}\cong \underbracket[0.1ex]{L_{\omega_0}^{(6)}}_{1}\oplus \underbracket[0.1ex]{L_{\omega_1}^{(6)}}_{27} \oplus \underbracket[0.1ex]{L_{\omega_5}^{(6)}}_{27} \oplus \underbracket[0.1ex]{L_{\omega_0}^{(6)}}_{1}\ ,} 
\eq{\label{E7 om6 E6}
\underbracket[0.1ex]{L_{\omega_6}}_{133}\cong \underbracket[0.1ex]{L_{\omega_1}^{(6)}}_{27} \oplus \underbracket[0.1ex]{L_{\omega_6}^{(6)}}_{78} \oplus \underbracket[0.1ex]{L_{\omega_0}^{(6)}}_{1}\oplus \underbracket[0.1ex]{L_{\omega_5}^{(6)}}_{27}\ ,
}
where $L^{(6)}_{\lambda}$ are $U_q(E_6)$-irreducible modules of highest $U_q(E_6)$-weight $\lambda$. The summands in \eqref{E7 om1 E6} are spans of $\{v_1\},\ \{v_i:2\le i\le 28\},\ \{v_{\ol i}:2\le i\le 28\},\ \{v_{\ol 1}\}$ respectively.

We now describe a basis $\{u_s\}_{s=1}^{133}$ for $L_{\omega_6}\se L_{\omega_1}^{\otimes 2}$ which is ordered such that the summands in \eqref{E7 om6 E6} are respectively spans of $\{u_s:1\le s\le 27\},\ \{u_s:28\le s\le 106,\,s\ne64\},\ \{u_{64}\},\ \{u_s:107\le i\le 133\}$. The vectors $u_s$, $64\le s\le 70$, are of zero weight. Vectors $u_{65},\dots,u_{70}$ come from $L_{\omega_6}^{(6)}$ and $u_{64}$ generates $L_{\omega_0}^{(6)}$. A diagram of the $L_{\om_6}$ representation in our choice of basis around zero weight vectors is shown in Figure \ref{fig:E7 adjoint}.
\begin{figure}
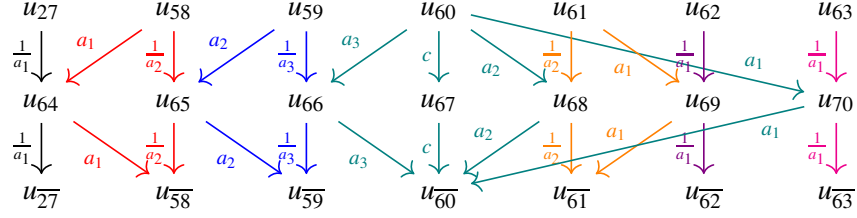

\dia{
u_{27}\ar[d,"{\frac{1}{a_1}}"'] \& u_{58} \ar[d,red,"{\frac{1}{a_2}}"'] \ar[ld,red,"a_1"'] \& u_{59} \ar[d,blue,"{\frac{1}{a_3}}"'] \ar[ld,blue,"a_2"'] \& u_{60} \ar[d,teal,"c"']\ar[rrrd,teal,"a_1",pos=0.8]\ar[rd,teal,"a_2"']\ar[ld,teal,"a_3"'] \& u_{61} \ar[d,orange,"{\frac{1}{a_2}}"'] \ar[rd,orange,"a_1"'] \& u_{62}\ar[d,violet,"{\frac{1}{a_1}}"'] \& u_{63} \ar[d,magenta,"{\frac{1}{a_1}}"'] \\ 
u_{64}\ar[d,"{\frac{1}{a_1}}"']\ar[rd,red,"a_1"'] \& u_{65}\ar[d,red,"{\frac{1}{a_2}}"']\ar[rd,blue,"a_2"'] \& u_{66}\ar[d,blue,"{\frac{1}{a_3}}"'] \ar[rd,teal,"a_3"'] \& u_{67} \ar[d,teal,"c"'] \& u_{68}\ar[d,orange,"{\frac{1}{a_2}}"']\ar[ld,teal,"a_2"'] \& u_{69}\ar[d,violet,"{\frac{1}{a_1}}"']\ar[ld,orange,"a_1"'] \& u_{70} \ar[d,magenta,"{\frac{1}{a_1}}"'] \ar[llld,teal,"a_1",pos=0.15] \\ 
u_{\ol{27}} \& u_{\ol{58}} \& u_{\ol{59}} \& u_{\ol{60}} \& u_{\ol{61}} \& u_{\ol{62}} \& u_{\ol{63}}
}
    \caption{The adjoint module for E$_7$ (around weight zero vectors $u_i,64\le i\le 70$).}
    \label{fig:E7 adjoint}
\end{figure}
Here $\ol{i}=134-i$, $a_j=\sqrt{\frac{[j]}{[j+1]}}$, $1\le j\le 3$, $c=\sqrt{\frac{[3]_3^{\mr{i}}}{[3]\,[4]}}$, and the colors of arrows correspond to simple roots as follows: 
\medskip
\dia{ \,\ar[r,"f_1"'] \& \,\,\,\,\,\,\,\,\ar[r,red,"f_2"'] \& \,\,\,\,\,\,\,\, \ar[r,blue,"f_3"'] \&  \,\,\,\,\,\,\,\, \ar[r,teal,"f_4"'] \&  \,\,\,\,\,\,\,\, \ar[r,orange,"f_5"'] \&  \,\,\,\,\,\,\,\, \ar[r,violet,"f_6"'] \&  \,\,\,\,\,\,\,\, \ar[r,magenta,"f_7"'] \& \, } 

We note that $[3]_3^{\mr{i}}=\kappa_{36}(q)$ is the symmetric form of 36-th cyclotomic polynomial.

To complete the  diagram in Figure \ref{fig:E7 adjoint}, one has to add 102 more vectors  and connect by arrows of color $i$ the pairs of vectors whose $\ell$-weight differ by an $i$-th affine root. All these arrows have coefficient one.
Then the total diagram describes the action of $f_i$, $i\in \mr{I}$.
For example, $f_2\,v_{58}=a_1\,v_{64}+\frac{1}{a_2}v_{65}$. The action of $e_i$'s is obtained by reversing all the arrows and keeping the same coefficient on each arrow.

\medskip

The basis vectors $u_s$ are of the form
\bee{
u_s=\sum_{(i,j)\in I_s^{\omega_6}}\ve_{ij}^{q,s}v_i\otimes v_j\ ,\quad 1\le s\le 133\ .
}
The sets $I_s^{\omega_6}$ and coordinates $\ve_{ij}^{q,s}$, $1\le s\le 133$, are used in the expression of $T(z)$ in $\eqref{E7TQ}$. The sets $I_s^{\omega_6}$ have cardinality $56$ for $64\le s\le 70$ and 12 otherwise, and they have the property $(i,j)\in I_s^{\omega_6}$ if and only if $(j,i)\in I_s^{\omega_6}$. The element $(j,i)$, $i<j$, is placed in $I_s^{\omega_6}$ symmetrically, that is such that $|(i,j)|+|(j,i)|=|I_s^{\omega_6}|+1$. We always have $i\ne j$ in this case. The corresponding coordinates $\ve_{ij}^{q,s}$ have the property $
\ve_{ji}^{q,s}=\ve_{ij}^{q^{-1},s},\ i<j,\ 1\le s\le 133$.
Therefore, we list $I_s^{\omega_6}$ and $\ve_{ij}^{q,s}$ here only for $i<j$.

For $1\le s\le 27$, the 12 element sets $I_s^{\omega_6}$ are related to the vectors in the first summand in \eqref{E7 om6 E6}, and the corresponding $6$ element subsets are written using the $5$-element $E_6$ lists $I_s^{\omega_5,(6)}$ in Figure \ref{fig:E6I} as 
\bee{
\big\{(1,s+28)\big\}\cup \big\{(i+1,j+1):(i,j)\in I_s^{\omega_5,(6)}\big\}\ ,\quad 1\le s\le 27.
}
Here the position of $(1,s+28)$, $1\le s\le 27$, is $1$, and the position of $(i+1,j+1)$ is one more than the position of $(i,j)$ in $I_s^{\om_5,(6)}$. For $28\le s\le 63$, the sets $I_s^{\omega_6}$ are related to the positive roots in the second summand in \eqref{E7 om6 E6}, which is the $78$-dimensional adjoint representation of $U_q($E$_6)$. These $36$ sets with $6$ elements are listed in Figure \ref{fig:E7I}.
\begin{figure}
\bee{
&\scriptstyle \big\{(2,36),(3,34),(4,32),(5,31),(7,30),(9,29)\big\},\, 
\big\{(2,37),(3,35),(4,33),(6,31),(8,30),(10,29)\big\},\, 
\big\{(2,39),(3,38),(5,33),(6,32),(11,30),(12,29)\big\},
}
\bee{
&\scriptstyle \big\{(2,43),(4,38),(5,35),(6,34),(13,30),(15,29)\big\},\, 
\big\{(2,42),(3,40),(7,33),(8,32),(11,31),(14,29)\big\},\, 
\big\{(3,43),(4,39),(5,37),(6,36),(16,30),(17,29)\big\},
}
\bee{
&\scriptstyle \big\{(2,45),(4,40),(7,35),(8,34),(13,31),(18,29)\big\},\, 
\big\{(2,44),(3,41),(9,33),(10,32),(12,31),(14,30)\big\},\,
\big\{(3,45),(4,42),(7,37),(8,36),(16,31),(19,29)\big\}, 
}
\bee{
&\scriptstyle\big\{(2,47),(5,40),(7,38),(11,34),(13,32),(20,29)\big\},\,
\big\{(2,46),(4,41),(9,35),(10,34),(15,31),(18,30)\big\},\,
\big\{(3,47),(5,42),(7,39),(11,36),(16,32),(22,29)\big\},
}
\bee{
&\scriptstyle \big\{(3,46),(4,44),(9,37),(10,36),(17,31),(19,30)\big\},\,
\big\{(2,49),(5,41),(9,38),(12,34),(15,32),(20,30)\big\},\,
\big\{(2,48),(6,40),(8,38),(11,35),(13,33),(21,29)\big\}, 
}
\bee{
&\scriptstyle \big\{(4,47),(5,45),(7,43),(13,36),(16,34),(24,29)\big\},\,
\big\{(3,49),(5,44),(9,39),(12,36),(17,32),(22,30)\big\},\,
\big\{(3,48),(6,42),(8,39),(11,37),(16,33),(23,29)\big\}, 
}
\bee{
&\scriptstyle \big\{(2,51),(7,41),(9,40),(14,34),(18,32),(20,31)\big\},\,
\big\{(2,50),(6,41),(10,38),(12,35),(15,33),(21,30)\big\},\,
\big\{(4,49),(5,46),(9,43),(15,36),(17,34),(24,30)\big\}, 
}
\bee{
&\scriptstyle \big\{(4,48),(6,45),(8,43),(13,37),(16,35),(25,29)\big\},\,
\{(3,51),(7,44),(9,42),(14,36),(19,32),(22,31)\big\},\,
\big\{(3,50),(6,44),(10,39),(12,37),(17,33),(23,30)\big\}, 
}
\bee{
&\scriptstyle\{(2,52),(8,41),(10,40),(14,35),(18,33),(21,31)\big\},\,
\big\{(4,51),(7,46),(9,45),(18,36),(19,34),(24,31)\big\},\,
\big\{(4,50),(6,46),(10,43),(15,37),(17,35),(25,30)\big\}, 
}
\bee{
&\scriptstyle \big\{(5,48),(6,47),(11,43),(13,39),(16,38),(26,29) \big\},\,
\big\{(3,52),(8,44),(10,42),(14,37),(19,33),(23,31)\big\},\,
\big\{(2,53),(11,41),(12,40),(14,38),(20,33),(21,32)\big\}, 
}
\bee{
&\scriptstyle \big\{(5,51),(7,49),(9,47),(20,36),(22,34),(24,32)\big\},\,
\big\{(7,48),(8,47),(11,45),(13,42),(16,40),(27,29)\big\},\,
\big\{(5,50),(6,49),(12,43),(15,39),(17,38),(26,30)\big\}, 
}
\bee{
&\scriptstyle\big\{(4,52),(8,46),(10,45),(18,37),(19,35),(25,31)\big\},\,
\big\{(3,53),(11,44),(12,42),(14,39),(22,33),(23,32)\big\},\,
\big\{(2,54),(13,41),(15,40),(18,38),(20,35),(21,34)\big\}\ .
}
\caption{The index sets $I_s^{\omega_6}\ (i<j),\ 28\le s\le 63$, for E$_7$.}
\label{fig:E7I}
\end{figure}
The coordinates are given by $\ve_{ij}^{q,s}=-(-q)^{6-|(i,j)|}$, $i<j$, $1\le s\le 63$.

For $71\le s\le 133$, we have
\bee{
I_s^{\omega_6}=\{(\ol{j},\ol{i}):(i,j)\in I_{134-s}^{\omega_6}\}\ ,\quad \ve_{ij}^{q,s}=\ve_{\ol{j}\,\ol{i}}^{q,134-s}\ .
}
Here $(\ol{j},\ol{i})$ has the same position in $I_s^{\omega_6}$ as $(i,j)$ in $I^{\omega_6}_{134-s}$. These sets correspond to the negative roots of $L_{\omega_6}$.

For $64\le s\le 70$, the sets $I_s^{\omega_6}$ are all the same. These sets are related to the zero weight vectors in $L_{\omega_6}$,
\bee{
I_s^{\omega_6}=\{(i,\ol{i}):1\le i\le 56\}\ .
}
The coordinates  $\{\ve_{i\,\ol{i}}^{q,s}:1\le i\le 28\}$, for $64\le s\le 70$ are listed in Figure \ref{fig:E7 parities}. Here $c_3=q^{-3}+q^{-1}-q^3$, $c_5^\pm=q^{\pm 5}+q^{\pm3}+q^{\pm1}-q^{\mp3}-q^{\mp5}$, and we use the notation $\{a\}^{k}$ to indicate repetitions, so that $\{0\}^4$ means $0, 0, 0, 0$.
\begin{figure}
\bee{\frac{q^{1/2}}{\sqrt{[2]}}\bigg\{q^4, q^5, \{0\}^{13}, -q^3, q^2, 0, -q, \{0\}^2, 1, \{-q^{-1} \}^2, q^{-2}, -q^{-3}, q^{-4}, -q^{-5}\bigg\}\ ,\ \frac{q^{1/2}}{\sqrt{[2][3]}}\bigg\{-q^4, q^3, q^5[2], \{0\}^9, -q^3[2], \\
0, q^2[2], -q^5, q^4, -q[2], -q^3, [2], -q^{-1}[2], q^2, -q, q^{-1}, -q^{-2}, q^{-3}, -q^{-4}, q^{-5}\bigg\}\ ,\ \frac{q^{1/2}}{\sqrt{[3][4]}}\bigg\{q^4, -q^3, q^2, q^5[3], \\ 
\{0\}^6, -q^3[3], q^2[3], -q^6, -q[3], \{q^5\}^2, \{-q^4\}^2, q^3, -[2], \{q^{-1}[2]\}^2, -q^{-2}[2], q^2[2], -q[2], -q^{-3}, q^{-4}, -q^{-5}\bigg\}\ , }
\bee{\frac{q^{1/2}}{\sqrt{[3][4][3]_3^\mr{i}}}\bigg\{ -q^{4}[3], q^{3}[3], -q^{2}[3], q[3], -q^{5}c_3[2]_2, \{q^{4}c_3[2]_2\}^2, \{-q^{3}c_3[2]_2\}^2, q^{2}c_3[2]_2, q^{4}c_5^-, \{-q^{3}c_5^-\}^2, \\
\{q^{2}c_5^-\}^3, \{-qc_5^-\}^2, c_5^-, -q^{2}[2]^\mr{i}[3]^2, \{q[2]^\mr{i}[3]^2\}^2, \{-[2]^\mr{i}[3]^2\}^2, q^{-1}[2]^\mr{i}[3]^2, -c_5^+, q^{-1}c_5^+, -q^{-2}c_5^+\bigg\}\ ,}
\bee{\frac{q^{1/2}}{\sqrt{[2][3]}}\bigg\{\{0\}^4, q^{4}[2], -q^{3}[2], q^{6}, \{-q^{5}\}^2, q^{4}, -q^{2}, \{q\}^2, q^{3}[2], \{-1\}^2, q^{-1}, -q^{2}[2], q[2], \{0\}^6, -q^{-1}[2], -q, 1 \bigg\}, }
\bee{\frac{q^{1/2}}{\sqrt{[2]}}\bigg\{\{0\}^6, q^4, -q^{3}, q^{5}, -q^{4}, q^{2}, q^{3}, -q, 0, -q^{2}, 1, q, \{0\}^9, -q^{-1}, -1\bigg\}\ , }
\bee{\frac{q^{1/2}}{\sqrt{[2]}}\bigg\{\{0\}^4, q^4, q^5, -q^3, -q^4, q^2, q^3,\{0\}^9, -q, -q^2, 1, q, -q^{-1}, -1,\{0\}^3\bigg\}\ .
}
\caption{The coordinates $\{\ve_{i\,\ol{i}}^{q,s}:1\le i\le 28\},\,64\le s\le 70$ corresponding to $L_{\omega_6}$ for E$_7$.}
\label{fig:E7 parities}
\end{figure}

\subsection{Type F\texorpdfstring{$_4$}{2}}
\label{F4 app}

A diagram of the first fundamental module $L_{\omega_1}$ is shown in Figure \ref{fig:F4 first fundamental}.
\comment{\dia{v_1\ar[r,teal,"f_4"'] \& v_2\ar[r,blue,"f_3"'] \& v_3\ar[r,red,"f_2"'] \& v_{13}\ar[lld,"f_1"]\ar[r,red,"f_2"'] \& v_4\ar[lld,"f_1"]\ar[r,blue,"f_3"] \& v_5\ar[lld,"f_1"]\ar[r,teal,"f_4"] \& v_6\ar[lld,"f_1"] \& \& \\ \& v_{14}\ar[r,red,"f_2"'] \& v_{15}\ar[lld,"f_1"] \ar[r,blue,"f_3"'] \& v_{16}\ar[lld,"f_1"] \ar[r,teal,"f_4"']\ar[d,red,"f_2"] \& v_{17}\ar[lld,"f_1"] \ar[d,red,"f_2"] \& \& \& \& \\ v_7\ar[r,blue,"f_2"] \& v_8\ar[rrd,red,"f_2"']\ar[r,teal,"f_4"] \& v_9\ar[rrd,red,"f_2"'] \& v_{18}\ar[r,teal,"f_4"']\ar[d,"f_1"] \& v_{20}\ar[r,blue,"f_3"']\ar[d,"f_1"] \& v_{22}\ar[r,red,"f_2"']\ar[d,"f_1"] \& v_{24}\ar[r,"{[2]}f_1"]\ar[d,"{[2]}f_1"] \& v_{28}\ar[d,"{\frac{[2]}{[3]}}f_1"] \& \\ \& \& \& v_{19}\ar[r,teal,"f_4"'] \& v_{21}\ar[r,blue,"f_3"'] \& v_{23}\ar[r,red,"f_2"'] \& v_{27}\ar[r,"{[2]}f_1"']\ar[d,red,"f_2"] \& v_{\ol{12}}\ar[d,red,"f_2"] \& \\ \& \& \& \& \& \& v_{\ol{11}}\ar[d,blue,"f_3"]\ar[r,"f_1"] \& v_{\ol{10}}\ar[d,blue,"f_3"] \& \\ \& \& \& \& \& \&v_{\ol{9}}\ar[r,"f_1"']\ar[d,teal,"f_4"] \& v_{\ol{8}}\ar[r,red,"f_2"']\ar[d,teal,"f_4"] \& v_{\ol{5}}\ar[d,teal,"f_4"]\\ \& \& \& \& \& \& v_{\ol{7}}\ar[r,"f_1"']\& v_{\ol{6}}\ar[r,red,"f_2"'] \& v_{\ol{4}}\ar[d,blue,"f_3"] \\ \& \& \& \& \& \& \& \& v_{\ol{3}}\ar[d,red,"f_2"] \\ \& \& \& \& \& \& \& \& v_{\ol{2}}\ar[d,"f_1"] \\ \& \& \& \& \& \& \& \& v_{\ol{1}}}}
\begin{figure}
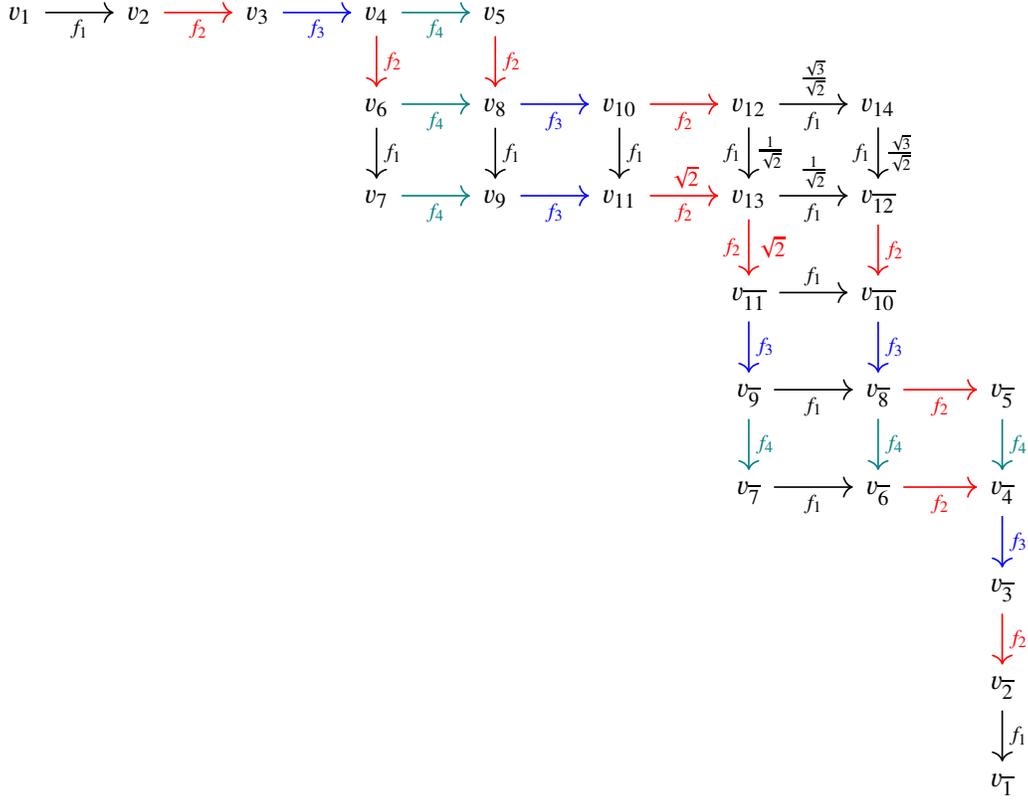

\dia{v_1\ar[r,"f_1"'] \& v_2\ar[r,red,"f_2"'] \& v_3\ar[r,blue,"f_3"'] \& v_4\ar[r,teal,"f_4"']\ar[d,red,"f_2"] \& v_5\ar[d,red,"f_2"] \& \& \& \& \\ \& \& \& v_6\ar[r,teal,"f_4"']\ar[d,"f_1"] \& v_8\ar[r,blue,"f_3"']\ar[d,"f_1"] \& v_{10}\ar[r,red,"f_2"']\ar[d,"f_1"] \& v_{12}\ar[r,"\frac{\sqrt{3}}{\sqrt{2}}","f_1"']\ar[d,"\frac{1}{\sqrt{2}}","f_1"'] \& v_{14}\ar[d,"\frac{\sqrt{3}}{\sqrt{2}}","f_1"'] \& \\ \& \& \& v_7\ar[r,teal,"f_4"'] \& v_9\ar[r,blue,"f_3"'] \& v_{11}\ar[r,red,"\sqrt{2}","f_2"'] \& v_{13}\ar[r,"\frac{1}{\sqrt{2}}","f_1"']\ar[d,red,"\sqrt{2}","f_2"'] \& v_{\ol{12}}\ar[d,red,"f_2"] \& \\ \& \& \& \& \& \& v_{\ol{11}}\ar[d,blue,"f_3"]\ar[r,"f_1"] \& v_{\ol{10}}\ar[d,blue,"f_3"] \& \\ \& \& \& \& \& \&v_{\ol{9}}\ar[r,"f_1"']\ar[d,teal,"f_4"] \& v_{\ol{8}}\ar[r,red,"f_2"']\ar[d,teal,"f_4"] \& v_{\ol{5}}\ar[d,teal,"f_4"]\\ \& \& \& \& \& \& v_{\ol{7}}\ar[r,"f_1"']\& v_{\ol{6}}\ar[r,red,"f_2"'] \& v_{\ol{4}}\ar[d,blue,"f_3"]\\ \& \& \& \& \& \& \& \& v_{\ol{3}}\ar[d,red,"f_2"] \\ \& \& \& \& \& \& \& \& v_{\ol{2}}\ar[d,"f_1"] \\ \& \& \& \& \& \& \& \& v_{\ol{1}}} 
    \caption{First fundamental module for F$_4$}
    \label{fig:F4 first fundamental}
\end{figure}
Here $v_j$ are ordered as their $\ell$-weights appear in the $q$-character of $\tl L_{\mr{1}_0}$ in \eqref{qchar F4} and $\ol{i}=27-i$, $1\leq i\leq 12$. The numbers in coefficients of arrows are quantum numbers, and if coefficient of an arrow is not given, it is assumed to be one. The action of $f_i$'s is indicated in this diagram. For example, $f_1\,v_{12}=\frac{1}{\sqrt{[2]}}v_{13}+\frac{\sqrt{[3]}}{\sqrt{[2]}}v_{14}$. The action of $e_i$'s is obtained by reversing all the arrows and keeping the same coefficient on each arrow. 

We now describe bases $\{w_s\}_{s=1}^{26}$ and $\{u_s\}_{s=1}^{52}$ for $L_{\om_1}\se L_{\om_1}^{\otimes 2}$ and $L_{\om_4}\se L_{\om_1}^{\otimes 2}$ respectively. These basis vectors are of the form 
\bee{
w_s=\sum_{(i,j)\in I_s^{\omega_1}}\mu_{ij}^{q,s}\,v_i\otimes v_j\ ,\quad 1\le s\le 26\ ,\quad u_s=\sum_{(i,j)\in I_s^{\omega_4}}\ve_{ij}^{q,s}v_i\otimes v_j\ ,\quad 1\le s\le 52\ .
}
The sets $I_s^{\omega_1}$ and corresponding coordinates $\mu_{ij}^{q,s}$, $1\le s\le 26$, are used in the expression of $S(z)$ in \eqref{F4TSz}, while the sets $I_s^{\omega_4}$ and the corresponding coordinates $\ve_{ij}^{q,s}$, $1\le s\le 52$, are used in the expression of $T(z)$ in \eqref{F4TSz}.

The sets $I_s^{\la}$, $\la=\om_1,\om_4$, have the property that if $(i,j)\in I_s^{\la}$ then $(j,i)\in I_s^{\la}$. The element $(j,i)$, $i<j$, is placed in $I_s^{\la}$ symmetrically, that is such that $|(i,j)|+|(j,i)|=|I_s^{\la}|+1$. 
The corresponding coordinates $\mu_{ij}^{q,s}$ and $\ve_{ij}^{q,s}$ satisfy 
\bee{
\mu_{ji}^{q,s}=\mu_{ij}^{q^{-1},s}\ ,\quad i\le j\ ,\ 1\le s\le 26\ ,\qquad \ve_{ji}^{q,s}=-\ve_{ij}^{q^{-1},s}\ ,\quad i\le j\ ,\ 1\le s\le 52\ .
}
Therefore, we list the sets $I_s^{\la}$ and the corresponding coordinates here only for $i\le j$. 
We have $i=j$ only for $i=j=13$ and $i=j=14$, in which case $v_i\otimes v_j$ has weight zero.

\medskip

For $1\le s\le 12$, we list the subsets of $I_s^{\omega_1}$ having first coordinate less than the second one in Figure \ref{fig:F4I}.
\begin{figure}
\bee{
& \scriptstyle \big\{(1,13),(1,14),(2,12),(3,10),(4,8),(5,6)\big\},\,
\scriptstyle \big\{(1,15),(2,13),(2,14),(3,11),(4,9),(5,7)\big\},\,
\scriptstyle \big\{(1,17),(2,16),(3,13),(3,14),(6,9),(7,8)\big\},
}
\bee{
& \scriptstyle \big\{(1,19),(2,18),(4,13),(4,14),(6,11),(7,10)\big\},\,
\scriptstyle \big\{(1,21),(2,20),(5,13),(5,14),(8,11),(9,10)\big\},\,
\scriptstyle \big\{(1,22),(3,18),(4,16),(6,13),(6,14),(7,12)\big\},
}
\bee{
& \scriptstyle \big\{(1,23),(3,20),(5,16),(8,13),(8,14),(9,12)\big\},\,
\scriptstyle \big\{(1,24),(4,20),(5,18),(10,13),(10,14),(11,12) \big\},\,
\scriptstyle \big\{(2,22),(3,19),(4,17),(6,15),(7,13),(7,14)\big\},
}
\bee{
& \scriptstyle \big\{(2,23),(3,21),(5,17),(8,15),(9,13),(9,14)\big\},\,
\scriptstyle \big\{(2,24),(4,21),(5,19),(10,15),(11,13),(11,14)\big\},\,
\scriptstyle \big\{(1,25),(6,20),(8,18),(10,16),(12,13),(12,14) \big\}.
}
\caption{The index sets $I_s^{\omega_1}\ (i<j),\ 1\le s\le 12$ for F$_4$.}
\label{fig:F4I}
\end{figure}

The corresponding coordinates $\big\{q^{-1/2}\sqrt{[2]}\,\mu_{ij}^{q,s}/\sqrt{[3]}:(i,j)\in I_s^{\omega_1},i<j\big\},\ 1\le s\le 12$, are listed below.
\bee{
\bigg\{0,\frac{q^{11/2}\sqrt{[2]}}{\sqrt{[3]}}, -q^4, q^3, -q, q^{-1}\bigg\},\ s=1\ ,\quad \bigg\{q^5, \frac{-q^{9/2}}{\sqrt{[2]}}, \frac{-q^{5/2}}{\sqrt{[2][3]}}, q^3, -q, q^{-1}\bigg\},\ s=2\ ,
}
\bee{
\bigg\{q^5, -q^4, \frac{q^{5/2}}{\sqrt{[2]}}, \frac{-q^{5/2}}{\sqrt{[2][3]}}, -q, 1\bigg\},\ 3\le s\le 5\ ,\quad \bigg\{q^5, -q^4, q^2, \frac{-q^{1/2}}{\sqrt{[2]}}, \frac{-q^{5/2}}{\sqrt{[2][3]}}, 1\bigg\},\ 6\le s\le 8\ ,
}
\bee{
\bigg\{q^5, -q^4, q^2, -q, 0, \frac{q^{-1/2}\sqrt{[2]}}{\sqrt{[3]}}\bigg\},\ 9\le s\le 11\ ,\quad \bigg\{q^5, -q^4, q^2, -1, \frac{q^{-3/2}}{\sqrt{[2]}}, \frac{-q^{5/2}}{\sqrt{[2][3]}}\bigg\},\ s=12\ .
}

For $13\le s\le 14$, the sets $I_s^{\omega_1}$ are the same. These sets correspond to zero weight vectors in $L_{\omega_1}$. We have 
$$I_s^{\omega_1}=\{(1,\ol{1}),\dots,(13,\ol{13}),(13,13),(14,14),(\ol{13},13),\dots,(\ol{1},1)\}\ ,\quad 13\le s\le 14\ .$$
The coordinates $\big\{\mu_{i\,\ol{i}}^{q,s}:1\le i\le 13\big\} \cup \big\{\mu_{13,13}^{q,s},\mu_{14,14}^{q,s}\big\}$, $13\le s\le 14$, are listed below.
\bee{
\frac{1}{\sqrt{[2]}}\bigg\{0, q^5, q^6, -q^4, q^2, -q^5, 0, q^3, 0, -q, 0, -q^2, \frac{[2]}{\sqrt{[3]}}, 0, 0\bigg\}\ , 
}
\bee{
\frac{1}{\sqrt{[2][3]}}\bigg\{q^5[2], q^7, -q^6, q^4, -q^2, -q^3, -q^5[2], q, q^3[2], -q^{-1}, -q[2], q^{-2}, 0, [2], -[2]_3 \bigg\}\ .
}

For $15\le s\le 26$, we have 
\bee{
I_s^{\omega_1}=\{(\ol{j},\ol{i}):(i,j)\in I_{27-s}^{\omega_1}\}\ ,\quad \mu_{ij}^{q,s}=\mu_{\ol{j}\,\ol{i}}^{q,27-s}\ .
}
Here $(\ol{j},\ol{i})$ has the same position in $I_s^{\omega_1}$ as $(i,j)$ in $I_{27-s}^{\omega_1}$, and $\ol{13}=13$, $\ol{14}=14$.

\medskip

A diagram of the adjoint representation $L_{\omega_4}$ is shown in Figure \ref{fig:F4 adjoint}. Here $v_{25}, v_{26},v_{27}, v_{28}$ are zero weight vectors spanning the Cartan subalgebra. Negative roots denoted by dots can be added symmetrically.
\begin{figure}
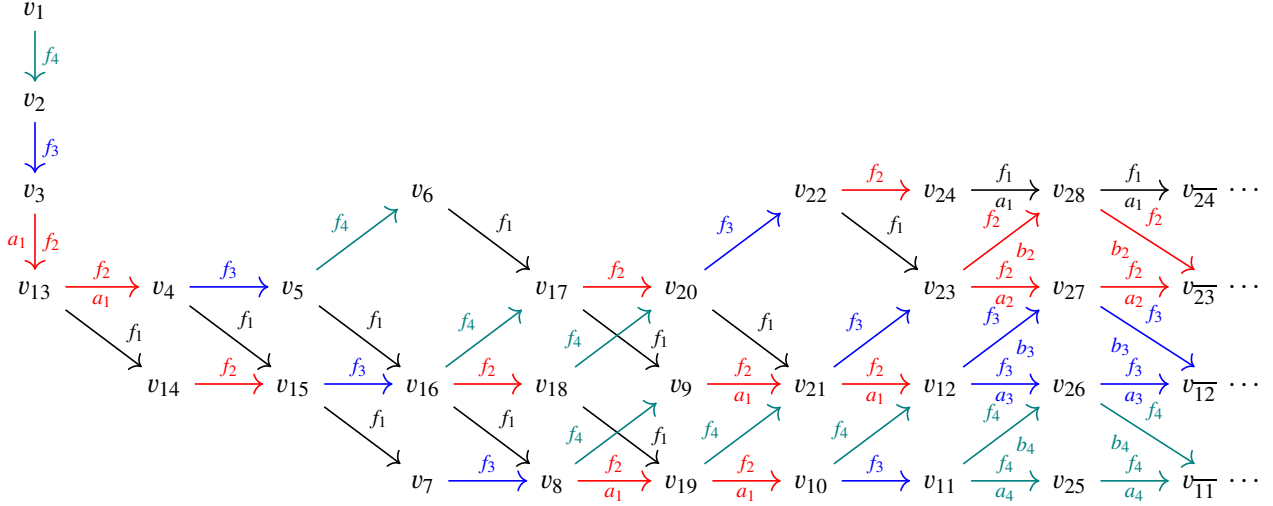

\dia{
v_1\ar[d,teal,"f_4"] \& \& \& \& \& \& \& \& \& \\
v_2\ar[d,blue,"f_3"] \& \& \& \& \& \& \& \& \& \\
v_3\ar[d,red,"f_2","a_1"'] \& \& \& v_6 \ar[rd,"f_1"] \& \& \& v_{22} \ar[r,red,"f_2"] \ar[rd,"f_1"] \& v_{24} \ar[r,"f_1","a_1"'] \& v_{28} \ar[r,"f_1","a_1"'] \ar[rd,red,"f_2","b_2"',pos=0.4] \& v_{\ol{24}} \ \cdots \\
v_{13} \ar[r,red,"f_2","a_1"'] \ar[rd,"f_1",pos=0.7] \& v_4\ar[r,blue,"f_3"]\ar[rd,"f_1"] \& v_5\ar[ru,teal,"f_4"]\ar[rd,"f_1"] \& \& v_{17} \ar[r,red,"f_2"] \ar[rd,"f_1",pos=0.8] \& v_{20} \ar[ru,blue,"f_3"] \ar[rd,"f_1"] \& \& v_{23} \ar[r,red,"f_2","a_2"'] \ar[ru,red,"f_2","b_2"',pos=0.6] \& v_{27} \ar[r,red,"f_2","a_2"'] \ar[rd,blue,"f_3","b_3"',pos=0.4] \& v_{\ol{23}} \ \cdots \\
\& v_{14}\ar[r,red,"f_2"] \& v_{15} \ar[r,blue,"f_3"] \ar[rd,"f_1"] \& v_{16} \ar[ru,teal,"f_4"] \ar[r,red,"f_2"] \ar[rd,"f_1"] \& v_{18} \ar[ru,teal,"f_4",pos=0.2] \ar[rd,"f_1",pos=0.8] \& v_9 \ar[r,red,"f_2","a_1"'] \& v_{21} \ar[r,red,"f_2","a_1"'] \ar[ru,blue,"f_3"] \& v_{12} \ar[r,blue,"f_3","a_3"'] \ar[ru,blue,"f_3","b_3"',pos=0.6] \& v_{26}\ar[r,blue,"f_3","a_3"'] \ar[rd,teal,"f_4","b_4"',pos=0.4] \& v_{\ol{12}} \ \cdots \\
\& \& \& v_7 \ar[r,blue,"f_3"] \& v_8 \ar[r,red,"f_2","a_1"'] \ar[ru,teal,"f_4",pos=0.2] \& v_{19} \ar[r,red,"f_2","a_1"'] \ar[ru,teal,"f_4",pos=0.3] \& v_{10} \ar[r,blue,"f_3"] \ar[ru,teal,"f_4",pos=0.3] \& v_{11} \ar[r,teal,"f_4","a_4"'] \ar[ru,teal,"f_4","b_4"',pos=0.6] \& v_{25} \ar[r,teal,"f_4","a_4"'] \& v_{\ol{11}} \ \cdots
}
    \caption{The adjoint module for F$_4$.}
    \label{fig:F4 adjoint}
\end{figure}
We have $\ol{i}=53-i$, $1\le i \le 24$, $\displaystyle a_1=\sqrt{[2]},\ a_2=\frac{\sqrt{[3]}}{\sqrt{[2]}},\ a_3=\frac{\sqrt{[2]_4}}{\sqrt{[3]}},\ a_4=\frac{\sqrt{[3]_3^{\mr{i}}}}{\sqrt{[2]_4}},\ b_2=a_1^{-1},\ b_3=\sqrt{[2]}\,a_2^{-1},\ b_4=a_3^{-1}$.
If coefficient of an arrow is not given, it is assumed to be one. The action of $f_i$'s is indicated in this diagram. For example, $f_2\,v_{23}=a_2\,v_{27}+b_2\,v_{28}$. The action of $e_i$'s is obtained by reversing all the arrows and keeping the same coefficient on each arrow.

We note that $[3]_3^{\mr{i}}=\kappa_{36}(q)$ is the symmetric form of 36-th cyclotomic polynomial and $[2]_4=\kappa_{16}(q)$ is the symmetric form of 16-th cyclotomic polynomial.

\medskip

For $1\le s\le 12$, we list below the subsets of $I_s^{\omega_4}$ having first coordinate less than the second one. 
\bee{
& \scriptstyle \big\{(1,7),(2,6),(3,4)\big\},\, \big\{(1,9),(2,8),(3,5)\big\},\, \big\{(1,11),(2,10),(4,5)\big\},\, \big\{(1,16),(3,12),(6,8)\big\},\, \big\{(1,18),(4,12),(6,10)\big\}, \big\{(1,20),(5,12),(8,10)\big\},
}
\bee{
& \scriptstyle \big\{(2,17),(3,15),(7,9)\big\},\, \big\{(2,19),(4,15),(7,11)\big\},\, \big\{(2,21),(5,15),(9,11)\big\},\, \big\{(3,22),(6,17),(7,16)\big\},\, \big\{(3,23),(8,17),(9,16)\big\},\, \big\{(4,22),(6,19),(7,18)\big\}.
}
The corresponding coordinates are given by $\ve_{ij}^{q,s}=-(-q)^{4-|(i,j)|},\ (i,j)\in I_s,\,i<j,\ 1\le s\le 12$.

For $13\le s\le 24$, $I_s^{\omega_4}=I_{s-12}^{\omega_1}$.  The corresponding coordinates $\{q^{-1/2}\sqrt{[2]}\,\ve_{ij}^{q,s}:(i,j)\in I_s,\,i<j\}$ are listed below.
\bee{
\bigg \{q^{3/2}\sqrt{[2]}, 0, -q, -q^2, 1, -q^{-2}\bigg\},\ s=13\ ,\quad \bigg\{q^2, \frac{q^{7/2}}{\sqrt{[2]}}, \frac{-q^{3/2}\sqrt{[3]}}{\sqrt{[2]}}, -q^2, 1, -q^{-2}\bigg\},\ s=14\ ,
}
\bee{
\bigg\{q^2, q^3, \frac{-q^{3/2}}{\sqrt{[2]}}, \frac{-q^{3/2}\sqrt{[3]}}{\sqrt{[2]}}, 1, q\bigg\},\ 15\le s\le 17\ ,\ \ \bigg\{q^2, q^3, -q, \frac{q^{-1/2}}{\sqrt{[2]}}, \frac{-q^{3/2}\sqrt{[3]}}{\sqrt{[2]}}, q\bigg\},\ 18\le s\le 20\ ,
}
\bee{
\bigg\{q^2, q^3, -q, -q^2, q^{1/2}\sqrt{[2]}, 0\bigg\},\ 21\le s\le 23\ ,\quad \bigg\{q^2, q^3, -q, q^{-1}, \frac{-q^{-5/2}}{\sqrt{[2]}}, \frac{-q^{3/2}\sqrt{[3]}}{\sqrt{[2]}}\bigg\},\ s=24\ .
}

For $25\le s\le 28$, the sets $I_s^{\omega_4}$ correspond to zero weight vectors in $L_{\omega_4}$ and are given by $$I_s^{\omega_4}=\{(1,\ol{1}),\dots,(13,\ol{13}), (13,13), (14,14),(\ol{13},13),\dots,(\ol{1},1)\}\ .$$
We list below the corresponding coordinates $\{\ve_{i\ol{i}}^{q,s}: 1\le i\le 13\}\cup \big\{\ve_{13,13}^{q,s},\ve_{14,14}^{q,s}\big\}$, $25\le s\le 28$.
\bee{
\frac{1}{\sqrt{[2]_4[3]_3^\mr{i}}}\bigg\{-q^2, q, -1, q^{-2}, q^5[2]_3^\mr{i}, -q^{-3}, q^{-4}, -q^4[2]_3^\mr{i}, q^3[2]_3^\mr{i}, q^2[2]_3^\mr{i}, -q[2]_3^\mr{i}, -q[2]_3^\mr{i}, 0, [2]_3^\mr{i}, [2]_3^\mr{i}\bigg\}\ ,\\ 
\frac{1}{\sqrt{[2]_4[3]}}\bigg\{q^2, -q, 1, q^6, q[2], -q^5, q^4, -[2], q^{-1}[2], -q^2[2]_3^\mr{i}, q[2]_3^\mr{i}, q[2]_3^\mr{i}, 0, -[2]_3^\mr{i}, -[2]_3^\mr{i} \bigg\}\ ,\\ 
\frac{1}{[2]\sqrt{[3]}}\bigg\{-q^2, q, q^3[2], q^3[2], -q[2], q^5, -q^4, -q^3, q^2, -q[3], [3], -q[2]_3^\mr{i}, 0, [2]_3^\mr{i}, [2]_3^\mr{i} \bigg\}\ ,\\ 
\frac{1}{[2]}\bigg\{q^2, q^3, 0, 0, 0, q^3, q^4, -q, -q^2, q^{-1}, 1, -[2]_2, 0, [2]^\mr{i}, -[2]_3^\mr{i}\bigg\}\ .
}

For $29\le s\le 52$, we have 
\bee{
I_s^{\omega_4}=\{(\ol{j},\ol{i}):(i,j)\in I_{53-s}^{\omega_4}\}\ ,\quad \ve_{ij}^{q,s}=\ve_{\ol{j}\,\ol{i}}^{q,53-s}\ .
}
Here $(\ol{j},\ol{i})$ in $I_s^{\om_4}$ has the same position as $(i,j)$ in $I_{53-s}^{\om_4}$, and $\ol{13}=13$, $\ol{14}=14$.

\subsection{Type G\texorpdfstring{$_2$}{2}}
\label{G2 app}

The following diagram shows the first fundamental representation $L_{\omega_1}$: 
\bigskip

\dia{v_1\ar[r,red,"f_1"'] \& v_2\ar[r,blue,"f_2"'] \& v_3\ar[r,red, "\sqrt{2}","f_1"'] \& v_4\ar[r,red,"\sqrt{2}","f_1"'] \& v_{\ol{3}}\ar[r,blue,"f_2"'] \& v_{\ol{2}}\ar[r,red,"f_1"'] \& v_{\ol{1}}} Here $v_j$ are ordered as their $\ell$-weights appear in the $q$-character of $\tl L_{\mr{1}_0}$ in \eqref{qchar G2} and $\ol{i}=8-i$, $1\le i\le 3$. The numbers in coefficients of arrows are quantum numbers, and if coefficient of an arrow is not given, it is assumed to be one. The action of $f_i$'s is indicated in the diagram above. For example, $f_1\,v_{3}=\sqrt{[2]}\,v_{4}$. The action of $e_i$'s is obtained by reversing all the arrows and keeping the same coefficient on each arrow.

\medskip

The sets $I_s^{\omega_1}$, $1\le s\le 7=\dim L_{\omega_1}$, appearing in the expression of $S(z)$ in \eqref{G2Sz},  have the property that if $(i,j)\in I_s^{\omega_1}$ then $(j,i)\in I_s^{\omega_1}$. For $1\le s\le 7$, $s\ne 4$, the sets $I_s^{\omega_1}$ have cardinality 4, and do not contain pairs of the form $(i,i)$. Moreover, the positions of $(i,j)$ and $(j,i)$ are symmetric, that is $|i,j|+|(j,i)|=5$.
We list below the subsets of $I_s^{\omega_1}$ with $s\neq 4$ which have the first coordinate less than the second one. 
\bee{
\big\{(1,4),(2,3)\big\},\, 
\big\{(1,5),(2,4)\big\},\, 
\big\{(1,6),(3,4)\big\},\, 
\big\{(2,7),(4,5)\big\},\, 
\big\{(3,7),(4,6)\big\},\, 
\big\{(4,7),(5,6)\big\}.
}
The set $I_4^{\omega_1}$ corresponds to the zero weight vector in $L_{\omega_1}\se L_{\omega_1}^{\otimes 2}$, and is given by $\{(i,\ol{i}):1\le i\le 7\}$. 

The corresponding coordinates $\mu_{ij}^{q,s}$ have the property that $\mu_{ij}^{q,s} = -\mu_{ji}^{q^{-1},s}$. The sets $\{\mu_{ij}^{q,s}:(i,j)\in I_s^{\omega_1},\ i<j\}$ are listed below for $1\le s\le 7$.
\bee{
\big\{q^3,-q^{\frac{3}{2}}\sqrt{[2]}\big\},\,
\big\{q^{\frac{5}{2}}\sqrt{[2]},-q\big\},\,
\big\{q^{\frac{5}{2}}\sqrt{[2]},-q\big\},\,
\big\{q^2,q^3,-1\big\},\,
\big\{q^{\frac{5}{2}}\sqrt{[2]}, -q\big\},\,
\big\{q^{\frac{5}{2}}\sqrt{[2]}, -q\big\},\,
\big\{q^3,-q^{\frac{3}{2}}\sqrt{[2]}\big\},
}
and $\mu_{4,4}^{q,4}=-[2]^\mr{i}$.

\subsection{Type E\texorpdfstring{$_8$}{2}}
\label{E8 app}

The $q$-character of $\tl{L}_{\mr{1}_0}$ has 248 monomials (one with coefficient two), with 8 (one with coefficient two) being zero-weight terms which are shown in the box.
$$\chi_q(\mr{1}_0)=\mr{1}_{0} + \bigg( \ul{\mr{1}_{2}^{-1}\mr{2}_{1}} + \Big[ \mr{2}_{3}^{-1}\mr{3}_{2} + \mr{3}_{4}^{-1}\mr{4}_{3} + \mr{4}_{5}^{-1}\mr{5}_{4} + \mr{5}_{6}^{-1}\mr{6}_{5}\mr{8}_{5} + \mr{6}_{5}\mr{8}_{7}^{-1} + \mr{6}_{7}^{-1}\mr{7}_{6}\mr{8}_{5} + \mr{5}_{6}\mr{6}_{7}^{-1}\mr{7}_{6}\mr{8}_{7}^{-1} $$
$$+ \mr{7}_{8}^{-1}\mr{8}_{5} + \mr{5}_{6}\mr{7}_{8}^{-1}\mr{8}_{7}^{-1} + \mr{4}_{7}\mr{5}_{8}^{-1}\mr{7}_{6} + \mr{4}_{7}\mr{5}_{8}^{-1}\mr{6}_{7}\mr{7}_{8}^{-1} + \mr{3}_{8}\mr{4}_{9}^{-1}\mr{7}_{6} + \mr{4}_{7}\mr{6}_{9}^{-1} + \mr{3}_{8}\mr{4}_{9}^{-1}\mr{6}_{7}\mr{7}_{8}^{-1} $$ 
$$+ \mr{2}_{9}\mr{3}_{10}^{-1}\mr{7}_{6} + \mr{2}_{9}\mr{3}_{10}^{-1}\mr{6}_{7}\mr{7}_{8}^{-1} + \mr{3}_{8}\mr{4}_{9}^{-1}\mr{5}_{8}\mr{6}_{9}^{-1} + \mr{2}_{9}\mr{3}_{10}^{-1}\mr{5}_{8}\mr{6}_{9}^{-1} + \mr{3}_{8}\mr{5}_{10}^{-1}\mr{8}_{9} + \mr{3}_{8}\mr{8}_{11}^{-1} + \mr{2}_{9}\mr{3}_{10}^{-1}\mr{4}_{9}\mr{5}_{10}^{-1}\mr{8}_{9} $$
$$+ \mr{2}_{9}\mr{3}_{10}^{-1}\mr{4}_{9}\mr{8}_{11}^{-1} + \mr{2}_{9}\mr{4}_{11}^{-1}\mr{8}_{9} + \mr{2}_{9}\mr{4}_{11}^{-1}\mr{5}_{10}\mr{8}_{11}^{-1} + \mr{2}_{9}\mr{5}_{12}^{-1}\mr{6}_{11} + \mr{2}_{9}\mr{6}_{13}^{-1}\mr{7}_{12} + \mr{2}_{9}\mr{7}_{14}^{-1} \Big] $$
$$+\, \Big[ \mr{1}_{10}\mr{2}_{11}^{-1}\mr{7}_{6} + \mr{1}_{10}\mr{2}_{11}^{-1}\mr{6}_{7}\mr{7}_{8}^{-1} + \mr{1}_{10}\mr{2}_{11}^{-1}\mr{5}_{8}\mr{6}_{9}^{-1} + \mr{1}_{10}\mr{2}_{11}^{-1}\mr{4}_{9}\mr{5}_{10}^{-1}\mr{8}_{9} + \mr{1}_{10}\mr{2}_{11}^{-1}\mr{4}_{9}\mr{8}_{11}^{-1} + \mr{1}_{10}\mr{2}_{11}^{-1}\mr{3}_{10}\mr{4}_{11}^{-1}\mr{8}_{9} $$
$$+ \mr{1}_{10}\mr{2}_{11}^{-1}\mr{3}_{10}\mr{4}_{11}^{-1}\mr{5}_{10}\mr{8}_{11}^{-1} + \mr{1}_{10}\mr{3}_{12}^{-1}\mr{8}_{9} + \mr{1}_{10}\mr{3}_{12}^{-1}\mr{5}_{10}\mr{8}_{11}^{-1} + \mr{1}_{10}\mr{2}_{11}^{-1}\mr{3}_{10}\mr{5}_{12}^{-1}\mr{6}_{11} + \mr{1}_{10}\mr{3}_{12}^{-1}\mr{4}_{11}\mr{5}_{12}^{-1}\mr{6}_{11} $$
$$+ \mr{1}_{10}\mr{2}_{11}^{-1}\mr{3}_{10}\mr{6}_{13}^{-1}\mr{7}_{12} + \mr{1}_{10}\mr{2}_{11}^{-1}\mr{3}_{10}\mr{7}_{14}^{-1} + \mr{1}_{10}\mr{3}_{12}^{-1}\mr{4}_{11}\mr{6}_{13}^{-1}\mr{7}_{12} + \mr{1}_{10}\mr{4}_{13}^{-1}\mr{6}_{11} + \mr{1}_{10}\mr{3}_{12}^{-1}\mr{4}_{11}\mr{7}_{14}^{-1} $$
$$+ \mr{1}_{10}\mr{4}_{13}^{-1}\mr{5}_{12}\mr{6}_{13}^{-1}\mr{7}_{12} + \mr{1}_{10}\mr{4}_{13}^{-1}\mr{5}_{12}\mr{7}_{14}^{-1} + \mr{1}_{10}\mr{5}_{14}^{-1}\mr{7}_{12}\mr{8}_{13} + \mr{1}_{10}\mr{7}_{12}\mr{8}_{15}^{-1} + \mr{1}_{10}\mr{5}_{14}^{-1}\mr{6}_{13}\mr{7}_{14}^{-1}\mr{8}_{13} $$
$$+ \mr{1}_{10}\mr{6}_{13}\mr{7}_{14}^{-1}\mr{8}_{15}^{-1} + \mr{1}_{10}\mr{6}_{15}^{-1}\mr{8}_{13} + \mr{1}_{10}\mr{5}_{14}\mr{6}_{15}^{-1}\mr{8}_{15}^{-1} + \mr{1}_{10}\mr{4}_{15}\mr{5}_{16}^{-1} + \mr{1}_{10}\mr{3}_{16}\mr{4}_{17}^{-1} + \mr{1}_{10}\mr{2}_{17}\mr{3}_{18}^{-1} \Big] + \mr{1}_{10}\mr{1}_{18}\mr{2}_{19}^{-1} \bigg) $$
$$+\, \bigg( \Big[ \ul{\mr{1}_{12}^{-1}\mr{7}_{6}} + \mr{1}_{12}^{-1}\mr{6}_{7}\mr{7}_{8}^{-1} + \mr{1}_{12}^{-1}\mr{5}_{8}\mr{6}_{9}^{-1} + \mr{1}_{12}^{-1}\mr{4}_{9}\mr{5}_{10}^{-1}\mr{8}_{9} + \mr{1}_{12}^{-1}\mr{4}_{9}\mr{8}_{11}^{-1} + \mr{1}_{12}^{-1}\mr{3}_{10}\mr{4}_{11}^{-1}\mr{8}_{9} + \mr{1}_{12}^{-1}\mr{3}_{10}\mr{4}_{11}^{-1}\mr{5}_{10}\mr{8}_{11}^{-1} $$
$$+ \mr{1}_{12}^{-1}\mr{2}_{11}\mr{3}_{12}^{-1}\mr{8}_{9} + \mr{1}_{12}^{-1}\mr{2}_{11}\mr{3}_{12}^{-1}\mr{5}_{10}\mr{8}_{11}^{-1} + \mr{1}_{12}^{-1}\mr{3}_{10}\mr{5}_{12}^{-1}\mr{6}_{11} + \mr{1}_{12}^{-1}\mr{2}_{11}\mr{3}_{12}^{-1}\mr{4}_{11}\mr{5}_{12}^{-1}\mr{6}_{11} + \mr{1}_{12}^{-1}\mr{3}_{10}\mr{6}_{13}^{-1}\mr{7}_{12} $$
$$+ \mr{1}_{12}^{-1}\mr{3}_{10}\mr{7}_{14}^{-1} + \mr{1}_{12}^{-1}\mr{2}_{11}\mr{3}_{12}^{-1}\mr{4}_{11}\mr{6}_{13}^{-1}\mr{7}_{12} + \mr{1}_{12}^{-1}\mr{2}_{11}\mr{4}_{13}^{-1}\mr{6}_{11} + \mr{1}_{12}^{-1}\mr{2}_{11}\mr{3}_{12}^{-1}\mr{4}_{11}\mr{7}_{14}^{-1} + \mr{1}_{12}^{-1}\mr{2}_{11}\mr{4}_{13}^{-1}\mr{5}_{12}\mr{6}_{13}^{-1}\mr{7}_{12}$$ 
$$+ \mr{1}_{12}^{-1}\mr{2}_{11}\mr{4}_{13}^{-1}\mr{5}_{12}\mr{7}_{14}^{-1} + \mr{1}_{12}^{-1}\mr{2}_{11}\mr{5}_{14}^{-1}\mr{7}_{12}\mr{8}_{13} + \mr{1}_{12}^{-1}\mr{2}_{11}\mr{7}_{12}\mr{8}_{15}^{-1} + \mr{1}_{12}^{-1}\mr{2}_{11}\mr{5}_{14}^{-1}\mr{6}_{13}\mr{7}_{14}^{-1}\mr{8}_{13} + \mr{1}_{12}^{-1}\mr{2}_{11}\mr{6}_{13}\mr{7}_{14}^{-1}\mr{8}_{15}^{-1}$$
$$+ \mr{1}_{12}^{-1}\mr{2}_{11}\mr{6}_{15}^{-1}\mr{8}_{13} + \mr{1}_{12}^{-1}\mr{2}_{11}\mr{5}_{14}\mr{6}_{15}^{-1}\mr{8}_{15}^{-1} + \mr{1}_{12}^{-1}\mr{2}_{11}\mr{4}_{15}\mr{5}_{16}^{-1} + \mr{1}_{12}^{-1}\mr{2}_{11}\mr{3}_{16}\mr{4}_{17}^{-1} + \mr{1}_{12}^{-1}\mr{2}_{11}\mr{2}_{17}\mr{3}_{18}^{-1} \Big] $$
$$+\, \Big[ \mr{2}_{13}^{-1}\mr{8}_{9} + \mr{2}_{13}^{-1}\mr{5}_{10}\mr{8}_{11}^{-1} + \mr{2}_{13}^{-1}\mr{4}_{11}\mr{5}_{12}^{-1}\mr{6}_{11} + \mr{2}_{13}^{-1}\mr{3}_{12}\mr{4}_{13}^{-1}\mr{6}_{11} + \mr{2}_{13}^{-1}\mr{4}_{11}\mr{6}_{13}^{-1}\mr{7}_{12} + \mr{2}_{13}^{-1}\mr{3}_{12}\mr{4}_{13}^{-1}\mr{5}_{12}\mr{6}_{13}^{-1}\mr{7}_{12} $$
$$+\, \mr{2}_{13}^{-1}\mr{4}_{11}\mr{7}_{14}^{-1} + \mr{3}_{14}^{-1}\mr{6}_{11} + \mr{3}_{14}^{-1}\mr{5}_{12}\mr{6}_{13}^{-1}\mr{7}_{12} + \mr{2}_{13}^{-1}\mr{3}_{12}\mr{5}_{14}^{-1}\mr{7}_{12}\mr{8}_{13} + \mr{2}_{13}^{-1}\mr{3}_{12}\mr{4}_{13}^{-1}\mr{5}_{12}\mr{7}_{14}^{-1} + \mr{3}_{14}^{-1}\mr{4}_{13}\mr{5}_{14}^{-1}\mr{7}_{12}\mr{8}_{13} $$
$$+\, \mr{3}_{14}^{-1}\mr{5}_{12}\mr{7}_{14}^{-1} + \mr{2}_{13}^{-1}\mr{3}_{12}\mr{5}_{14}^{-1}\mr{6}_{13}\mr{7}_{14}^{-1}\mr{8}_{13} + \mr{2}_{13}^{-1}\mr{3}_{12}\mr{7}_{12}\mr{8}_{15}^{-1} + \mr{4}_{15}^{-1}\mr{7}_{12}\mr{8}_{13} + \mr{3}_{14}^{-1}\mr{4}_{13}\mr{5}_{14}^{-1}\mr{6}_{13}\mr{7}_{14}^{-1}\mr{8}_{13} + \mr{3}_{14}^{-1}\mr{4}_{13}\mr{7}_{12}\mr{8}_{15}^{-1} $$
$$+\, \mr{2}_{13}^{-1}\mr{3}_{12}\mr{6}_{15}^{-1}\mr{8}_{13} + \mr{2}_{13}^{-1}\mr{3}_{12}\mr{6}_{13}\mr{7}_{14}^{-1}\mr{8}_{15}^{-1} + \mr{4}_{15}^{-1}\mr{6}_{13}\mr{7}_{14}^{-1}\mr{8}_{13} + \mr{4}_{15}^{-1}\mr{5}_{14}\mr{7}_{12}\mr{8}_{15}^{-1} + \mr{3}_{14}^{-1}\mr{4}_{13}\mr{6}_{15}^{-1}\mr{8}_{13} + \mr{3}_{14}^{-1}\mr{4}_{13}\mr{6}_{13}\mr{7}_{14}^{-1}\mr{8}_{15}^{-1} $$
$$+\, \mr{2}_{13}^{-1}\mr{3}_{12}\mr{5}_{14}\mr{6}_{15}^{-1}\mr{8}_{15}^{-1} + \mr{4}_{15}^{-1}\mr{5}_{14}\mr{6}_{15}^{-1}\mr{8}_{13} + \mr{4}_{15}^{-1}\mr{5}_{14}\mr{6}_{13}\mr{7}_{14}^{-1}\mr{8}_{15}^{-1} + \mr{5}_{16}^{-1}\mr{6}_{15}\mr{7}_{12} + \mr{3}_{14}^{-1}\mr{4}_{13}\mr{5}_{14}\mr{6}_{15}^{-1}\mr{8}_{15}^{-1} + \mr{2}_{13}^{-1}\mr{3}_{12}\mr{4}_{15}\mr{5}_{16}^{-1} $$
$$+\, \mr{2}_{13}^{-1}\mr{3}_{12}\mr{3}_{16}\mr{4}_{17}^{-1} + \mr{3}_{14}^{-1}\mr{4}_{13}\mr{4}_{15}\mr{5}_{16}^{-1} + \mr{4}_{15}^{-1}\mr{5}_{14}^{2}\mr{6}_{15}^{-1}\mr{8}_{15}^{-1} + \mr{5}_{16}^{-1}\mr{6}_{13}\mr{6}_{15}\mr{7}_{14}^{-1} + \mr{6}_{17}^{-1}\mr{7}_{12}\mr{7}_{16} + \mr{5}_{16}^{-1}\mr{8}_{13}\mr{8}_{15} $$
$$+\,\boxed{\ul{\mr{1}_{20}^{-1}\mr{1}_{10}}} + \boxed{\mr{1}_{12}^{-1}\mr{1}_{18}\mr{2}_{19}^{-1}\mr{2}_{11}} + \boxed{\mr{2}_{13}^{-1}\mr{2}_{17}\mr{3}_{18}^{-1}\mr{3}_{12} + \mr{3}_{14}^{-1}\mr{3}_{16}\mr{4}_{17}^{-1}\mr{4}_{13} + {\it 2}\cdot\mr{5}_{16}^{-1}\mr{5}_{14} + \mr{6}_{17}^{-1}\mr{6}_{13}\mr{7}_{14}^{-1}\mr{7}_{16} + \mr{7}_{18}^{-1}\mr{7}_{12} + \mr{8}_{17}^{-1}\mr{8}_{13} }$$
$$+\mr{5}_{14}\mr{8}_{15}^{-1}\mr{8}_{17}^{-1} + \mr{6}_{13}\mr{7}_{14}^{-1}\mr{7}_{18}^{-1} + \mr{5}_{14}\mr{6}_{15}^{-1}\mr{6}_{17}^{-1}\mr{7}_{16} + \mr{4}_{15}\mr{5}_{16}^{-2}\mr{6}_{15}\mr{8}_{15} + \mr{3}_{16}\mr{4}_{15}^{-1}\mr{4}_{17}^{-1}\mr{5}_{14} + \mr{2}_{17}\mr{3}_{14}^{-1}\mr{3}_{18}^{-1}\mr{4}_{13}$$
$$+\mr{2}_{17}\mr{3}_{18}^{-1}\mr{4}_{15}^{-1}\mr{5}_{14} + \mr{3}_{16}\mr{4}_{17}^{-1}\mr{5}_{16}^{-1}\mr{6}_{15}\mr{8}_{15} + \mr{5}_{14}\mr{6}_{15}^{-1}\mr{7}_{18}^{-1} + \mr{4}_{15}\mr{5}_{16}^{-1}\mr{6}_{17}^{-1}\mr{7}_{16}\mr{8}_{15} + \mr{4}_{15}\mr{5}_{16}^{-1}\mr{6}_{15}\mr{8}_{17}^{-1} + \mr{2}_{17}\mr{3}_{18}^{-1}\mr{5}_{16}^{-1}\mr{6}_{15}\mr{8}_{15}$$
$$+\mr{3}_{16}\mr{4}_{17}^{-1}\mr{6}_{17}^{-1}\mr{7}_{16}\mr{8}_{15} + \mr{3}_{16}\mr{4}_{17}^{-1}\mr{6}_{15}\mr{8}_{17}^{-1} + \mr{4}_{15}\mr{5}_{16}^{-1}\mr{7}_{18}^{-1}\mr{8}_{15} + \mr{4}_{15}\mr{6}_{17}^{-1}\mr{7}_{16}\mr{8}_{17}^{-1} + \mr{2}_{17}\mr{3}_{18}^{-1}\mr{6}_{17}^{-1}\mr{7}_{16}\mr{8}_{15} + \mr{2}_{17}\mr{3}_{18}^{-1}\mr{6}_{15}\mr{8}_{17}^{-1}$$
$$+\mr{3}_{16}\mr{4}_{17}^{-1}\mr{7}_{18}^{-1}\mr{8}_{15} + \mr{3}_{16}\mr{4}_{17}^{-1}\mr{5}_{16}\mr{6}_{17}^{-1}\mr{7}_{16}\mr{8}_{17}^{-1} + \mr{4}_{15}\mr{7}_{18}^{-1}\mr{8}_{17}^{-1} + \mr{2}_{17}\mr{3}_{18}^{-1}\mr{7}_{18}^{-1}\mr{8}_{15} + \mr{2}_{17}\mr{3}_{18}^{-1}\mr{5}_{16}\mr{6}_{17}^{-1}\mr{7}_{16}\mr{8}_{17}^{-1} + \mr{3}_{16}\mr{5}_{18}^{-1}\mr{7}_{16}$$
$$+\mr{3}_{16}\mr{4}_{17}^{-1}\mr{5}_{16}\mr{7}_{18}^{-1}\mr{8}_{17}^{-1} + \mr{2}_{17}\mr{3}_{18}^{-1}\mr{4}_{17}\mr{5}_{18}^{-1}\mr{7}_{16} + \mr{2}_{17}\mr{3}_{18}^{-1}\mr{5}_{16}\mr{7}_{18}^{-1}\mr{8}_{17}^{-1} + \mr{3}_{16}\mr{5}_{18}^{-1}\mr{6}_{17}\mr{7}_{18}^{-1} + \mr{3}_{16}\mr{6}_{19}^{-1} + \mr{2}_{17}\mr{4}_{19}^{-1}\mr{7}_{16}$$
$$+\mr{2}_{17}\mr{3}_{18}^{-1}\mr{4}_{17}\mr{5}_{18}^{-1}\mr{6}_{17}\mr{7}_{18}^{-1} + \mr{2}_{17}\mr{4}_{19}^{-1}\mr{6}_{17}\mr{7}_{18}^{-1} + \mr{2}_{17}\mr{3}_{18}^{-1}\mr{4}_{17}\mr{6}_{19}^{-1} + \mr{2}_{17}\mr{4}_{19}^{-1}\mr{5}_{18}\mr{6}_{19}^{-1} + \mr{2}_{17}\mr{5}_{20}^{-1}\mr{8}_{19} + \mr{2}_{17}\mr{8}_{21}^{-1} \Big]$$
$$+\, \Big[ \mr{1}_{18}\mr{2}_{13}^{-1}\mr{2}_{19}^{-1}\mr{3}_{12} + \mr{1}_{18}\mr{2}_{19}^{-1}\mr{3}_{14}^{-1}\mr{4}_{13} + \mr{1}_{18}\mr{2}_{19}^{-1}\mr{4}_{15}^{-1}\mr{5}_{14} + \mr{1}_{18}\mr{2}_{19}^{-1}\mr{5}_{16}^{-1}\mr{6}_{15}\mr{8}_{15} + \mr{1}_{18}\mr{2}_{19}^{-1}\mr{6}_{15}\mr{8}_{17}^{-1}$$
$$+ \mr{1}_{18}\mr{2}_{19}^{-1}\mr{6}_{17}^{-1}\mr{7}_{16}\mr{8}_{15} + \mr{1}_{18}\mr{2}_{19}^{-1}\mr{5}_{16}\mr{6}_{17}^{-1}\mr{7}_{16}\mr{8}_{17}^{-1} +\mr{1}_{18}\mr{2}_{19}^{-1}\mr{7}_{18}^{-1}\mr{8}_{15} + \mr{1}_{18}\mr{2}_{19}^{-1}\mr{5}_{16}\mr{7}_{18}^{-1}\mr{8}_{17}^{-1} + \mr{1}_{18}\mr{2}_{19}^{-1}\mr{4}_{17}\mr{5}_{18}^{-1}\mr{7}_{16}$$
$$+ \mr{1}_{18}\mr{2}_{19}^{-1}\mr{4}_{17}\mr{5}_{18}^{-1}\mr{6}_{17}\mr{7}_{18}^{-1} + \mr{1}_{18}\mr{2}_{19}^{-1}\mr{3}_{18}\mr{4}_{19}^{-1}\mr{7}_{16} + \mr{1}_{18}\mr{2}_{19}^{-1}\mr{4}_{17}\mr{6}_{19}^{-1} + \mr{1}_{18}\mr{2}_{19}^{-1}\mr{3}_{18}\mr{4}_{19}^{-1}\mr{6}_{17}\mr{7}_{18}^{-1} +\mr{1}_{18}\mr{3}_{20}^{-1}\mr{7}_{16}$$
$$+ \mr{1}_{18}\mr{3}_{20}^{-1}\mr{6}_{17}\mr{7}_{18}^{-1} + \mr{1}_{18}\mr{2}_{19}^{-1}\mr{3}_{18}\mr{4}_{19}^{-1}\mr{5}_{18}\mr{6}_{19}^{-1} + \mr{1}_{18}\mr{3}_{20}^{-1}\mr{5}_{18}\mr{6}_{19}^{-1} + \mr{1}_{18}\mr{2}_{19}^{-1}\mr{3}_{18}\mr{5}_{20}^{-1}\mr{8}_{19} + \mr{1}_{18}\mr{2}_{19}^{-1}\mr{3}_{18}\mr{8}_{21}^{-1} $$
$$+ \mr{1}_{18}\mr{3}_{20}^{-1}\mr{4}_{19}\mr{5}_{20}^{-1}\mr{8}_{19} + \mr{1}_{18}\mr{3}_{20}^{-1}\mr{4}_{19}\mr{8}_{21}^{-1} + \mr{1}_{18}\mr{4}_{21}^{-1}\mr{8}_{19} + \mr{1}_{18}\mr{4}_{21}^{-1}\mr{5}_{20}\mr{8}_{21}^{-1} + \mr{1}_{18}\mr{5}_{22}^{-1}\mr{6}_{21} + \mr{1}_{18}\mr{6}_{23}^{-1}\mr{7}_{22} + \mr{1}_{18}\mr{7}_{24}^{-1} \Big]\bigg)$$
$$+ \bigg( \mr{1}_{12}^{-1}\mr{1}_{20}^{-1}\mr{2}_{11} + \Big[ \mr{1}_{20}^{-1}\mr{2}_{13}^{-1}\mr{3}_{12} + \mr{1}_{20}^{-1}\mr{3}_{14}^{-1}\mr{4}_{13} +\mr{1}_{20}^{-1}\mr{4}_{15}^{-1}\mr{5}_{14} + \mr{1}_{20}^{-1}\mr{5}_{16}^{-1}\mr{6}_{15}\mr{8}_{15} + \mr{1}_{20}^{-1}\mr{6}_{15}\mr{8}_{17}^{-1} + \mr{1}_{20}^{-1}\mr{6}_{17}^{-1}\mr{7}_{16}\mr{8}_{15} $$
$$+ \mr{1}_{20}^{-1}\mr{5}_{16}\mr{6}_{17}^{-1}\mr{7}_{16}\mr{8}_{17}^{-1} + \mr{1}_{20}^{-1}\mr{7}_{18}^{-1}\mr{8}_{15} +\mr{1}_{20}^{-1}\mr{5}_{16}\mr{7}_{18}^{-1}\mr{8}_{17}^{-1} + \mr{1}_{20}^{-1}\mr{4}_{17}\mr{5}_{18}^{-1}\mr{7}_{16} + \mr{1}_{20}^{-1}\mr{4}_{17}\mr{5}_{18}^{-1}\mr{6}_{17}\mr{7}_{18}^{-1} $$
$$+ \mr{1}_{20}^{-1}\mr{3}_{18}\mr{4}_{19}^{-1}\mr{7}_{16} + \mr{1}_{20}^{-1}\mr{4}_{17}\mr{6}_{19}^{-1} + \mr{1}_{20}^{-1}\mr{3}_{18}\mr{4}_{19}^{-1}\mr{6}_{17}\mr{7}_{18}^{-1} +\mr{1}_{20}^{-1}\mr{2}_{19}\mr{3}_{20}^{-1}\mr{7}_{16} + \mr{1}_{20}^{-1}\mr{2}_{19}\mr{3}_{20}^{-1}\mr{6}_{17}\mr{7}_{18}^{-1} $$
$$+ \mr{1}_{20}^{-1}\mr{3}_{18}\mr{4}_{19}^{-1}\mr{5}_{18}\mr{6}_{19}^{-1} + \mr{1}_{20}^{-1}\mr{2}_{19}\mr{3}_{20}^{-1}\mr{5}_{18}\mr{6}_{19}^{-1} + \mr{1}_{20}^{-1}\mr{3}_{18}\mr{5}_{20}^{-1}\mr{8}_{19} + \mr{1}_{20}^{-1}\mr{3}_{18}\mr{8}_{21}^{-1}+\mr{1}_{20}^{-1}\mr{2}_{19}\mr{3}_{20}^{-1}\mr{4}_{19}\mr{5}_{20}^{-1}\mr{8}_{19} $$
$$+ \mr{1}_{20}^{-1}\mr{2}_{19}\mr{3}_{20}^{-1}\mr{4}_{19}\mr{8}_{21}^{-1} + \mr{1}_{20}^{-1}\mr{2}_{19}\mr{4}_{21}^{-1}\mr{8}_{19} + \mr{1}_{20}^{-1}\mr{2}_{19}\mr{4}_{21}^{-1}\mr{5}_{20}\mr{8}_{21}^{-1} + \mr{1}_{20}^{-1}\mr{2}_{19}\mr{5}_{22}^{-1}\mr{6}_{21} + \mr{1}_{20}^{-1}\mr{2}_{19}\mr{6}_{23}^{-1}\mr{7}_{22}+\mr{1}_{20}^{-1}\mr{2}_{19}\mr{7}_{24}^{-1} \Big]$$
$$+ \Big[ \mr{2}_{21}^{-1}\mr{7}_{16} + \mr{2}_{21}^{-1}\mr{6}_{17}\mr{7}_{18}^{-1} + \mr{2}_{21}^{-1}\mr{5}_{18}\mr{6}_{19}^{-1} + \mr{2}_{21}^{-1}\mr{4}_{19}\mr{5}_{20}^{-1}\mr{8}_{19} + \mr{2}_{21}^{-1}\mr{4}_{19}\mr{8}_{21}^{-1} +\mr{2}_{21}^{-1}\mr{3}_{20}\mr{4}_{21}^{-1}\mr{8}_{19} $$
$$+ \mr{2}_{21}^{-1}\mr{3}_{20}\mr{4}_{21}^{-1}\mr{5}_{20}\mr{8}_{21}^{-1} + \mr{3}_{22}^{-1}\mr{8}_{19} + \mr{3}_{22}^{-1}\mr{5}_{20}\mr{8}_{21}^{-1} + \mr{2}_{21}^{-1}\mr{3}_{20}\mr{5}_{22}^{-1}\mr{6}_{21} + \mr{3}_{22}^{-1}\mr{4}_{21}\mr{5}_{22}^{-1}\mr{6}_{21} +\mr{2}_{21}^{-1}\mr{3}_{20}\mr{6}_{23}^{-1}\mr{7}_{22} + \mr{2}_{21}^{-1}\mr{3}_{20}\mr{7}_{24}^{-1} $$
$$+ \mr{3}_{22}^{-1}\mr{4}_{21}\mr{6}_{23}^{-1}\mr{7}_{22} + \mr{4}_{23}^{-1}\mr{6}_{21} + \mr{3}_{22}^{-1}\mr{4}_{21}\mr{7}_{24}^{-1} + \mr{4}_{23}^{-1}\mr{5}_{22}\mr{6}_{23}^{-1}\mr{7}_{22} +\mr{4}_{23}^{-1}\mr{5}_{22}\mr{7}_{24}^{-1} + \mr{5}_{24}^{-1}\mr{7}_{22}\mr{8}_{23} + \mr{7}_{22}\mr{8}_{25}^{-1} $$
$$+ \mr{5}_{24}^{-1}\mr{6}_{23}\mr{7}_{24}^{-1}\mr{8}_{23} + \mr{6}_{23}\mr{7}_{24}^{-1}\mr{8}_{25}^{-1} + \mr{6}_{25}^{-1}\mr{8}_{23} +\mr{5}_{24}\mr{6}_{25}^{-1}\mr{8}_{25}^{-1} + \mr{4}_{25}\mr{5}_{26}^{-1} + \mr{3}_{26}\mr{4}_{27}^{-1} + \mr{2}_{27}\mr{3}_{28}^{-1} \Big] + \mr{1}_{28}\mr{2}_{29}^{-1} \bigg) + \ul{\mr{1}_{30}^{-1}}\ .$$

\medskip

Here we group the monomials in the parenthesis and square brackets according to the restriction of $U_q(\hat{\text{E}}_8)$-module $\tl L_{1}$ to $U_q(\hat{\text{E}}_7)$ and $U_q(\hat{\text{E}}_6)$ subalgebras respectively. 
On the level of $q$-characters, the restriction to $U_q(\hat{\text{E}}_7)$ subalgebra amounts to  $\mr{1}_a\mapsto 1$ and $\mr{i}_a\mapsto\mr{(i-1)}_a$, $2\le i\le 8$. Then the restriction of $\chi_q^{\text{E}_8}(\mr{1}_0)$ is 
$$1+\chi_q^{\text{E}_7}(\mr{1}_1)+\chi_q^{\text{E}_7}(\mr{6}_6)+1+\chi_q^{\text{E}_7}(\mr{1}_{11}) +1\ .$$ 
The restriction to $U_q(\hat{\text{E}}_6)$ subalgebra amounts to $\mr{1}_a\mapsto 1$, $\mr{2}_a\mapsto 1$, $\mr{i}_a\mapsto \mr{(i-2)}_a$, $3\le i\le 8$. Then the restriction of $\chi_q^{\text{E}_8}(\mr{1}_0)$ is  
$$1+\Big(1+\chi_q^{\text{E}_6}(\mr{1}_2)+\chi_q^{\text{E}_6}(\mr{5}_6)+1\Big)+\Big(\chi_q^{\text{E}_6}(\mr{5}_6)+\chi_q^{\text{E}_6}(\mr{6}_9)+1+\chi_q^{\text{E}_6}(\mr{1}_{12})\Big)+1+\Big(1+\chi_q^{\text{E}_6}(\mr{1}_{12})+\chi_q^{\text{E}_6}(\mr{5}_{16})+1\Big)+1\ .$$

\medskip

The structure of the representation around the weight $0$ part is shown in Figure \ref{fig:E8 adjoint}. 
\begin{figure}[h!]
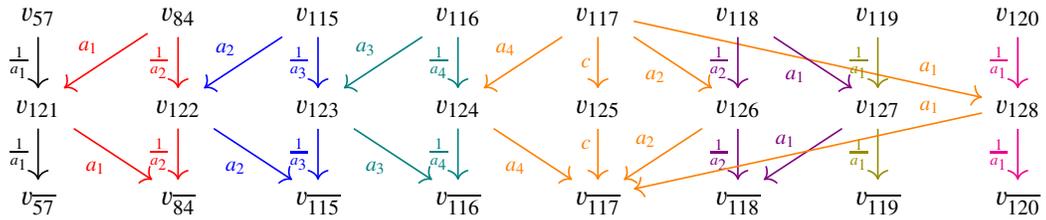

\dia{
v_{57}\ar[d,"{\frac{1}{a_1}}"'] \& v_{84}\ar[d,red,"{\frac{1}{a_2}}"']\ar[ld,red,"a_1"'] \& v_{115}\ar[ld,blue,"a_2"']\ar[d,blue,"{\frac{1}{a_3}}"'] \& v_{116}\ar[d,teal,"{\frac{1}{a_4}}"']\ar[ld,teal,"a_3"'] \& v_{117}\ar[d,orange,"c"']\ar[rrrd,orange,"a_1",pos=0.8]\ar[rd,orange,"a_2"']\ar[ld,orange,"a_4"'] \& v_{118}\ar[d,violet,"{\frac{1}{a_2}}"']\ar[rd,violet,"a_1"'] \& v_{119}\ar[d,olive,"{\frac{1}{a_1}}"'] \& v_{120}\ar[d,magenta,"{\frac{1}{a_1}}"'] \\ 
v_{121}\ar[d,"{\frac{1}{a_1}}"']\ar[rd,red,"a_1"'] \& v_{122}\ar[d,red,"{\frac{1}{a_2}}"']\ar[rd,blue,"a_2"'] \& v_{123}\ar[d,blue,"{\frac{1}{a_3}}"']\ar[rd,teal,"a_3"'] \& v_{124}\ar[d,teal,"{\frac{1}{a_4}}"']\ar[rd,orange,"a_4"'] \& v_{125}\ar[d,orange,"c"', pos=0.3] \& v_{126}\ar[d,violet,"{\frac{1}{a_2}}"']\ar[ld,orange,"a_2"'] \& v_{127}\ar[d,olive,"{\frac{1}{a_1}}"']\ar[ld,violet,"a_1"'] \& v_{128} \ar[d,magenta,"{\frac{1}{a_1}}"']\ar[llld,orange,"a_1"',pos=0.1] \\
v_{\ol{57}} \& v_{\ol{84}} \& v_{\ol{115}} \& v_{\ol{116}} \& v_{\ol{117}} \& v_{\ol{118}} \& v_{\ol{119}} \& v_{\ol{120}}
}
    \caption{The first fundamental/adjoint module for E$_8$ (shown around weight zero vectors $v_i,121\le i\le 128$).}
    \label{fig:E8 adjoint}
\end{figure}

Here $\ol{i}=249-i$, $a_i=\sqrt{\frac{[i]}{[i+1]}}$, $1\le i\le 4$, $c=\sqrt{\frac{[2]_{8}+[2]_{6}-[3]}{[2]\,[3]\,[5]}}$, and the colors of arrows correspond to simple roots as follows: 
\medskip
\dia{ \,\ar[r,"f_1"'] \& \,\,\,\,\,\,\,\,\ar[r,red,"f_2"'] \& \,\,\,\,\,\,\,\, \ar[r,blue,"f_3"'] \&  \,\,\,\,\,\,\,\, \ar[r,teal,"f_4"'] \&  \,\,\,\,\,\,\,\, \ar[r,orange,"f_5"'] \&  \,\,\,\,\,\,\,\, \ar[r,violet,"f_6"'] \&  \,\,\,\,\,\,\,\, \ar[r,olive,"f_7"'] \&  \,\,\,\,\,\,\,\, \ar[r,magenta,"f_8"'] \& \, } 

We note that $[2]_8+[2]_6-[3]=\kappa_{60}(q)$ is the symmetric form of 60-th cyclotomic polynomial.

To complete the  diagram one has to add vectors for all other 224 monomials of the $q$-character and connect by arrows of color $i$ the pairs of monomials which differ by an $i$-th affine root. All these arrows have coefficient one.
Then the total diagram describes the action of $f_i$, $i\in \mr{I}$.
For example, $f_3 v_{115}= a_2v_{122}+\frac{1}{a_3}v_{123}$,  $f_5 v_{125}= c v_{\ol{117}}$, etc. The action of $e_i$'s is obtained by reversing all the arrows and keeping the same coefficient on each arrow.


\bigskip

{\bf Acknowledgments.\ }
The authors are partially supported by Simons Foundation grant number \#709444.

\bigskip

\input{mybib}
\end{document}

%% file: mypreamble.tex
\usepackage{amssymb, amsmath, amsthm, amsbsy, mathrsfs, mathtools, mathdots, bbm, stackrel, cancel, ulem, graphicx, relsize, dynkin-diagrams} 

\usepackage[toc, page]{appendix}

\usepackage[all, cmtip]{xy}
\usepackage{tikz, tikz-cd}

\usepackage[mathscr]{euscript}
\usepackage[scr=boondoxo]{mathalfa}

\usepackage{setspace}
\usepackage[shortlabels]{enumitem}
\setlist[enumerate]{itemsep=1pt, topsep=2pt}
\setlist[itemize]{itemsep=1pt, topsep=2pt}

\usepackage{color, hyperref}
\hypersetup{colorlinks=true, linktoc=all, linkcolor=blue, citecolor=blue}

\DeclareMathOperator{\op}{op}

\DeclareMathOperator{\rep}{Rep}

\DeclareMathOperator{\Tr}{Tr}

\DeclareMathOperator{\en}{End}

\DeclareMathOperator{\id}{Id}
\makeatletter
\newcommand{\oset}[3][0ex]{%
  \mathrel{\mathop{#3}\limits^{
    \vbox to#1{\kern-2\ex@
    \hbox{$\scriptstyle#2$}\vss}}}}
\makeatother

\newcommand{\tl}[1]{\tilde{#1}}

\newcommand{\ve}{\varepsilon}

\newcommand{\A}{\alpha}
\newcommand{\B}{\beta}
\newcommand{\G}{\gamma}
\newcommand{\D}{\delta}
\newcommand{\la}{\lambda}
\newcommand{\om}{\omega}

\newcommand{\oto}{\leftrightarrow}

\newcommand{\xto}[1]{\xrightarrow{#1}}

\newcommand{\se}{\subseteq}

\newcommand{\C}{\mathbb{C}}

\newcommand{\Z}{\mathbb{Z}}

\newcommand{\tb}[1]{\textbf{#1}}

\newcommand{\ti}[1]{\textit{#1}}
\newcommand{\ul}[1]{\underline{#1}}
\newcommand{\ol}[1]{\overline{#1}}

\newcommand{\mr}[1]{\mathrm{#1}}
\newcommand{\cl}[1]{\mathcal{#1}}
\newcommand{\fk}[1]{\mathfrak{#1}}
\newcommand{\sr}[1]{\mathscr{#1}}

\newcommand{\p}[1]{\begin{proof}#1\end{proof}}

\newcommand{\dia}[1]{\vspace{-2mm}\begin{center}\begin{tikzcd}[ampersand replacement = \&]#1\end{tikzcd}\end{center}\vspace{-1mm}}
\newcommand{\ddia}[2]{\vspace{-2mm}\begin{center}\begin{tikzcd}[ampersand replacement = \&, #1]#2\end{tikzcd}\end{center}\vspace{-1mm}}
\newcommand{\eq}[1]{\begin{equation}\begin{aligned}#1\end{aligned}\end{equation}}
\numberwithin{equation}{section}

\newcommand{\comment}[1]{}

\newcommand{\bee}[1]{$$\begin{aligned}#1\end{aligned}$$}